\documentclass{amsart}

\newtheorem{theorem}{Theorem}[section]
\newtheorem{lemma}[theorem]{Lemma}

\theoremstyle{definition}
\newtheorem{definition}[theorem]{Definition}
\newtheorem{example}[theorem]{Example}

\theoremstyle{remark}
\newtheorem{remark}[theorem]{Remark}
\theoremstyle{plain}
\newtheorem{theo+}{Theorem}
\numberwithin{theo+}{section}
\newtheorem{prop+}[theo+]{Proposition}
\newtheorem{coro+}[theo+]{Corollary}
\newtheorem{lemm+}[theo+]{Lemma}
\newtheorem{conjecture}[theo+]{Conjecture}

\theoremstyle{definition}
\newtheorem{defi+}[theo+]{Definition}

\usepackage{amssymb}
\usepackage{graphicx}
\usepackage{float}
\usepackage{subfigure}
\usepackage[]{caption2} 
\renewcommand{\thesubfigure}{(\roman{subfigure})}
\makeatletter \renewcommand{\@thesubfigure}{\thesubfigure \space}
\renewcommand{\p@subfigure}{} \makeatother

\numberwithin{equation}{section}

\newcommand{\abs}[1]{\lvert#1\rvert}

\begin{document}

\title{Principle of generating critical points of rational functions and Gauss-Fourier conjecture for second-kind zeros}


\author{Lande Ma}
\address{}
\curraddr{School of Mathematical Sciences, Tongji University, Shanghai, 200092, China}
\email{dzy200408@sina.cn}
\thanks{}

\author{Zhaokun Ma}
\address{}
\curraddr{Yanzhou College, Shandong Radio and TV University, YanZhou, 272100, China}
\email{dzy200408@126.com}
\thanks{}


\date{October 17th, 2025.}

\dedicatory{}


\begin{abstract}
The precise origin of critical points in rational functions remains a longstanding open problem. This paper elucidates the mechanism underlying critical point generation in rational functions. For over two centuries, a conjecture attributed to Gauss and Fourier has remained unresolved: every real zero of $(p^{'}(s)/p(s))^{'}$ is associated with a pair of non-real zeros of the real polynomial $p(s)$.  The real zeros of $(p^{'}(s)/p(s))^{'}$ that are multiple real zeros of $p^{'}(s)$ are termed second-kind zeros of $(p^{'}(s)/p(s))^{'}$. By developing the root locus method, we establish a necessary and sufficient condition for the emergence of second-kind zeros: a second-kind zero $M$ arises if and only if at least two root loci of $2q\pi$ or $2q\pi+\pi$ degree, emitted from at least one pair of non-real conjugate zeros of $p(s)$ (in the extended complex plane, excluding the real axis), directly intersect another root locus at $M$ on the real axis. Our results resolve one conjecture of the Gauss-Fourier conjectures and reveal the fundamental mechanism by which critical points of rational functions are generated.
\end{abstract}
\maketitle

\section*{Introduction}
In 2004, Borcea and Shapiro proposed a conjecture \cite{Shapiro1}, motivated by their efforts to prove the Hawaii Conjecture put forward by Craven, Csordas, and Smith \cite{Csordas1}. The Hawaii Conjecture asserts that for a real polynomial $p(s)$ of degree $n \ge 2$, the number of real zeros of $(p'(s)/p(s))'$ does not exceed the number of non-real zeros of $p(s)$. Tracing its roots to foundational work by Gauss and Fourier—with historical details documented by Edwards and Hinkkanen \cite{Stephanie1}—this conjecture has long been a focal point in the mathematical community.

In the correspondence \textit{G\"{o}ttingische gelehrte Anzeigen} (G.G.A.) dated February 25, 1833 \cite{Gauss7,Gauss8}, Gauss described observations and a conjecture by Fourier. Gauss explained that Fourier believed for each real zero of $(p^{'}(s)/p(s))^{'}$ (termed "critical points" by Fourier), there exists an associated pair of non-real zeros of the polynomial $p(s)$.

The Hawaii Conjecture represents an effort by Craven, Csordas, and Smith to quantify the ideas explored by Gauss and Fourier, though it does not explicitly reference the above-mentioned association between zeros. This conjecture was resolved in 2011 by M.~Tyaglov \cite{Tyaglov1}. Subsequently, B.~Shapiro proposed three new conjectures related to the Hawaii Conjecture (Conjectures 11, 12, and 13 in \cite{Shapiro0}). In our prior work \cite{Ma} (January 2024), we established the necessary and sufficient conditions for Shapiro's Conjecture 12 to hold.

Mathematical historiography shows that the Gauss-Fourier conjecture has shaped a distinct research trajectory over the past two centuries. While derivative conjectures (e.g., the Hawaii Conjecture) have been resolved, the original conjectures by Gauss and Fourier remain unsolved. Beyond the initial examples provided by these pioneers, no further general positive results have been proven.

The expression $(\frac{p^{'}(s)}{p(s)})^{'}=\frac{p^{''}(s)p(s)-(p^{'}(s))^2}{p^2(s)}$ implies that multiple zeros of $p^{'}(s)$ are zeros of $p^{''}(s)p(s)-(p^{'}(s))^2$ and thus critical points of $\frac{p^{'}(s)}{p(s)}$. These critical points of $\frac{p^{'}(s)}{p(s)}$ are inherited from the original multiple zeros of $p^{'}(s)$, whereas multiple zeros of $p(s)$ are poles (not zeros) of $\frac{p^{'}(s)}{p(s)}$. Note that zeros of polynomials are categorized into simple zeros and multiple zeros, with multiple zeros constituting a distinct class rather than a special case of zeros.

Since the Gauss-Fourier conjecture concerns the real zeros of $(p'(s)/p(s))'$, we classify these zeros as follows:

\begin{definition}
A real zero of $(p^{'}(s)/p(s))^{'}$ is called a \emph{second kind of zero of $(p^{'}(s)/p(s))^{'}$} (or \emph{second kind zeros}) if it is a real multiple zero of $p^{'}(s)$.
\end{definition}

\begin{definition}
In addition to the second kind of zero, the remaining real zeros of $(p^{'}(s)/p(s))^{'}$ are referred to as the \emph{first kind of zeros of $(p^{'}(s)/p(s))^{'}$} (or \emph{first kind zeros}).
\end{definition}

The zeros of the first kind of $(p^{'}(s)/p(s))^{'}$ represent a class of real zeros of $(p^{'}(s)/p(s))^{'}$. The zeros of the second kind of $(p^{'}(s)/p(s))^{'}$ are a class of real zeros of $(p^{'}(s)/p(s))^{'}$. Both first and second kind zeros represent classes of real zeros of $(p^{'}(s)/p(s))^{'}$.  Both kinds of zeros are equally important in the Gauss-Fourier conjecture; the conjecture related to second kind zeros is not a specialized issue.

Gauss described Fourier's conjecture that for each real zero of $(p^{'}(s)/p(s))^{'}$ (or, as Fourier calls them, "critical points"), there existed an associated pair of non-real zeros of  $p(s)$. Edwards and Hinkkanen note in \cite{Stephanie1}: "Gauss and Fourier found it difficult to make their ideas and conjectures articulate precisely since they attempted to find a specific association. However, it is evident from the correspondence that they were trying to understand exactly how these non-real zeros and real zero of $(p^{'}(s)/p(s))^{'}$ are connected. Both Gauss and Fourier seemed to express frustrations about not being able to precisely state and prove these conjectures."

Gauss's description may represent one facet of a broader framework encompassed by the Gauss-Fourier conjecture—one that includes multiple interrelated problems. An optimal resolution of this framework would clarify how real Fourier critical points are generated, with the key insight that these points cannot emerge without considering the non-real zeros of $p(s)$. Proving this mechanism would address a major challenge within the conjecture.

The Gauss-Fourier conjecture concerns real critical points of the rational function $p^{'}(s)/p(s)$, centering on understanding their generation. In this paper, we prove Theorem 0.5, providing necessary and sufficient conditions for generating critical points of general rational functions. This theorem is essential for formulating and solving the main problem of the Gauss-Fourier conjecture.

Understanding critical point generation is profoundly important and arguably surpasses the conjecture itself. Gauss and Fourier sought principles governing critical point generation for general functions, leading to this conjecture.

Here, we investigate the emergence of second kind zeros of $(p^{'}(s)/p(s))^{'}$. Following Gauss and Fourier's approach, we propose Conjectures 0.1 and 0.2 and prove Theorems 0.3 and 0.4, addressing core problems in the Gauss-Fourier conjecture.

We also establish results in root locus theory. In the extended complex plane $\mathbb{C} \cup \{\infty\}$, the set of points $s$ where the argument of a rational function $RF(s)$ is equal to $\alpha$ degrees forms a curve known as the root loci of $\alpha$ degree. The argument of $RF(s)$ at each point on a specific curve has the same phase angle. The argument of $RF(s)$ at a point on a certain root locus is $\alpha$ degree. The degrees of distinct curves are typically distinct. (The phase angle is the argument). $\alpha$ is the degree of the root locus. Notably, adjacent root loci generally have distinct degrees.

When root loci have distinct unit complex values, they cannot intersect. When unit complex values coincide, intersection is possible. Such intersection points are `"breakaway points" in root locus theory, representing critical points of complex rational functions. Thus, critical points emerge from intersections of root loci sharing identical unit complex values.

Notation 0: Let $\mathbb{C}\cup\{\infty\}$ denote the extended complex plane, $\mathbb{C}^{*}=\{s=x+iy: y\neq 0\}$ denote the extended complex plane excluding the real axis, $RF$ denote any complex rational function, $M$ denote a second kind of zero of $(p^{'}(s)/p(s))^{'}$.

The real polynomial $p(s)$ forms the root locus equation:
$K\frac{1}{p(s)}=a+ib$.

\begin{conjecture}
A second-kind zero of $(p^{'}(s)/p(s))^{'}$ can only be generated if root loci emitted from non-real zeros of real polynomial $p$ intersect on the real axis.
\end{conjecture}

\begin{conjecture}
Each second kind of zero $M$ of $(p^{'}(s)/p(s))^{'}$ is generated, if and only if at least two root loci of $2q\pi$ or $2q\pi+\pi$ degree, emitted from a pair of non-real conjugate zeros of the real polynomial $p(s)$, directly intersect at least one other root locus at $M$ on the real axis in $\mathbb{C}^{*}$.
\end{conjecture}

In the subsequent two centuries, no results have linked the zeros of the second kind of $(p^{'}(s)/p(s))^{'}$ to non-real zeros of real polynomials. Relevant studies appear in \cite{Totik,Ge,Mir,Angst}.

In this paper, we have proved Theorem 0.3 and Theorem 0.4.

\begin{theorem}
A second kind of zero of $(p^{'}(s)/p(s))^{'}$ can only be generated if root loci emitted from non-real zeros of the real polynomial $p(s)$ intersect on the real axis.
\end{theorem}

By Theorem 3.22 in this paper, root loci emitted from real zeros of $p(s)$ enter $\mathbb{C}^{*}$ and cannot re-intersect the real axis. If root loci emitted from non-real zeros do not intersect other root loci on the real axis, second kind zeros cannot be generated. Thus, intersections from non-real zeros are necessary.

\begin{theorem}
Each second kind of zero $M$ of $(p^{'}(s)/p(s))^{'}$ is generated if and only if at least two root loci of $2q\pi$ or $2q\pi+\pi$ degree, emitted from a pair of conjugate non-real zeros of $p(s)$, directly intersect at least one other root locus at $M$ on the real axis in $\mathbb{C}^{*}$,
\end{theorem}

Here, root loci of $2q\pi$ or $2q\pi+\pi$ degree emitted from at least one pair of conjugate non-real zeros of the real polynomial $p(s)$ intersect at $M$ exactly once. At least one other root locus (but no more than two, excluding those from the pair) must intersect at $M$.

Monotonicity properties of real functions (e.g., Rolle's Theorem) identify intervals containing critical points for real rational functions but do not fully explain their generation. We show that critical points arise when root loci intersect, clarifying their generation mechanism.

Since calculus began, understanding how and why real critical points emerge has been a central mathematical concern. This question is fundamental, not superfluous. After centuries, we now uncover the principles governing critical point generation for rational functions.

\begin{theorem}
All critical points of rational function $RF$ (excluding multiple zeros/poles) are breakaway points of its root loci. Conversely, all breakaway points of the root loci of $RF$ are its critical points.
\end{theorem}
Using the root locus method, we establish its origination desired by Fourier and Gauss. Consequently, we have achieved a complete resolution of one conjecture of the Gauss-Fourier conjectures for second-kind zeros of $(p^{'}(s)/p(s))^{'}$.

In the context of associated relationships, the first type is the generating relationship. Theorem 0.4 proves that the emergence of each second-kind zero of $(p^{'}(s)/p(s))^{'}$ necessarily involves a pair of non-real zeros. This result also demonstrates that each second-kind zero of $(p^{'}(s)/p(s))^{'}$ is associated with a pair of non-real zeros of the real polynomial $p(s)$.

When $p(s)=s^4+1$, it has four conjugate complex zeros $s_{1,2}=\frac{\sqrt{2}}{2}\pm i\frac{\sqrt{2}}{2}$, $s_{3,4}=-\frac{\sqrt{2}}{2}\pm i\frac{\sqrt{2}}{2}$.
$\frac{1}{p(s)}$ can form the root locus equation: $K\frac{1}{p(s)}=\pm 1$. $p^{'}(s)=4s^3$. Therefore, there must exist four root loci emitted from four conjugate complex zeros of $p(s)$, four root loci intersect at the breakaway point $b_{1}=0$. The breakaway point $b_{1}=0$ is a second kind of zero of $(p^{'}(s)/p(s))^{'}$.

In this example, we obtain one second kind of zero of $(p^{'}(s)/p(s))^{'}$. The second kind of zero of $(p^{'}(s)/p(s))^{'}$ is generated by the intersection of the root loci of non-real zeros of $p(s)$. However, The generation of second kind of zero of $(p^{'}(s)/p(s))^{'}$ has no any relationship with real zeros of $p(s)$. This example also proves that each second-kind zero of $(p^{'}(s)/p(s))^{'}$ is associated with a pair of non-real zeros of  $p(s)$, but, each second-kind zero of $(p^{'}(s)/p(s))^{'}$ may be not associated with real zeros of  $p(s)$.

Gauss has provided an example, that he thought under the example, Gauss-Fourier conjecture is not true. However, this example belongs to the case of first-kind zero of $(p^{'}(s)/p(s))^{'}$, not the second-kind zero of $(p^{'}(s)/p(s))^{'}$. Therefore, his example and ideas do not affect the results in this paper.

\section{Root Locus Theory for Rational Functions}
All results in this section concern root locus theory. Lemmas 1.2, 1.3, 1.4, 1.13, 1.16, and 1.24, along with Theorems  1.14,  1.27, 1.31 and 1.32 are adapted from textbooks on automatic control theory \cite{Dorf1,Y1}. The remaining 25 results are original discoveries and proofs by the authors. Furthermore, even for the 10 results from automatic control theory, some proofs lack mathematical rigor. In this paper, we provide rigorous proofs and extend root locus theory.

Notation 1:
A non-constant rational function with constant coefficients denoted by $RF(s)$, is given by $RF(s)=\frac{\prod^{m}_{l=1} (s-z_l)^{\gamma_l}}{ \prod^{n}_{j=1}  (s-p_j)^{\beta_j}}$, where $z_l, p_j\in \mathbb{C}$, $l=1,2,\cdots,m$, $j=1,2,\cdots,n$. $\gamma_l, \beta_j\in \mathbb{N}$. $m, n$ are non-negative integers. $m$ and $n$ may be zero. $\gamma_l$ and $\beta_j$ are positive integers. Specifically, $RF$ may consist solely of a numerator that contains variables, and $RF$ may have only a denominator that contains variables. $RF$ cannot be constant.

Define the gain as $K=|\frac{1}{RF}|$, equivalently, $K=|\frac{\prod^{n}_{j=1}  (s-p_j)^{\beta_j}}{\prod^{m}_{l=1} (s-z_l)^{\gamma_l}}|$. Let $a+ib$ be a unit complex number ($a^2+b^2=1$). The gain function is: $K(s)=\frac{\prod^{n}_{j=1}  (s-p_j)^{\beta_j}}{\prod^{m}_{l=1} (s-z_l)^{\gamma_l}}$.

Let $\Delta\subset\mathbb{C}\cup\{\infty\}$ denote the set of points $s=\sigma+it\in\mathbb{C}\cup\{\infty\}$ that are neither zeros nor poles of (1.1). $s=\sigma+it\in\Delta$ is an arbitrary point in $\Delta$. The argument (or phase angle) of $RF$ is denoted $\varphi(\sigma,t)$.

The gain $K$ is the modulus of the reciprocal of $RF(s)$. Multiplying $K$ by $RF(s)$ results in a unit complex number.

\begin{definition}
In the extended complex plane $\mathbb{C}\cup\{\infty\}$, the equation (1.1) is called the \emph{root locus equation}.
\begin{equation}
K\frac{\prod^{m}_{l=1} (s-z_l)^{\gamma_l}}{\prod^{n}_{j=1} (s-p_j)^{\beta_j}}=a+ib.
\end{equation}
$a+ib$ is a fixed value.
\end{definition}

\begin{lemma}
After removing coincident poles and zeros in (1.1):

1. $n$ finite poles $p_j$ of $RF$ are also $n$ finite poles of (1.1) with $K=0$.

2. $m$ finite zeros $z_l$ of $RF$ are also $m$ finite zeros of (1.1) with $K=+\infty$.
\end{lemma}
\begin{proof}
1. Rewriting (1.1) as:
$$K\prod^{m}_{l=1} (s-z_l)^{\gamma_l}=(a+ib)\prod^{n}_{j=1}(s-p_j)^{\beta_j}.$$ $n$ finite poles of $RF$ correspond to zeros of $\prod^{n}_{j=1}  (s-p_j)^{\beta_j}$. If $K=0$, $n$ finite poles of $RF$ are $n$ roots of this equation. Conversely, $n$ finite roots of this equation when $K=0$ are $n$ finite poles of $RF$. Thus, at $n$ finite poles of (1.1), which are the $n$ finite poles of $RF$, $K=0$

2. Similarly, equation (1.1) can be expressed as
$$\prod^{m}_{l=1} (s-z_l)^{\gamma_l}=\frac{(a+ib)}{K}\prod^{n}_{j=1}  (s-p_j)^{\beta_j}.$$
If $K=+\infty$, $m$ finite zeros of $RF$ are $m$ roots of the equation. Conversely, $m$ roots of the equation when $K=+\infty$ are $m$ finite zeros of $RF$. Therefore, at $m$ finite zeros of (1.1), which are the $m$ finite zeros of $RF$, $K=+\infty$.
\end{proof}

$K$ is the modulus of the reciprocal of $RF$. Lemma 1.3 holds trivially.

\begin{lemma}
For any infinite point $s_{\infty}$ of $\mathbb{C}\cup\{\infty\}$,
(1). If $s_{\infty}$ satisfies

$K=\lim\limits_{s\to \infty}\frac{\prod^{n}_{j=1}|s-p_{j}|^{\beta_{j}}}{{\prod^{m}_{l=1}|s-z_{l}|^{\gamma_{l}}}}=0$, then $s_{\infty}$ is an infinite pole of (1.1).

(2). If $s_{\infty}$ satisfies

$K=\lim\limits_{s\to \infty}\frac{\prod^{n}_{j=1}|s-p_{j}|^{\beta_{j}}}{{\prod^{m}_{l=1}|s-z_{l}|^{\gamma_{l}}}}=A$, in which $A\neq 0$ is a finite non-zero value, then $s_{\infty}$ is a general point with a finite gain.

(3). If $s_{\infty}$ satisfies

$K=\lim\limits_{s\to \infty}\frac{\prod^{n}_{j=1}|s-p_{j}|^{\beta_{j}}}{{\prod^{m}_{l=1}|s-z_{l}|^{\gamma_{l}}}}=+\infty$, then $s_{\infty}$ is an infinite zero of (1.1).
\end{lemma}

The continuity of $K$ with respect to $s\in\mathbb{C}\cup\{\infty\}$ follows directly from the expression of $K$ in (1.1). Applying Lemma 1.2, we directly obtain Lemma 1.4.
\begin{lemma}
For any $s\in\mathbb{C}\cup\{\infty\}$, the value $K$ takes all non-negative real numbers, ranging from 0 at poles of (1.1) to $+\infty$ at zeros of (1.1).
\end{lemma}

\begin{lemma}
For an arbitrary $s\in\Delta$, the phase angle of $RF$ at the left-hand side of (1.1) can be expressed as the function $\varphi(\sigma,t)$ given by: $\varphi(\sigma,t)=\arg{RF(\sigma+it)}$, or,
$\varphi(\sigma,t)=\sum^{m}_{l=1} \gamma_l\arg(s-z_l)- \sum^{n}_{j=1} \beta_j\arg(s-p_j)$.
\end{lemma}

The phase angle condition equation of (1.1) on $s=\sigma+it\in\Delta$ is as follows:
\begin{equation}
\varphi(\sigma,t)=2q\pi+\arg{(a+ib)}.
\end{equation}

In which, $q\in \mathbb{N}$.

Curves formed by the roots of (1.1) are initially called the locus of roots; once defined, they are referred to as the root locus of (1.1).

\begin{lemma}
If $s_1=(\sigma_1,t_1)\in\Delta$ satisfies (1.2) of $2q\pi+\arg{(a_1+ib_1)}$ degree, then $s_{1}$ must be on the locus of roots of $2q\pi+\arg{(a_1+ib_1)}$ degree of (1.1).
\end{lemma}
\begin{proof}
Suppose that $s_1=(\sigma_1,t_1)\in\Delta$ satisfies (1.2), which implies that $\varphi(\sigma_1,t_1)=2q\pi+\arg{(a_1+ib_1)}$, where $\varphi(\sigma,t)$ is the phase angle of $RF$. Substituting $s_{1}$ into $RF$ yields a complex value. By the phase angle expression $\varphi(\sigma,t)$ of the point $s_{1}$, and by the condition that the point $s_{1}$ satisfies (1.2), this complex value can be written as $K_1^{*}(a_{1}+ib_{1})$. $K_1^{*}$ is a positive real number and $2q\pi+\arg{(a_1+ib_1)}$ is the phase angle of $RF$ at the point $s_{1}$, which is the value of the right-hand side of (1.2).

By substituting $s_1$ into the gain expression $K=|\frac{\prod^{n}_{j=1}  (s-p_j)^{\beta_j}}{\prod^{m}_{l=1} (s-z_l)^{\gamma_l}}|$, via the definition of the gain, we have an unique gain $K_1=|1/RF(s_1)|$, where $K_1$ is the reciprocal $K_1^{*}$ of the modulus of $RF$ of the point $s_1$.
That is, $K_1^{*}=1/K_1$. Multiplying $RF(s_1)$ by $K_1$ yields: $K_1\frac{\prod^{m}_{l=1} (s_1-z_l)^{\gamma_l}}{\prod^{n}_{j=1}  (s_1-p_j)^{\beta_j}}=a_1+ib_1$. This equation represents a concrete situation in which the point $s_1$ satisfies (1.1). Therefore, $s_1$ represents a root of (1.1). Since $s_1$ satisfies (1.1), it is proven that the point $s_1$ is on the locus of roots of (1.1) with a gain of $K_1$ and a degree of $2q\pi+\arg{(a_1+ib_1)}$. Hence, if $s_1$ satisfies (1.2), then $s_1$ must be on the locus of roots of (1.1) with a gain of $K_1$ and a degree of $2q\pi+\arg{(a_1+ib_1)}$.
\end{proof}

\begin{lemma}
For any point $s_2=(\sigma_2,t_2)\in\Delta$, if $s_2$ is on the locus of roots of $2q\pi+\arg{(a_2+ib_2)}$ degree of (1.1), then $s_2$ must satisfy (1.2) of $2q\pi+\arg{(a_2+ib_2)}$ degree.
\end{lemma}
\begin{proof}
Assume that the point $s_2=\sigma_2+it_2\in\Delta$ is an arbitrary point on the locus of roots of $2q\pi+\arg(a_2+ib_2)$ degree of (1.1).
The locus of roots of the $2q\pi+\arg{(a_2+ib_2)}$ degree of (1.1) are obtained by (1.1). Therefore, the point $s_2$ surely satisfies (1.1).
$K_2\frac{\prod^{m}_{l=1} (s_2-z_l)^{\gamma_l}}{\prod^{n}_{j=1}  (s_2-p_j)^{\beta_j}}=a_2+ib_2$.
By Lemma 1.4, the gain $K_2$ is a positive real number, the phase angle of the positive real number is $2q\pi$. Therefore, the phase angle of $K_2$ is $2q_2\pi$. The phase angle of $a_2+ib_2$ is $2q\pi+\arg{(a_2+ib_2)}$.

The factor on the left-hand side of the equation in the previous paragraph has two factors: $K_2$ and $RF(s_2)$. Thus, the phase angle of this factor at the left-hand side of the equation in the previous paragraph on point $s_2$ is equal to the sum of the phase angles of the two factors $K_2$ and $RF(s_2)$. The phase angle $2q\pi+\arg{(a_2+ib_2)}$ subtracts the phase angle $2q_2\pi$ is equal to
$2q\pi+\arg{(a_2+ib_2)}-2q_2\pi=2q\pi+\arg{(a_2+ib_2)}$. The difference in the phase angles of factors $a_2+ib_2$ and factor $K_2$ is: $2q\pi+\arg{(a_2+ib_2)}$. Therefore, the phase angle of
$RF(s_2)$ is $2q\pi+\arg{(a_2+ib_2)}$. The phase angle of $RF$ is $\varphi(\sigma,t)=\sum^{m}_{l=1} \gamma_l\arg(s-z_l)-\sum^{n}_{j=1} \beta_j\arg(s-p_j)$. Therefore, via the previous proof, we obtain $ \varphi(\sigma_2,t_2)=2q\pi+\arg{(a_2+ib_2)}$, and point $s_2$ satisfies (1.2). Therefore, if the point $s_2$ lies on the locus of roots of (1.1) with a gain of $K_2$ and a degree of $2q\pi+\arg{(a_2+ib_2)}$, or satisfies (1.1), then the point $s_2$ satisfies (1.2).
\end{proof}

Lemma 1.6 provides the sufficient condition for Theorem 1.8. Lemma
1.7 provides the necessary condition for Theorem 1.8.  By combining Lemma
1.6 and Lemma 1.7, we obtain Theorem 1.8.

\begin{theorem}
Let $s=\sigma+it\in\Delta$ be an arbitrary point. The point $s$ lies on the locus of roots of $\alpha$ degree of (1.1) if and only if $s$ satisfies (1.2).
\end{theorem}
\begin{definition}
Let $\Xi$ be the point set of all points on the locus of roots of equation (1.1) traced out in $\mathbb{C}\cup\{\infty\}$ when $\alpha=2q\pi+\arg(a+ib)$ is a constant. The path of the roots of equation (1.1) in $\mathbb{C}\cup\{\infty\}$ is called \emph{the root locus of (1.1)}. In other words, the set $\Xi$ is \emph{the root locus of the $\alpha$ degree}.
\end{definition}

When $2q\pi+\arg{(a+ib)}$ is a constant, the points in $\mathbb{C}\cup\{\infty\}$ that satisfy equation (1.2) form curves. These curves are the root loci of (1.1) and (1.2). Theorem 1.8 proves that the root locus of (1.1) is the same as the root locus of (1.2).

Determining the root locus requires two elements. One is the unit complex value $a+ib$. Generally, distinct unit complex values $a+ib$ lead to distinct root loci. The other is the points on the root locus. Points on the root locus can be determined via (1.1), and they are both sufficient and necessary to be determined via (1.2).

Lemma 1.5 offers a method to calculate the degree of the root locus of (1.1) via the function $\varphi(\sigma,t)$. The factor $a+ib$ on the right-hand side of (1.1) serves as the computational value. In $\mathbb{C}\cup\{\infty\}$, the degree of the root locus of (1.1) is equal to the degree of the phase angle of $RF$. Therefore, we can define the degree of the root locus of (1.1) as follows:

\begin{definition}
When the phase angle of $RF$ is $\alpha$ degree, we call the root locus of (1.1) the \emph{$\alpha$ degree root locus}.
\end{definition}

\begin{lemma}
Let $z_{1}$ be a finite zero of (1.1) with exponent $\gamma_1$. Then, $z_{1}$ lies simultaneously on the root loci of all degrees of (1.1) from $2q\pi$ degree to $2q\pi+2\gamma_1\pi$ degree.
\end{lemma}

\begin{proof}
Assume that the point $z_{1}$ is an arbitrary finite zero of (1.1). When we substitute the point $z_{1}$ into $RF$, we obtain $\frac{\prod^{m}_{l=1}(z_1-z_{l})^{\gamma_l}}
{\prod^{n}_{j=1}(z_1-p_{j})^{\beta_j}}=0$.
This equation can also be written in polar form as $\frac{\prod^{m}_{l=1}(z_1-z_{l})^{\gamma_l}}
{\prod^{n}_{j=1}(z_1-p_{j})^{\beta_j}}=K_{z_1}e^{i\theta_{z_1}}=0$.
$K_{z_{1}}$ is the modulus of  $RF$ and $\theta_{z_{1}}$ is the argument of $RF$.
Since the product is zero, we have $K_{z_{1}}=0$. Then $K_{z_1}e^{i(\theta_{z_1})}=0\cdot(\cos(\theta_{z_1})+i\sin(\theta_{z_1}))=0$ holds for any value of $\theta_{z_{1}}$. Here, $\theta_{z_{1}}$ can take any value from the range $2q\pi$ degree to $2q\pi+2\gamma_1\pi$ degree.

For the point $z_{1}$ which makes the left-hand side of (1.1) equal to 0, regardless of the value of the argument $\theta_{z_{1}}$. $\cos(\theta_{z_{1}})+i\sin(\theta_{z_{1}})$ is the non-zero and finite, and the modulus of the factor $cos(\theta_{z_{1}})+isin(\theta_{z_{1}})$ is 1. $\theta_{z_{1}}$ represents an arbitrary degree, which shows the following: no matter what value $\theta_{z_{1}}$ is, because the modulus $K_{z_{1}}$ is 0,
there is $K_{z_{1}}e^{i(\theta_{z_{1}})}=0\cdot(\cos(\theta_{z_{1}})+i\sin(\theta_{z_{1}}))=0$.

When a finite zero $z_{1}$ is substituted into $RF$, for any argument $\theta_{z_{1}}$ of $RF$ with degree in the given range, the function values of $z_{1}$ are all equal to 0. By Definition 1.10, the argument of $RF$ is the degree of the root locus of (1.1). Since $\theta_{z_{1}}$ can take all values from $2q\pi$ degree to $2q\pi+2\gamma_1\pi$ degree, this implies that finite zeros of (1.1) lie simultaneously on the root loci of all degrees of (1.1) from $2q\pi$ degree to $2q\pi+2\gamma_1\pi$ degree.
\end{proof}

\begin{lemma}
Let $p_{1}$ be a finite pole of (1.1) with exponent $\beta_1$. Then, $p_{1}$ lies simultaneously on the root loci of all degrees of (1.1) from $2q\pi$ degree to $2q\pi+2\beta_1\pi$ degree.
\end{lemma}

\begin{proof}
Assume that the point $p_{1}$ is an arbitrary finite pole of (1.1). When we substitute the point $p_{1}$ into $RF$, we obtain $\frac{\prod^{m}_{l=1}(p_1-z_{l})^{\gamma_l}}
{\prod^{n}_{j=1}(p_1-p_{j})^{\beta_j}}=\infty$.

In complex analysis, 0 can be written as: $0\cdot e^{i\theta}=0\cdot(\cos\theta+i\sin\theta)$.
The non-zero finite values can also be written as: $k\cdot e^{i\theta}=k\cdot(\cos\theta+i\sin\theta)$.
Infinity can also be written as: $(+\infty)\cdot e^{i\theta}=(+\infty)\cdot(\cos\theta+i\sin\theta)$. In which, $(\cos\theta+i\sin\theta)$ is a non-zero and infinite, and its modulus is 1 of the unit complex number value.
For $(+\infty)\cdot e^{i\theta}=(+\infty)\cdot(\cos\theta+i\sin\theta)$, no matter what value $\theta$ is obtained on the unit circle in $\mathbb{C}\cup\{\infty\}$, the values of this expression are all infinity.

The equation above can also be expressed as: $\frac{\prod^{m}_{l=1}(p_1-z_{l})^{\gamma_l}}
{\prod^{n}_{j=1}(p_1-p_{j})^{\beta_j}}=K_{p_1}e^{i\theta_{p_1}}=\infty$, where $K_{p_1}$ is the modulus of  $RF$ and $\theta_{p_1}$ is the argument of  $RF$. Since the result is infinite, we have $K_{p_1}=+\infty$. Then $K_{p_1}\cdot e^{i\theta}=(+\infty)\cdot(\cos\theta+i\sin\theta)=\infty$ holds for any degree value of $\theta_{p_1}$ in the range of $2q\pi$ degree to $2q\pi+2\beta_1\pi$ degree.

When a finite pole $p_{1}$ of (1.1) is substituted into $RF$, for any argument $\theta_{p_{1}}$ of $RF$ in the given range, the function values of $p_{1}$ are all equal to infinity. By Definition 1.10, the argument of $RF$ is the degree of the root locus of (1.1). Since $\theta_{p_{1}}$ can take all values from $2q\pi$ degree to $2q\pi+2\beta_1\pi$ degree, this implies that the finite poles of (1.1) lie simultaneously on the root loci of all degrees of (1.1) from $2q\pi$ degree to $2q\pi+2\beta_1\pi$ degree.
\end{proof}

\begin{lemma}
$n$ finite poles of (1.1) are $n$ origination points of the root locus of any $\alpha$ degree of (1.1). $m$ finite zeros of (1.1) are $m$ ending points of the root locus of any $\alpha$ degree of (1.1).
\end{lemma}
\begin{proof}
The root locus of any $\alpha$ degree of (1.1) consists of curves in $\mathbb{C}\cup\{\infty\}$ that must have their own origination and ending points. On the basis of the relationship between the points in $\mathbb{C}\cup\{\infty\}$ and the root locus of (1.1), we can classify the points in $\mathbb{C}\cup\{\infty\}$ into four types: $n$ finite poles of (1.1), $m$ finite zeros of (1.1), general finite points on the root locus of $\alpha$ degree of (1.1), and the infinity point.

Except for $n$ finite poles and $m$ finite zeros of (1.1), a general finite point in $\mathbb{C}\cup\{\infty\}$ lies on the root locus of a specific degree of (1.1). These general points only have a phase angle and a non-zero finite gain value, so they are not suitable as the origination or ending points of the root locus of (1.1).

$n$ finite poles and $m$ finite zeros of (1.1) lie on the root loci of all degrees of (1.1). Since an infinite number of root locus curves of distinct degrees intersect at the same point, they meet the requirements for being the origination or ending points of the root locus. Moreover, for $n$ finite poles of (1.1), the gain $K=0$, which means that they act as origination points. For $m$ finite zeros of (1.1), the gain $K=+\infty$, which means that they act as ending points of the root locus.

The result of the infinity point is in Lemma 1.3. When (1) in Lemma 1.3 holds, the point  at infinity $s_{\infty}$ is an origination point of the root locus of (1.1). When (3) in Lemma 1.3 holds, the point at infinity $s_{\infty}$ is a ending point of the root locus of (1.1).
\end{proof}

By Lemma 1.2, Lemma 1.3 and Lemma 1.13 we have Theorem 1.14.

\begin{theorem}
The $2q\pi+\pi$  and $2q\pi$ degree root loci of (1.1) begin at $n$ poles of (1.1) or the infinity and end at $m$ zeros of (1.1) or the infinity.

1. When $\sum^n_{j=1} \beta_j > \sum^m_{l=1} \gamma_l$, $\sum^n_{j=1} \beta_j- \sum^m_{l=1} \gamma_l$ branches of the $2q\pi+\pi$  and $2q\pi$ degree root loci ending at the infinity in $\mathbb{C}\cup\{\infty\}$.

2. When $\sum^n_{j=1} \beta_j < \sum^m_{l=1} \gamma_l$, $ \sum^m_{l=1} \gamma_l - \sum^n_{j=1} \beta_j$ branches of the $2q\pi+\pi$  and $2q\pi$ degree root loci beginning at the infinity in $\mathbb{C}\cup\{\infty\}$.
\end{theorem}

For a point $s_1=\sigma_1+it_1$ on a root locus, substituting $s_1$ into the rational function $RF(s_1)$ yields a complex value $u(\sigma_1,t_1)+iv(\sigma_1,t_1)$ of $RF(s_1)$. The argument $\arg(u(\sigma_1,t_1)+iv(\sigma_1,t_1))$ of $RF(s_1)$ is referred to as the phase angle, which represents the degree of the root locus of (1.1). Each degree of the root locus of (1.1) is the degree value of an angle. Each degree of the root locus corresponds to a unique unit complex value. Consequently, each root locus is characterized by its degree value and its unit complex value. Root loci with distinct degrees may share the same unit complex value, and the difference in their degrees is $2q\pi$. The unit complex value of the root locus is the unit complex value of the degree of the root locus of (1.1).

\begin{theorem}
Two root loci of (1.1) with distinct unit complex values $a_1+ib_1$ and $a_2+ib_2$ cannot intersect in $\mathbb{C}\cup\{\infty\}$.
\end{theorem}
\begin{proof}
Assume that two root loci of (1.1) with distinct unit complex values $a_1+ib_1$ and $a_2+ib_2$ intersect at a point $s_0$, which is neither a zero nor a pole of (1.1).

Since $s_0$ is the same point, its gain $K_0$ should be the same. $s_0$ allows the following two equations to hold simultaneously. $K_0\frac{\prod^{m}_{l=1} (s_0-z_l)^{\gamma_l}}
{\prod^{n}_{j=1}  (s_0-p_j)^{\beta_j}}=a_{1}+ib_{1}$, and $K_0\frac{\prod^{m}_{l=1} (s_0-z_l)^{\gamma_l}}{\prod^{n}_{j=1}  (s_0-p_j)^{\beta_j}}=a_{2}+ib_{2}$.
Subtracting these equations yields $(a_{1}-a_{2})+i(b_{1}-b_{2})=0$. By the conditions of this theorem, $a_{1}\neq a_{2}$ or $b_{1}\neq b_{2}$, so this equation cannot hold. Since the contradiction arises from the assumption that the two root loci intersect at $s_0$, this assumption is false.
\end{proof}

\begin{theorem}
Two root loci of (1.1) with identical unit complex values $a+ib$ may intersect in $\mathbb{C}\cup\{\infty\}$.
\end{theorem}
\begin{proof}
Consider the root locus equation $\frac{K}{(s-1)(s-3)}=a+ib$. The interval $[1,3]$ contains root loci of $2q\pi+\pi$ degree from both poles at $s=1$ and $s=3$, each with a unit complex value of $-1$. These root loci intersect at point $s=2$.

By applying Rolle's theorem, there exists a critical point of the function $(s-1)(s-3)$ in the interval $[1,3]$. This critical point of $(s-1)(s-3)$ is $s=2$, which is the intersection point of the root loci.
\end{proof}

\begin{lemma}
All points in $\mathbb{C}\cup\{\infty\}$ are roots of the root locus equation (1.1). Except for zeros and poles of (1.1), each finite point in $\mathbb{C}\cup\{\infty\}$ lies on a root locus of (1.1). The root loci of (1.1) fill the entire extended complex plane.
\end{lemma}

\begin{proof}
By Lemma 1.2, at the zeros of (1.1), $K=+\infty$. The zeros of (1.1) satisfy (1.1) with $K=+\infty$. At the poles of (1.1), $K=0$. The poles of (1.1) satisfy (1.1) with $K=0$.

For any finite point $f_{a}=\sigma_a+it_a$ that is neither a zero nor a pole of (1.1), substituting $f_{a}$ into the gain expression $K=|\frac{\prod^{n}_{j=1}  (s-p_j)^{\beta_j}}{\prod^{m}_{l=1} (s-z_l)^{\gamma_l}}|$ yields a  positive, finite real number. Similarly, substituting $f_{a}$ into the phase angle expression $\varphi(\sigma,t)$ yields a finite certain phase angle $\varphi(\sigma_a,t_a)$. Therefore, the finite point $f_{a}$ must lie on a root locus of  $\varphi(\sigma_a,t_a)$ degree of (1.1) with a gain $K(f_{a})$ that is a positive, finite real number. Thus, $f_{a}$ satisfies (1.1) with the gain $K(f_{a})$.

By Lemma 1.3, points at infinity in $\mathbb{C}\cup\{\infty\}$ can be zeros, poles, or general points with finite gain values $A$. Repeating the previous proof, we conclude that all points at infinity in $\mathbb{C}\cup\{\infty\}$ lie on a root locus of (1.1) and satisfy (1.1) with a specific gain.
\end{proof}

\begin{definition}
If two or more root loci intersect in $\mathbb{C}\cup\{\infty\}$, these intersected root loci are called \emph{"common root loci"}. The points of intersection are called \emph{the breakaway points} of the root loci of (1.1).
\end{definition}
Multiplying both sides of (1.1) by $\prod (s-p_j)^{\beta_j}$ gives:
$  K(s) \prod_{l=1}^m (s-z_l)^{\gamma_l} =(a+ib)\prod_{j=1}^n (s-p_j)^{\beta_j},$
  which is a polynomial whose roots coincide with solutions of (1.1). All non-zero factors are moved to one side.  This transformation yields the characteristic equation of (1.1), which can be written as: $K(s)\prod^{m}_{l=1}(s-z_{l})^{\gamma_{l}}-
(a+ib)\prod^{n}_{j=1}(s-p_{j})^{\beta_{j}}=0$. The characteristic equation of (1.1) is equivalent to (1.1). Therefore, the roots of the characteristic equation are the roots of (1.1), and vice versa. We have Lemma 1.19.

\begin{lemma}
Equation (1.1) can be transformed into its characteristic equation. The roots of (1.1) and its characteristic equation are identical.
\end{lemma}

\begin{lemma}
For any finite point $s_0$ in the complex plane, the left-hand side of the characteristic equation with gain $K(s_0)$ can be rewritten as another expression: $(s-s_0)^{\gamma_0}g(s)$, where $g(s_0)\neq0$, $g(s_0)\neq\infty$ and $g^{'}(s_0)\neq\infty$.
\end{lemma}

\begin{proof}
By Lemma 1.17, the finite point $s_0$ lies on the root loci of (1.1). Substituting gain $K(s_0)$ of the finite point $s_0$ into the characteristic equation yields: $(a_0+ib_0)\frac{\prod^{n}_{j=1}(s_0-p_{j})^{\beta_{j}}}
{\prod^{m}_{l=1}(s_0-z_{l})^{\gamma_{l}}}\prod^{m}_{l=1}(s-z_{l})^{\gamma_{l}}-
(a+ib)\prod^{n}_{j=1}(s-p_{j})^{\beta_{j}}=0$.
Since $s_0$ is a root of the characteristic equation, the left-hand side of the characteristic equation can be expressed as: $(s-s_0)^{\gamma_0}g(s)$. $g(s_0)\neq0$. $(a_0+ib_0)\frac{\prod^{n}_{j=1}(s_0-p_{j})^{\beta_{j}}}
{\prod^{m}_{l=1}(s_0-z_{l})^{\gamma_{l}}}\prod^{m}_{l=1}(s-z_{l})^{\gamma_{l}}-
(a+ib)\prod^{n}_{j=1}(s-p_{j})^{\beta_{j}}=(s-s_0)^{\gamma_0}g(s)=0$.

Because the function $(a_0+ib_0)\frac{\prod^{n}_{j=1}(s_0-p_{j})^{\beta_{j}}}
{\prod^{m}_{l=1}(s_0-z_{l})^{\gamma_{l}}}\prod^{m}_{l=1}(s-z_{l})^{\gamma_{l}}-
(a+ib)\prod^{n}_{j=1}(s-p_{j})^{\beta_{j}}$ on the left side of the characteristic equation is a polynomial. After factoring out $(s-s_0)^{\gamma_0}$ in the polynomial $(a_0+ib_0)\frac{\prod^{n}_{j=1}(s_0-p_{j})^{\beta_{j}}}
{\prod^{m}_{l=1}(s_0-z_{l})^{\gamma_{l}}}\prod^{m}_{l=1}(s-z_{l})^{\gamma_{l}}-
(a+ib)\prod^{n}_{j=1}(s-p_{j})^{\beta_{j}}$, the remaining function $g(s)$ is still a polynomial. Polynomials do not have finite poles, and their derivatives are also polynomials. Therefore, $g(s)$ and $g^{'}(s)$ cannot have finite poles, $g(s_0)\neq\infty$ and $g^{'}(s_0)\neq\infty$.
\end{proof}
\begin{lemma}
Except for the zeros of (1.1), for all other points in the extended complex plane, (1.3) holds.
\begin{equation}
TT1(s)=TT2(s).
\end{equation}

In which, $TT1(s)=(a_0+ib_0)\prod^{n}_{j=1}(s_0-p_{j})^{\beta_{j}}(\prod^{m}_{l=1}(s-z_{l})^{\gamma_{l}})^{'}-
(a+ib)\prod^{n}_{j=1}(s-p_{j})^{\beta_{j}})^{'}
\prod^{m}_{l=1}(s_0-z_{l})^{\gamma_{l}}$.

$TT2(s)=((s-s_0)^{\gamma_0}g(s))^{'}
\prod^{m}_{l=1}(s_0-z_{l})^{\gamma_{l}}$.
\end{lemma}

\begin{proof}
Differentiate the equation: $(a_0+ib_0)\frac{\prod^{n}_{j=1}(s_0-p_{j})^{\beta_{j}}}
{\prod^{m}_{l=1}(s_0-z_{l})^{\gamma_{l}}}\prod^{m}_{l=1}(s-z_{l})^{\gamma_{l}}-
(a+ib)\prod^{n}_{j=1}(s-p_{j})^{\beta_{j}}=(s-s_0)^{\gamma_0}g(s)$.

$(a_0+ib_0)\frac{\prod^{n}_{j=1}(s_0-p_{j})^{\beta_{j}}}
{\prod^{m}_{l=1}(s_0-z_{l})^{\gamma_{l}}}
(\prod^{m}_{l=1}((s-z_{l})^{\gamma_{l}}))^{'}-(a+ib)
(\prod^{n}_{j=1}((s-p_{j})^{\beta_{j}}))^{'}
=((s-s_0)^{\gamma_0}g(s))^{'}$.

Since we exclude the zeros of (1.1), at the finite point $s_0$, the factor ${\prod^{m}_{l=1}(s_0-z_{l})^{\gamma_{l}}}$ of (1.1) is non-zero. Multiplying both sides by this factor ${\prod^{m}_{l=1}(s_0-z_{l})^{\gamma_{l}}}$ yields: $(a_0+ib_0)\prod^{n}_{j=1}(s_0-p_{j})^{\beta_{j}}(\prod^{m}_{l=1}(s-z_{l})^{\gamma_{l}})^{'}-
(a+ib)(\prod^{n}_{j=1}(s-p_{j})^{\beta_{j}})^{'}
\prod^{m}_{l=1}(s_0-z_{l})^{\gamma_{l}}=(\gamma_0(s-s_0)^{\gamma_0-1}g(s)+(s-s_0)^{\gamma_0}g^{'}(s))
\prod^{m}_{l=1}(s_0-z_{l})^{\gamma_{l}}$.

Let $TT1(s)=(a_0+ib_0)\prod^{n}_{j=1}(s_0-p_{j})^{\beta_{j}}(\prod^{m}_{l=1}(s-z_{l})^{\gamma_{l}})^{'}-
(a+ib)(\prod^{n}_{j=1}(s-p_{j})^{\beta_{j}})^{'}
\prod^{m}_{l=1}(s_0-z_{l})^{\gamma_{l}}$.

$TT2(s)=((s-s_0)^{\gamma_0}g(s))^{'}
\prod^{m}_{l=1}(s_0-z_{l})^{\gamma_{l}}$. Using these two new expressions, we obtain (1.3).
\end{proof}
The breakaway points must lie on the root loci sharing the same unit complex value. Therefore, in $TT1(s)$, the sign $a_0+ib_0$ before $\prod^{n}_{j=1}(s_0-p_{j})^{\beta_{j}}(\prod^{m}_{l=1}(s-z_{l})^{\gamma_{l}})^{'}$ is identical to the sign $a+ib$ before $(\prod^{n}_{j=1}(s-p_{j})^{\beta_{j}})^{'}\prod^{m}_{l=1}(s_0-z_{l})^{\gamma_{l}}$.

 The gain function is given by: $K(s)=(a+ib)\frac{\prod^{n}_{j=1}  (s-p_j)^{\beta_j}}{\prod^{m}_{l=1} (s-z_l)^{\gamma_l}}$.
\begin{lemma}
Let $s_0=(\sigma_0+it_0)\in\Delta $. If the finite point $s_0$ satisfies equation (1.4),
\begin{equation}
\frac{dK(s)}{ds}=0.
\end{equation}
Then, $TT1(s_0)=0$ and $((s-s_0)^{\gamma_0}g(s))^{'}\mid{_{s=s_0}}=0$.
\end{lemma}

\begin{proof}
Substituting the finite point $s_0$ into (1.4), according to the requirement of this lemma, we obtain:

$\frac{dK(s_0)}{ds}
=(a_0+ib_0)\frac{(\prod^{n}_{j=1}(s-p_{j})^{\beta_{j}})^{'}\mid{_{s=s_0}}
\prod^{m}_{l=1}(s_0-z_{l})^{\gamma_{l}}-
\prod^{n}_{j=1}(s_0-p_{j})^{\beta_{j}}
}{(\prod^{m}_{l=1}(s_0-z_{l})^{\gamma_{l}})^{2}}$

$\frac{(\prod^{m}_{l=1}((s-z_{l})^{\gamma_{l}}))^{'}\mid{_{s=s_0}}}{}=0$.

From the result of $\frac{dK(s_0)}{ds}$, substituting $s_0$ into the left side of (1.3) yields: $TT1(s_0)=(a_0+ib_0)(\prod^{n}_{j=1}(s_0-p_{j})^{\beta_{j}}(\prod^{m}_{l=1}(s-z_{l})^{\gamma_{l}})^{'}\mid{_{s=s_0}}-
(\prod^{n}_{j=1}(s-p_{j})^{\beta_{j}})^{'}\mid{_{s=s_0}}
\prod^{m}_{l=1}(s_0-z_{l})^{\gamma_{l}})=0.$
Thus, the left side of (1.3) equals 0.

For the right side of (1.3), consider the expression: $TT2(s)=((s-s_0)^{\gamma_0}g(s))^{'}
\prod^{m}_{l=1}(s_0-z_{l})^{\gamma_{l}}$. Since the factor $\prod^{m}_{l=1}(s-z_{l})^{\gamma_{l}}$  has no finite pole, and $s_0$ is neither a zero nor a pole of the function $\prod^{m}_{l=1}(s-z_{l})^{\gamma_{l}}$, if $TT2(s_0)=((s-s_0)^{\gamma_0}g(s))^{'}
\prod^{m}_{l=1}(s_0-z_{l})^{\gamma_{l}}=0$, we can remove the factor $\prod^{m}_{l=1}(s_0-z_{l})^{\gamma_{l}}$. This yields: $((s-s_0)^{\gamma_0}g(s))^{'}\mid{_{s=s_0}}=0$.
\end{proof}

\begin{lemma}
Let $s_0=(\sigma_0+it_0)\in\Delta $. If the finite point $s_0$ satisfies the equation (1.4), then $s_0$ is the finite breakaway point of the root loci of (1.1).
\end{lemma}

\begin{proof}
We analyse three cases on the basis of $\gamma_0>0$:

1.  Case $\gamma_0=1$:

If $\gamma_0=1$,  $g^{'}(s_0)\neq\infty$, $(s_0-s_0)^{\gamma_0}g^{'}(s_0)
=0$. $g(s_0)\neq 0$. $(g(s_0)+(s_0-s_0)^{\gamma_0}g^{'}(s_0))=g(s_0)\neq 0$.

$((s-s_0)^{\gamma_0}g(s))^{'}\mid{_{s=s_0}}=(g(s_0)+(s_0-s_0)^{\gamma_0}g^{'}(s_0))
=g(s_0)\neq 0$. By Lemma 1.22, $((s-s_0)^{\gamma_0}g(s))^{'}\mid{_{s=s_0}}=0$. Therefore, This leads to a contradiction. Hence, $\gamma_0=1$ is invalid.

2. Case $1>\gamma_0>0$:

If $1>\gamma_0>0$,  $g^{'}(s_0)\neq\infty$, $(s_0-s_0)^{\gamma_0}g^{'}(s_0)
=0$. $\gamma_0(s_0-s_0)^{\gamma_0-1}=\infty$, $g(s_0)\neq 0$. $((s-s_0)^{\gamma_0}g(s))^{'}\mid{_{s=s_0}}=(\gamma_0(s_0-s_0)^{\gamma_0-1}g(s_0)+(s_0-s_0)^{\gamma_0}g^{'}(s_0))=\infty$.

$((s-s_0)^{\gamma_0}g(s))^{'}\mid{_{s=s_0}}=\infty$.

By Lemma 1.22, $((s-s_0)^{\gamma_0}g(s))^{'}\mid{_{s=s_0}}=0$. Therefore, This leads to a contradiction. Hence, $1>\gamma_0>0$ is invalid.

3. Case $\gamma_0>1$:

Because $\gamma_0>1$, $((s-s_0)^{\gamma_0}g(s))^{'}\mid{_{s=s_0}}=\gamma_0(s_0-s_0)^{\gamma_0-1}g(s_0)+(s_0-s_0)^{\gamma_0}g^{'}(s_0)=0$ holds. The left-hand side of (1.3) is equal to the right-hand side of (1.3). (1.3) holds. Hence, $\gamma_0>1$ is valid.

By Lemma 1.19, the roots of (1.1) and its characteristic equation are identical.
The left-hand side of the characteristic equation with gain $K(s_0)$ is rewritten as another expression: $(s-s_0)^{\gamma_0}g(s)$. $s_0$ is a root of (1.1).  $\gamma_0$ is a positive integer. Therefore, the condition $\gamma_0>1$ implies that the point $s_0$ is a multiple root of (1.1), and the finite point $s_0$ is the intersection point of at least two root loci. The finite point $s_0$ is the finite breakaway point of (1.1).

Only $\gamma_0>1$ is valid, confirming that $s_0$ is a finite breakaway point.
\end{proof}
\begin{lemma}
Let $s_0=(\sigma_0+it_0)\in\Delta $. If $s_0$ is a finite breakaway point of the root loci of (1.1), then $s_0$ must satisfy (1.4).
\end{lemma}
\begin{proof}
From (1.3), we obtain:

$(a_0+ib_0)\frac{\prod^{n}_{j=1}(s_0-p_{j})^{\beta_{j}}}
{\prod^{m}_{l=1}(s_0-z_{l})^{\gamma_{l}}}
(\prod^{m}_{l=1}((s-z_{l})^{\gamma_{l}}))^{'}-
(a+ib)(\prod^{n}_{j=1}((s-p_{j})^{\beta_{j}}))^{'}
=\gamma_0(s-s_0)^{\gamma_0-1}g(s)+(s-s_0)^{\gamma_0}g^{'}(s)$. $g(s_0)\neq\infty$ and  $g^{'}(s_0)\neq\infty$.

Substituting $s_0$ into the right-hand side of the above equation, since $s_0$ is the finite breakaway point of (1.1), $s_0$ is a multiple root of the characteristic equation. Therefore, $\gamma_0>1$. $\gamma_0(s_0-s_0)^{\gamma_0-1}g(s_0)+(s_0-s_0)^{\gamma_0}g^{'}(s_0)=0$. $TT2(s_0)=0$.
By Lemma 1.21, $TT1(s_0)=0$, we obtain: $(a_0+ib_0)(\prod^{n}_{j=1}(s_0-p_{j})^{\beta_{j}}
(\prod^{m}_{l=1}(s-z_{l})^{\gamma_{l}})^{'}\mid{_{s=s_0}}-
\prod^{m}_{l=1}(s_0-z_{l})^{\gamma_{l}}
(\prod^{n}_{j=1}(s-p_{j})^{\beta_{j}})^{'}\mid{_{s=s_0}})
=0$.
Differentiating the gain function and substituting $s_0$ into the differential expression, we obtain:

$\frac{dK(s_0)}{ds}
=(a_0+ib_0)\frac{(\prod^{n}_{j=1}(s-p_{j})^{\beta_{j}})^{'}\mid{_{s=s_0}}
\prod^{m}_{l=1}(s_0-z_{l})^{\gamma_{l}}-
\prod^{n}_{j=1}(s_0-p_{j})^{\beta_{j}}
}{(\prod^{m}_{l=1}(s_0-z_{l})^{\gamma_{l}})^{2}}$

$\frac{(\prod^{m}_{l=1}((s-z_{l})^{\gamma_{l}}))^{'}\mid{_{s=s_0}}}{}
=0$.
\end{proof}

Lemma 1.23 proves the sufficient condition for breakaway points, whereas Lemma 1.24 proves the necessary condition. Combining these two lemmas, we obtain the following sufficient and necessary conditions for breakaway points.

\begin{theorem}
Let $s_0=(\sigma_0+it_0)\in\Delta $. The finite point $s_0$ is a finite breakaway point of the root loci of (1.1) if and only if $s_0$ satisfies (1.4).
\end{theorem}

When the unit complex values of the root loci are distinct, these root loci cannot intersect. When the unit complex values of the root loci are identical, (under these conditions)the root loci may intersect. The intersection points at finite locations are finite breakaway points. These intersected root loci separate from these finite breakaway points again. Multiple zeros and poles of the root locus equation cannot exhibit this property. Thus, the finite breakaway points of the root loci are distinct from the multiple zeros and poles of (1.1).

\begin{theorem}
Let $s=(\sigma+it)\in\Delta$. (Except at the multiple zeros and multiple poles of (1.1).)
Then,

1. $K^{'}(s)$ and $RF^{'}(s)$ share the same zeros.

2. Each finite breakaway point of (1.1) is a finite zero of $K^{'}(s)$ and $RF^{'}(s)$.

3. Each finite zero of $K^{'}(s)$ and $RF^{'}(s)$ is a finite breakaway point of (1.1).
\end{theorem}
\begin{proof}
In the proof of Lemma 1.22 and Lemma 1.24, we present the expression of $\frac{dK(s)}{ds}$. Therefore, we omit it here. For $RF(s)=\frac{\prod^{m}_{l=1} (s-z_l)^{\gamma_l}}{ \prod^{n}_{j=1}  (s-p_j)^{\beta_j}}$. Its derivative is as follows:

$(RF(s))^{'}
=\frac{(\prod^{m}_{l=1} (s-z_l)^{\gamma_l})^{'}\prod^{n}_{j=1}  (s-p_j)^{\beta_j}-\prod^{m}_{l=1} (s-z_l)^{\gamma_l}(\prod^{n}_{j=1}  (s-p_j)^{\beta_j})^{'}}{(\prod^{n}_{j=1}  (s-p_j)^{\beta_j})^2}$.

The multiple zeros and multiple poles of $RF$ are zeros of the numerator polynomials of $\frac{dK(s)}{ds}$ and $(RF(s))^{'}$. Specifically:

The multiple zeros of $RF$ are zeros of $(RF(s))^{'}$, but not zeros of $\frac{dK(s)}{ds}$.

The multiple poles of $RF$ are zeros of $\frac{dK(s)}{ds}$, but not zeros of $(RF(s))^{'}$.

1. Excluding the multiple zeros and multiple poles of the $RF$, since $-(a+ib)$ lies on the unit circle and cannot be zero, the critical points of $K(s)$ and $RF(s)$ are identical.

2. The zeros of (1.4) are the critical points of $K(s)$. By Lemma 1.24, each finite breakaway point of (1.1) is finite critical point of $K(s)$, so, each finite breakaway point of (1.1) is finite critical point of $RF(s)$.

3. Excluding the multiple zeros and multiple poles of the $RF$, by Theorem 1.25, each zero of $K^{'}(s)$  is finite breakaway point of (1.1). Each finite zero of  $RF^{'}(s)$ is finite breakaway point of (1.1).
\end{proof}

The multiple zeros can be critical points of $RF$. Similarly, the multiple poles can be critical points of $K(s)$. However, they cannot be breakaway points of (1.1), so we exclude them in Theorem 1.25 and Theorem 1.26.

Theorem 1.25 shows the breakaway points of (1.1) are identical to the critical points of $K(s)$. Theorem 1.26 shows the critical points of $K(s)$ are identical to the critical points of $RF(s)$. Therefore, Theorem 1.26 also shows the breakaway points of (1.1) are identical to the critical points of $RF(s)$. By  Theorem 1.25, the intersection points of the root loci of (1.1) are critical points of $RF(s)$.
That is, the breakaway points of (1.1) are critical points of $RF(s)$. So, we have Theorem 0.5.

\begin{lemma}
Assume $RF$ is a real constant coefficient rational function. Then:

(1). The root loci of $2q\pi$ and $2q\pi+\pi$ degree of (1.1) are symmetric about the real axis.

(2). The root loci on the real axis are the root loci of $2q\pi$ and $2q\pi+\pi$ degree of (1.1).

(3). An interval on the real axis is a $2q\pi+\pi$ degree root locus if and only if the number of real poles and zeros of (1.1) to its right is odd.

(4). An interval on the real axis is a $2q\pi$ degree root locus if and only if the number of real poles and zeros of (1.1) to its right is even.

(5). The interval between two adjacent zero and pole (zero-pole, zero-zero, or pole-pole) is the root loci of the same degree.
\end{lemma}
\begin{proof}
(1). The rational function can be written as $RF(s)=u(\sigma,t)+iv(\sigma,t)$. Since $RF$ is a real constant coefficient rational function, it is a conjugate function. For the conjugate function $RF$, $RF(\overline{s})=u(\sigma,t)-iv(\sigma,t)$. For the root loci of $2q\pi$ and $2q\pi+\pi$ degree of (1.1), $v(\sigma,t)=0$. Therefore, for all points on these root loci of $2q\pi$ and $2q\pi+\pi$ degree of (1.1), $RF(s)=RF(\overline{s})=u(\sigma,t)$. Thus, the root loci of $2q\pi$ and $2q\pi+\pi$ degree of (1.1) are symmetric about the real axis.

(2). Since $RF$ is a real constant coefficient rational function, it takes real values on the real axis. Hence, the root loci on the real axis are the root loci of $2q\pi$ and $2q\pi+\pi$ degree of (1.1).

For any point $s_{0}$ on the selected section of the real axis, the zeros and poles of (1.1) are conjugate symmetric with respect to the real axis. The sum of the phase angles of the vectors between conjugate complex numbers and an arbitrary point on the real axis is $2\pi$. Therefore, determining the phase angle of an arbitrary point on the real axis does not require considering the effects of conjugate zeros and poles. When determining the degree of the root locus of a real constant coefficient rational function on the real axis, we do not need to consider the conjugate zeros and poles.

The phase angle of vector between a real zero  on the left side of $s_{0}$ and point $s_{0}$ is $2\pi$. The phase angle of vector between a real pole on the left side of $s_{0}$ and point $s_{0}$ is $2\pi$. Thus, determining the phase angle of an arbitrary point on the real axis does not need to consider the real zeros and real poles on the left side of $s_{0}$. When determining the degree of the root locus of a real constant coefficient rational function on the real axis, we do not need to consider the real zeros and real poles on the left side of $s_{0}$.

Let $z_{l}$ be a real zero to the right of $s_{0}$, and let $\gamma_l$ be the exponent of $z_{l}$. The phase angle of the vector between $z_{l}$ and $s_{0}$ is $\gamma_l\pi$. The sum of the phase angles of the vectors between the real zeros on the right side of $s_{0}$ and $s_{0}$ is $\sum_{l=1}^{m_1} \gamma_l\pi$. Similarly, let $p_{j}$ be a real pole to the right of $s_{0}$, and let $\beta_j$ be the exponent of $p_{j}$. The phase angle of the vector between $p_{j}$ and $s_{0}$ is $\beta_j\pi$. The sum of the phase angles of the vectors between the real poles on the right side of $s_{0}$ and $s_{0}$ is $\sum_{j=1}^{n_1} \beta_j\pi$. The degree of the root locus of a real constant coefficient rational function in this interval is $\sum_{l=1}^{m_1} \gamma_l\pi-\sum_{j=1}^{n_1} \beta_j\pi=(\sum_{l=1}^{m_1} \gamma_l-\sum_{j=1}^{n_1} \beta_j)\pi$.

In $(\sum_{l=1}^{m_1} \gamma_l-\sum_{j=1}^{n_1} \beta_j)\pi$, $-\pi$ and $\pi$ represents the same degrees, subtracting an integer multiple of $\pi$ is equivalent to adding an integer multiple of $\pi$. Therefore, $(\sum_{l=1}^{m_1} \gamma_l-\sum_{j=1}^{n_1} \beta_j)\pi=(\sum_ {l=1}^{m_1} \gamma_l+\sum_{j=1}^{n_1} \beta_j)\pi$. By $(\sum_ {l=1}^{m_1} \gamma_l+\sum_{j=1}^{n_1} \beta_j)\pi$, (3) and (4) hold.

(5). By $(\sum_ {l=1}^{m_1} \gamma_l+\sum_{j=1}^{n_1} \beta_j)\pi$, the phase angles of all the points in this interval are identical. The interval between two adjacent zero and pole (zero and pole, zero and zero, pole and pole) is the same degree root loci.
\end{proof}

\begin{theorem}
1. The number of  breakaway points of (1.1) is finite. The breakaway points are isolated to distributed in $\mathbb{C}\cup\{\infty\}$

2. Any segment of all the root loci of an arbitrary degree cannot overlap in the finite positions of $\mathbb{C}\cup\{\infty\}$.
\end{theorem}

\begin{proof}
1. To compute the finite breakaway points and prove their existence, we use the formula for computing finite breakaway points. Equation (1.4) provides a necessary and sufficient condition for computing the finite breakaway point:

$\frac{dK(s)}{ds}
=(a+ib)\frac{(\prod^{n}_{j=1}(s-p_{j})^{\beta_{j}})^{'}
\prod^{m}_{l=1}(s-z_{l})^{\gamma_{l}}-
\prod^{n}_{j=1}(s-p_{j})^{\beta_{j}}(\prod^{m}_{l=1}(s-z_{l})^{\gamma_{l}})^{'}
}{(\prod^{m}_{l=1}(s-z_{l})^{\gamma_{l}})^{2}}=0.$

$\frac{dK(s)}{ds}$ has only finite number of zeros. The breakaway points are zeros of $\frac{dK(s)}{ds}$. These finite zeros must be isolated in $\mathbb{C}\cup\{\infty\}$; they cannot be continuously distributed. Therefore, these finite zeros cannot form a continuous curve.

2. The finite breakaway point is the intersection point of the root locus of the root locus equation (1.1). An overlapping segment consists of continuous intersection points of the root locus of (1.1). If such overlapping segments exist, they would be special intersection points, specifically an infinite number of continuous intersection points.  On the basis of the previous proof, any segment of all the root loci of an arbitrary degree cannot overlap.
\end{proof}
\begin{definition}
In $\mathbb{C}\cup\{\infty\}$, if an angle $\phi$ is between the positive direction of the real axis and the tangent line of the root loci of arbitrary $2q\pi+\alpha$ degree that originates from a finite pole, then the angle $\phi$ is called the \emph{angle of origination} of the root loci of $2q\pi+\alpha$ degree of (1.1).
\end{definition}

\begin{definition}
In $\mathbb{C}\cup\{\infty\}$, if an angle $\vartheta$ is between the positive direction of the real axis and the tangent line of the root loci of arbitrary $2q\pi+\alpha$ degree that terminates at a finite zero, then the angle $\vartheta$ is called the \emph{angle of receiving} of the root loci of $2q\pi+\alpha$ degree of (1.1).
\end{definition}

The \emph{angle of origination} and \emph{angle of receiving} are both positive in the counterclockwise orientation.

\begin{theorem}
In $\mathbb{C}\cup\{\infty\}$, let $P_k$ be a finite pole of (1.1) with multiplicity $\beta_k$. The angle of origination of the root locus of an arbitrary $2q\pi+\alpha$ degree of (1.1), which originates from the finite pole $P_k$, is given by:

$\phi_{k}=\frac{1}{\beta_k}(2q\pi-\alpha+\sum\limits_{l=1}^{m}
(\gamma_{l}arg{(P_k-z_{l})})-\sum\limits_{j=1,j \neq k}^{n}
(\beta_{j}arg{(P_k-p_{j})})).$
\end{theorem}
\begin{proof}
Consider a finite point $s_1$ on the root locus of $2q\pi+\alpha$ degree. Let $s_1$ approach infinitely close to the finite pole $P_k$ with multiplicity $\beta_k$, which needs to compute its angle of origination. Since $s_1$ is infinitely close to $P_k$, the phase angles $\beta_k arg{(s_1-p_j)}$ of vectors between the finite point $s_{1}$ and the finite poles $p_j$, $j\neq k$. (Namely, except for the finite pole $P_k$.) can all be substituted by angles $\beta_{k}arg{(P_k-p_{j})}$, $j\neq k$. Similarly, the phase angle $\gamma_{k}arg{(s_1-z_{l})}$ of vectors between the finite point $s_{1}$ and all of the finite zeros $z_l$ of (1.1) all can be substituted by angles $\gamma_{k}arg{(P_k-z_{l})}$.

Therefore, the vector between finite point $s_{1}$ and finite pole $P_k$ is infinitely approaching to the tangent line of the root locus of arbitrary $2q\pi+\alpha$ degree that originates from $P_k$. By Definition 1.29, the phase angle of vectors between finite point $s_{1}$ and finite pole $P_k$ be infinitely approach the angle of origination $\phi_{k}$.

The finite point $s_{1}$ must satisfy the phase angle condition equation,
$\sum\limits_{l=1}^{m}
(\gamma_{l}arg{(s_1-z_{l})})-\sum\limits_{j=1}^{n}
(\beta_{j}arg{(s_1-p_{j})})=2q\pi+\alpha$.
Substituting $s_{1}$ with $P_k$, we obtain: $\sum\limits_{l=1}^{m}
(\gamma_{l}arg{(P_k-z_{l})})-\sum\limits_{j=1,j \neq k}^{n}
(\beta_{j}arg{(P_k-p_{j})})-\beta_k\phi_{k}=2q\pi+\alpha$.
$\beta_k\phi_{k}=\sum\limits_{l=1}^{m} (\gamma_{l}arg{(P_k-z_{l})})-\sum\limits_{j=1,j \neq k}^{n}
(\beta_{j}arg{(P_k-p_{j})})-2q\pi-\alpha$.

Since $2q\pi$ and $-2q\pi$ are equivalent, we substitute $2q\pi$ with $-2q\pi$. After transposing, we obtain the formula in Theorem 1.31.
\end{proof}

By repeating the proof of Theorem 1.31, we obtain Theorem 1.32.

\begin{theorem}
In $\mathbb{C}\cup\{\infty\}$, let $Z_k$ be a finite zero of (1.1) with multiplicity $\gamma_k$. The angle of receiving on the root locus of an arbitrary $2q\pi+\alpha$ degree of (1.1), which ends at the finite zero $Z_k$, is given by:

$\vartheta_{k}=\frac{1}{\gamma_k}(2q\pi+\alpha-\sum\limits_{l=1, l \neq k}^{m}
(\gamma_{l}arg{(Z_k-z_{l})})+\sum\limits_{j=1}^{n}
(\beta_{j}arg{(Z_k-p_{j})})).$
\end{theorem}

\begin{theorem}
For the root locus originating from an arbitrary finite pole $P_k$ of (1.1) in $\mathbb{C}\cup\{\infty\}$, if the angle $\phi_{k}$ of origination is treated as an independent variable, then the degree values of these root loci strictly and monotonically increase in the clockwise direction around $P_k$.
\end{theorem}
\begin{proof}
The angle of origination of the $\alpha_1$ degree root locus of (1.1) at the finite pole $P_k$ is given by:
$\phi_{k,1}=\frac{1}{\beta_k}(2q\pi-\alpha_1+\sum\limits_{l=1}^{m}
(\gamma_{l}arg{(P_k-z_{l})})-\sum\limits_{j=1,j \neq k}^{n}
(\beta_{j}arg{(P_k-p_{j})}))$.

Similarly, the angle of origination of the $\alpha_2$ degree root locus of (1.1) at the finite pole is as follows:
$\phi_{k,2}=\frac{1}{\beta_k}(2q\pi-\alpha_2+\sum\limits_{l=1}^{m}
(\gamma_{l}arg{(P_k-z_{l})})-\sum\limits_{j=1,j \neq k}^{n}
(\beta_{j}arg{(P_k-p_{j})}))$. Subtracting these two equations yields: $\phi_{k,1}-\phi_{k,2}=\frac{1}{\beta_k}(\alpha_2-\alpha_1)$, thus, $\alpha_2-\alpha_1=\beta_k(\phi_{k,1}-\phi_{k,2})$.

Here, $\beta_k$ is the exponent of the finite pole $P_k$, a positive constant integer. The angles $\phi_{k,1}$ and $\phi_{k,2}$ represent the origination angles of the $\alpha_1$-degree and $\alpha_2$-degree root loci, respectively. The angle of origination is defined as the angle between the positive real axis and the tangent line of the root locus originating from a finite pole $P_k$, measured counterclockwise from the positive real axis.

1. If $\phi_{k,1}>\phi_{k,2}$, $\phi_{k,1}-\phi_{k,2}>0$. The angle of origination of the root loci of $\alpha_1$ degree is $\phi_{k,1}$. The angle of origination of the root loci of $\alpha_2$ degree is $\phi_{k,2}$. $\alpha_2-\alpha_1>0$, $\alpha_2>\alpha_1$ holds. Therefore, for two root loci that originate at two distinct angles of origination from a same pole, the root locus of the higher degree is in the clockwise direction of the root locus of the lower degree.

2. If $\phi_{k,1}<\phi_{k,2}$, $\phi_{k,1}-\phi_{k,2}<0$, then $\alpha_2-\alpha_1<0$, $\alpha_2<\alpha_1$ holds. Therefore, for two root loci that originate at two distinct angles of origination from the same pole, the root locus of the higher degree is in the clockwise direction of the root locus of the lower degree.

In both cases, the degree values of the root loci originating from $P_k$ strictly and monotonically increase in the clockwise direction around $P_k$.
\end{proof}

By $\alpha_2-\alpha_1=\gamma_k(\varphi_{k,2}-\varphi_{k,1})$ and repeating the proof of Theorem 1.33, we can obtain Theorem 1.34.

\begin{theorem}
For the root locus received by an arbitrary finite zero $Z_k$ of (1.1) in $\mathbb{C}\cup\{\infty\}$, the angle of receiving $\vartheta_{k}$ is an independent variable. Rotating the finite zero $Z_k$, the degree values of these root loci strictly and monotonously decrease clockwise around $Z_k$.
\end{theorem}

\begin{theorem}
Assume that a circle is rotated from $0$ to $2\pi$ degrees around a pole $P_k$. The pole $P_k$ with the exponent $\beta_k$ of (1.1) emits root loci from $2q\pi$ degrees to $2q\pi+2\beta_k\pi$ degrees. Specifically, the pole $P_k$ with the exponent $\beta_k$ of (1.1) emits $\beta_k$ branches of $2q\pi$ degree root loci and $\beta_k$ branches of $2q\pi+\pi$ degree root loci.
\end{theorem}
\begin{proof}
By Theorem 1.31, the two angles of departure at the pole $P_k$ with the exponent $\beta_k$ in the same direction are given by:

$\phi_{k}=\frac{1}{\beta_k}(2q\pi-\alpha_1+\sum\limits_{l=1}^{m}
(\gamma_{l}arg{(P_k-z_{l})})-\sum\limits_{j=1,j \neq k}^{n}
(\beta_{j}arg{(P_k-p_{j})})$.

$\phi_{k}+2\pi=\frac{1}{\beta_k}(2q\pi-\alpha_2+\sum\limits_{l=1}^{m}
(\gamma_{l}arg{(P_k-z_{l})})-\sum\limits_{j=1,j \neq k}^{n}
(\beta_{j}arg{(P_k-p_{j})})$.

Subtracting the two equations side by side, we obtain:  $-2\pi=\frac{1}{\beta_k}(-\alpha_1+\alpha_2)$.
 $-2\pi\beta_k=-\alpha_1+\alpha_2$. Thus, $\alpha_1-\alpha_2=2\beta_k\pi$.

$\alpha_1-\alpha_2=2\beta_k\pi$. This result shows that the pole $P_k$ with the exponent $\beta_k$ of (1.1) emits root loci from $2q\pi$ degree to $2q\pi+2\beta_k\pi$ degree. The root loci from $2q\pi$ degree to $2q\pi+2\beta_k\pi$ degree, consist of the $\beta_k$ branches of $2q\pi$ degree root loci and $\beta_k$ branches of $2q\pi+\pi$ degree root loci. Therefore, the pole $P_k$ with the exponent $\beta_k$ of (1.1) emits $\beta_k$ branches of $2q\pi$ degree root loci and $\beta_k$ branches of $2q\pi+\pi$ degree root loci.
\end{proof}

By repeating the proof of Theorem 1.35, we can obtain Theorem 1.36.

\begin{theorem}
Assume that a circle is rotated from 0 to $2\pi$ degrees around a zero $Z_k$, the zero $Z_k$ with the exponent $\gamma_k$ of (1.1) receives the root loci from $2q\pi$ degree to $2q\pi+2\gamma_k\pi$ degree. Specifically, the zero $Z_k$ with the exponent $\gamma_k$ of (1.1) receives $\gamma_k$ branches of $2q\pi$ degree root loci and $\gamma_k$ branches of $2q\pi+\pi$ degree root loci.
\end{theorem}

\begin{lemma}
For each root locus of (1.1), no two distinct points share the same gain value.
\end{lemma}
\begin{proof}
We consider two cases on the basis of the relationship between the number of poles and zeros in (1.1).

Case 1: $\sum^n_{j=1} \beta_j \ge \sum^m_{l=1} \gamma_l$.

Since (1.1) has $\sum^n_{j=1} \beta_j$ poles, each pole emits a root locus of (1.1) with degree $\arg(a_f+ib_f)$. Thus, (1.1) emits $\sum^n_{j=1} \beta_j$ root loci, each with degree $\arg(a_f+ib_f)$. By Theorem 1.3, the gain values of points on these root loci with degree $\arg(a_f+ib_f)$ range from 0 at the poles to positive infinity at the zeros. Consequently, on each root locus of the $\sum^n_{j=1} \beta_j$ root loci, there exists at least one point with gain value $\frac{1}{\abs{FF}}$. Therefore, there exist at least $\sum^n_{j=1} \beta_j$ points with phase angle $\arg(a_f+ib_f)$ and gain value $\frac{1}{\abs{FF}}$.

There are at least $\sum^n_{j=1} \beta_j$ points that satisfy the equation $\frac{1}{\abs{FF}}RF(s)=a_f+ib_f$, which implies that $RF(s)=\abs{FF}(a_f+ib_f)$. Letting $FF=\abs{FF}(a_f+ib_f)$, $FF$ is a non-zero complex constant. Therefore, $RF(s)=FF$. $FF\prod^{n}_{j=1}(s-p_{j})^{\beta_j}-\prod^{m}_{l=1}(s-z_{l})^{\gamma_l}=0$, we obtain an $\sum^n_{j=1} \beta_j$ order polynomial equation. This equation has exactly $\sum^n_{j=1} \beta_j$ roots. Hence, it is impossible for two points on the same root locus to share the same gain value.

Case 2: $\sum^n_{j=1} \beta_j < \sum^m_{l=1} \gamma_l$.

In this case, the points  at infinity in $\mathbb{C}\cup\{\infty\}$ emit the root locus to finite zeros of (1.1). Each finite zero receives a root locus of (1.1) with degree $\arg(a_f+ib_f)$. Thus, (1.1) receives $\sum^m_{l=1} \gamma_l$ root loci, each with degree $\arg(a_f+ib_f)$. By replacing $\sum^n_{j=1} \beta_j$ with $\sum^m_{l=1} \gamma_l$ in the previous proof, we conclude that no two points on the same root locus can share the same gain value.
\end{proof}

\begin{theorem}
For each root locus of (1.1), the gain $K$ of the root loci is strictly monotonic increasing from $0$ at the poles of (1.1) to $+\infty$ at the zeros of (1.1).
\end{theorem}
\begin{proof}
If the gains $|K(s)|$ are not monotonic on the root locus of (1.1), two points $s_1$ and $s_2$ exist on the same root locus such that $|K(s_1)|=|K(s_2)|$. This would contradict Lemma 1.37. By Lemma 1.4, the gain must be strictly monotonic. We obtain Theorem 1.38.
\end{proof}

When the rational function $RF(s)$ takes a finite value at the point at infinity in the extended complex plane, Lemma 1.37 and Theorem 1.38 do not apply.
In such cases, the point at infinity must be excluded. This is an exceptional scenario, but it does not affect the study of the Gauss-Fourier conjecture. For this reason, we have not explicitly noted this issue in Lemma 1.37 and Theorem 1.38.

By Theorem 1.38, if a root locus intersects with itself, the monotonicity of the gains of the root locus cannot hold on the same root locus. This leads to the following Theorem 1.39.

\begin{theorem}
No root locus intersects itself.
\end{theorem}

\begin{theorem}
Let $\gamma_0>1$. If $\gamma_0$ root loci of (1.1) intersect at a point $s_0$, then $s_0$ is a critical point of $RF$ with multiplicity $\gamma_0-1$.
\end{theorem}
\begin{proof}
If the point $s_0$ is a breakaway point of the root loci of (1.1), by Theorem 1.25, $s_0$ satisfies (1.4). Therefore, $TT1(s_0)=0$. $TT2(s)=(\gamma_0(s-s_0)^{\gamma_0-1}g(s)+(s-s_0)^{\gamma_0}g^{'}(s))
\prod^{m}_{l=1}(s_0-z_{l})^{\gamma_{l}}=(s-s_0)^{\gamma_0-1}(\gamma_0g(s)+(s-s_0)g^{'}(s))
\prod^{m}_{l=1}(s_0-z_{l})^{\gamma_{l}}$. $\prod^{m}_{l=1}(s_0-z_{l})^{\gamma_{l}}\neq 0$. $\gamma_0g(s_0)+(s_0-s_0)g^{'}(s)=\gamma_0g(s_0)\neq 0$. By Lemma 1.21, $TT1(s)=TT2(s)$. $TT2(s_0)=0$. Therefore, the point $s_0$ is a critical point of $\gamma_0-1$ multiplicity of $K(s)$. By Theorem 1.26, $s_0$ is also a critical point of $\gamma_0-1$ multiplicity of $RF$.
\end{proof}

For a rational function, if its root loci intersect at a breakaway point, and the root locus that leaves from this breakaway point re-intersects at the same breakaway point, a situation arises where the root locus intersects itself. According to Theorem 1.39, such a situation cannot occur. Therefore, we have the following theorem:

\begin{theorem}
For a rational function, if its root loci of (1.1) intersect at a breakaway point, the root locus leaving this breakaway point cannot re-intersect it.
\end{theorem}

\section{Continuity of Root Loci}
In prior literature, the continuity of the root locus has often been treated as self-evident, with no rigorous proof provided. However, proving the continuity of the root locus is not a trivial matter and cannot be overlooked. It is only under the assumption that the root locus is fully determined by the phase angle condition equations that we can rigorously establish its continuity. The continuity of the root locus is a fundamental property in root locus theory, and providing a rigorous proof of this property is essential.

Prior to our work, there were no established results regarding the relationship between the argument of complex functions of one variable and their critical points. In this paper, we present Theorem 2.7, which addresses this relationship. While this result may appear elementary and seemingly of limited practical use, it plays a crucial role in our proof of the continuity of the root locus. Without this result, the continuity of the root locus cannot be established. By combining this elementary result with the implicit function theorem, we provide a rigorous proof of the continuity of the root locus.

Except the implicit function Theorem 2.12 and 2.13, all lemmas and theorems in this section are original to the authors and do not rely on prior literature.

Notation  3:

Let $\Delta_1$ be a subset of $\Delta$, defined as $\Delta_1=\{s=(\sigma+it)\in\Delta$, except for the infinite points in $\mathbb{C}\cup\{\infty\}$ $\}$. We further define a subset $\Delta_{u}$ of $\Delta_1$, as $\Delta_{u}=\{s=(\sigma+it)\in\Delta_1$, where $u(\sigma,t)\neq0\}$. We define another subset $\Delta_{v}$ of $\Delta_1$ as $\Delta_{v}=\{s=(\sigma+it)\in\Delta_1$, where $v(\sigma,t)\neq0\}$. We define another subset $\Delta_{vu}$ of $\Delta_{v}$ as $\Delta_{vu}=\{s=(\sigma+it)\in\Delta_{v}$, where $u(\sigma,t)=0\}$. Let $\Delta_{d}=\{s=(\sigma+it)\in\Delta_1$. $s$ is not the finite critical point of $RF(s)\}$.

Definition 2.11, Theorem 2.12 and Theorem 2.13 are classic results. For the symbols in them, we do not provide a notation.
The parameters $\varepsilon$ and $\delta$ in Lemma 2.14 are the same as those in the standard $\varepsilon$-$\delta$ definition of limits in mathematical analysis. The proof of Lemma 2.16 and Lemma 2.18 is based principally on $\varepsilon$-$\delta$ arguments.  $\varkappa_1$ and  $\varkappa_2$ are two sets defined in this section.

In this section, we present the expressions of the argument\cite{Stein} of the rational function $RF$ in $\mathbb{C}\cup\{\infty\}$. The argument $\varphi(\sigma,t)$ is defined as follows:

(1). When $u(\sigma,t)>0$, $v(\sigma,t)\ge0$. $\varphi(\sigma,t)=arg(\frac{v(\sigma,t)}{u(\sigma,t)})=\arctan(\frac{v(\sigma,t)}{u(\sigma,t)})$.

(2). When $u(\sigma,t)<0$, $v(\sigma,t)\ge0$. $\varphi(\sigma,t)=arg(\frac{v(\sigma,t)}{u(\sigma,t)})=\pi-\arctan(\frac{v(\sigma,t)}{-u(\sigma,t)})$.

(3). When $u(\sigma,t)<0$, $v(\sigma,t)<0$. $\varphi(\sigma,t)=arg(\frac{v(\sigma,t)}{u(\sigma,t)})=\pi+\arctan(\frac{v(\sigma,t)}{u(\sigma,t)})$.

(4). When $u(\sigma,t)>0$, $v(\sigma,t)<0$. $\varphi(\sigma,t)=arg(\frac{v(\sigma,t)}{u(\sigma,t)})=-\arctan(\frac{-v(\sigma,t)}{u(\sigma,t)})$.

Lemma 2.1 is proved via partial derivative calculations.

\begin{lemma}
Let $s=(\sigma,t)\in\Delta_u$. Then, $\frac{\partial{\varphi}}{\partial{\sigma}}(\sigma,t)=\frac{v^{'}_{\sigma}(\sigma,t)u(\sigma,t)-u^{'}_{\sigma}(\sigma,t)v(\sigma,t)}{u^2(\sigma,t)+v^2(\sigma,t)}$,
$\frac{\partial{\varphi}}{\partial{t}}(\sigma,t)=\frac{v^{'}_{t}(\sigma,t)u(\sigma,t)-u^{'}_{t}(\sigma,t)v(\sigma,t)}{u^2(\sigma,t)+v^2(\sigma,t)}$.
\end{lemma}

Let $RF_3(s)=iRF(s)$, so that $RF_3(s)=iu(\sigma,t)-v(\sigma,t)$. By repeating the proof of Lemma 2.1, we obtain Lemma 2.2.

\begin{lemma}
Let $s=(\sigma,t)\in\Delta_v$. Then, $\frac{\partial{\varphi_3}}{\partial{\sigma}}(\sigma,t)=
\frac{v^{'}_{\sigma}(\sigma,t)u(\sigma,t)-u^{'}_{\sigma}(\sigma,t)v(\sigma,t)}{u^2(\sigma,t)+v^2(\sigma,t)}$,
$\frac{\partial{\varphi_3}}{\partial{t}}(\sigma,t)=
\frac{v^{'}_{t}(\sigma,t)u(\sigma,t)-u^{'}_{t}(\sigma,t)v(\sigma,t)}{u^2(\sigma,t)+v^2(\sigma,t)}$
\end{lemma}

Using Lemma 2.2, we obtain Lemma 2.3.

\begin{lemma}
Let $s=(\sigma,t)\in\Delta_{vu}$. Then,
$\frac{\partial{\varphi}}{\partial{\sigma}}(\sigma,t)=\frac{v^{'}_{\sigma}(\sigma,t)u(\sigma,t)-u^{'}_{\sigma}(\sigma,t)v(\sigma,t)}{u^2(\sigma,t)+v^2(\sigma,t)}$,
$\frac{\partial{\varphi}}{\partial{t}}(\sigma,t)=\frac{v^{'}_{t}(\sigma,t)u(\sigma,t)-u^{'}_{t}(\sigma,t)v(\sigma,t)}{u^2(\sigma,t)+v^2(\sigma,t)}$.
\end{lemma}

Note that $\Delta_u\cap\Delta_{vu}=\emptyset$, $\Delta_u\subset\Delta_1$, $\Delta_{vu}\subset\Delta_1$, and $\Delta_1=\Delta_u\cup\Delta_{vu}$.
Combining Lemma 2.1 and Lemma 2.3, we obtain Theorem 2.4.

\begin{theorem}
Let $s=(\sigma,t)\in\Delta_1$. Then, $\frac{\partial{\varphi}}{\partial{\sigma}}(\sigma,t)=\frac{v^{'}_{\sigma}(\sigma,t)u(\sigma,t)-u^{'}_{\sigma}(\sigma,t)v(\sigma,t)}{u^2(\sigma,t)+v^2(\sigma,t)}$,
$\frac{\partial{\varphi}}{\partial{t}}(\sigma,t)=\frac{v^{'}_{t}(\sigma,t)u(\sigma,t)-u^{'}_{t}(\sigma,t)v(\sigma,t)}{u^2(\sigma,t)+v^2(\sigma,t)}$.
\end{theorem}
In many textbooks on functions of one complex variable, the following results are established: $\frac{dRF(s)}{ds}=\frac{\partial{u(\sigma,t)}}{\partial{\sigma}}+i\frac{\partial{v(\sigma,t)}}{\partial{\sigma}}=\frac{\partial{u(\sigma,t)}}{\partial{t}}-i\frac{\partial{v(\sigma,t)}}{\partial{t}}$. $\frac{\partial{u(\sigma,t)}}{\partial{\sigma}}=\frac{\partial{v(\sigma,t)}}{\partial{t}}$, $\frac{\partial{v(\sigma,t)}}{\partial{\sigma}}=-\frac{\partial{u(\sigma,t)}}{\partial{t}}$.

\begin{theorem}
For any point $s=(\sigma,t)\in\Delta_1$, we have $\frac{\partial{\varphi}}{\partial{t}}(\sigma,t)=Re({\frac{\frac{dRF(s)}{ds}}{RF(s)}})$, $\frac{\partial{\varphi}}{\partial{\sigma}}(\sigma,t)=Im({\frac{\frac{dRF(s)}{ds}}{RF(s)}}).$
\end{theorem}
\begin{proof}
$\frac{\frac{dRF(s)}{ds}}{RF(s)}=\frac{\frac{\partial{u(\sigma,t)}}{\partial{\sigma}}+i\frac{\partial{v(\sigma,t)}}{\partial{\sigma}}}{u(\sigma,t)+iv(\sigma,t)}
=\frac{\frac{\partial{u(\sigma,t)}}{\partial{\sigma}}u(\sigma,t)+\frac{\partial{v(\sigma,t)}}{\partial{\sigma}}v(\sigma,t)+i(\frac{\partial{v(\sigma,t)}}{\partial{\sigma}}u(\sigma,t)-\frac{\partial{u(\sigma,t)}}{\partial{\sigma}}v(\sigma,t))}{u(\sigma,t)^2+v(\sigma,t)^2}
=\frac{\frac{\partial{v(\sigma,t)}}{\partial{t}}u(\sigma,t)-\frac{\partial{u(\sigma,t)}}{\partial{t}}v(\sigma,t)}{u(\sigma,t)^2+v(\sigma,t)^2}+i\frac{\frac{\partial{v(\sigma,t)}}{\partial{\sigma}}u(\sigma,t)-\frac{\partial{u(\sigma,t)}}{\partial{\sigma}}v(\sigma,t)}{u(\sigma,t)^2+v(\sigma,t)^2}
$.

By Theorem 2.4, we derive Theorem 2.5.
\end{proof}

By Theorem 2.5, we derive Theorem 2.6.

\begin{theorem}
For any point $s=(\sigma,t)\in\Delta_1$:

1. If $\frac{\partial{\varphi}}{\partial{t}}(\sigma,t)=0$, then $Re({\frac{\frac{dRF(s)}{ds}}{RF(s)}})=0$.

2. If $\frac{\partial{\varphi}}{\partial{\sigma}}(\sigma,t)=0$, then $Im({\frac{\frac{dRF(s)}{ds}}{RF(s)}})=0$.

Conversely:

1. If $Re({\frac{\frac{dRF(s)}{ds}}{RF(s)}})=0$, then $\frac{\partial{\varphi}}{\partial{t}}(\sigma,t)=0$.

2. If $Im({\frac{\frac{dRF(s)}{ds}}{RF(s)}})=0$, then $\frac{\partial{\varphi}}{\partial{\sigma}}(\sigma,t)=0$.
\end{theorem}

\begin{theorem}
A point $s_0\in\Delta_1$ is a finite zero of $RF^{'}(s)$
if and only if $\frac{\partial{\varphi}}{\partial{\sigma}}(\sigma,t)\mid{_{s=s_0}}=0$
and $\frac{\partial{\varphi}}{\partial{t}}(\sigma,t)\mid{_{s=s_0}}=0$ hold simultaneously.
\end{theorem}
\begin{proof}
If $s_0\in\Delta_1$ is a zero of $RF^{'}(s)$, then ${\frac{\frac{dRF(s)}{ds}}{RF(s)}}\mid{_{s=s_0}}=0$, except when $s_0$ is an infinite point in $\mathbb{C}\cup\{\infty\}$. Therefore, both $Re({\frac{\frac{dRF(s)}{ds}}{RF(s)}})\mid{_{s=s_0}}=0$ and $Im({\frac{\frac{dRF(s)}{ds}}{RF(s)}})\mid{_{s=s_0}}=0$ hold, by Theorem 2.5, implying that $\frac{\partial{\varphi}}{\partial{t}}(\sigma,t)\mid{_{s=s_0}}=0$ and $\frac{\partial{\varphi}}{\partial{\sigma}}(\sigma,t)\mid{_{s=s_0}}=0$ hold simultaneously.

Conversely, if $s_0\in\Delta$ and both $\frac{\partial{\varphi}}{\partial{t}}(\sigma,t)\mid{_{s=s_0}}=0$ and $\frac{\partial{\varphi}}{\partial{\sigma}}(\sigma,t)\mid{_{s=s_0}}=0$ hold, by Theorem 2.5, then $Re({\frac{\frac{dRF(s)}{ds}}{RF(s)}})\mid{_{s=s_0}}=0$, and $Im({\frac{\frac{dRF(s)}{ds}}{RF(s)}})\mid{_{s=s_0}}=0$ hold simultaneously, so, the equation ${\frac{\frac{dRF(s)}{ds}}{RF(s)}}\mid{_{s=s_0}}=0$ holds.
\end{proof}

Since the argument of the zeros and poles of $RF(s)$ cannot be determined, we exclude these points in this paper. Additionally, infinite points in $\mathbb{C}\cup\{\infty\}$ may also be zeros or poles (e.g., $\frac{1}{s}$ or $s$), so we exclude these points as well.

\begin{theorem}
For $s_0=(\sigma_0+it_0)\in\Delta_d$:

If $\frac{\partial{\varphi}}{\partial{\sigma}}(\sigma,t)\mid{_{s=s_0}}=0$, then $\frac{\partial{\varphi}}{\partial{t}}(\sigma,t)\mid{_{s=s_0}}\neq 0$.

If $\frac{\partial{\varphi}}{\partial{t}}(\sigma,t)\mid{_{s=s_0}}=0$, then $\frac{\partial{\varphi}}{\partial{\sigma}}(\sigma,t)\mid{_{s=s_0}}\neq0$.
\end{theorem}
\begin{proof}
By Theorem 2.7, if both $\frac{\partial{\varphi}}{\partial{t}}(\sigma,t)=0$ and $\frac{\partial{\varphi}}{\partial{\sigma}}(\sigma,t)=0$ hold simultaneously, then the point $s$ is necessarily a finite critical point of function $RF(s)$. Excluding the finite zeros and poles of $RF(s)$, the infinite points in $\mathbb{C}\cup\{\infty\}$, and the finite critical points of $RF(s)$, the partial derivatives of the phase angle of $RF(s)$ with respect to the variables $t$ and $\sigma$ cannot both equal 0 at any other finite point in $\mathbb{C}\cup\{\infty\}$.
\end{proof}

\begin{remark}
At any point $s=(\sigma,t)\in\Delta_{d}$, the partial derivative of the phase angle of $RF(s)$ with respect to $\sigma$ and $t$ cannot both simultaneously equal 0.
\end{remark}

When $2q\pi+\arg{(a+ib)}$ is a constant, $\varphi(\sigma,t)=2q\pi+\arg{(a+ib)}$ is a function equation involving two real variables, representing the phase angle condition equation of (1.1). According to Theorem 1.8, excluding zeros and poles of (1.1), if this function equation has roots in $\mathbb{C}\cup\{\infty\}$, then (1.1) also has roots in $\mathbb{C}\cup\{\infty\}$. If the roots of this function equation form a root locus in $\mathbb{C}\cup\{\infty\}$, then the roots of (1.1) also form a root locus in $\mathbb{C}\cup\{\infty\}$.

\begin{lemma}
The root locus of the arbitrary $2q\pi+\alpha$ degree of (1.1) is identical to the locus generated by points satisfying the implicit function equation $\varphi(\sigma,t)=2q\pi+\arg{(a+ib)}$ in the two-dimensional real plane. Their properties such as continuity are equivalent.
\end{lemma}
\begin{proof}
When $2q\pi+\alpha$ is constant, the root locus of the arbitrary $2q\pi+\alpha$ degree of (1.1) is a curve composed of roots of (1.1). Theorem 1.8 establishes that this root locus can also be sufficiently and necessarily determined by a real function equation of two real variables. Specifically, when $2q\pi+\alpha$ is constant, the root locus is sufficiently and necessarily determined by the implicit function equation $ \varphi(\sigma,t)=2q\pi+\arg{(a+ib)}$.

Thus, Theorem 1.8 essentially proves that if the extended complex plane $\mathbb{C}\cup\{\infty\}$ is treated as a two-dimensional real plane with coordinates $\sigma$ and $t$, then the root locus of the arbitrary $2q\pi+\alpha$ degree of (1.1) consists of points $(\sigma, t)$ satisfying the implicit function equation  $ \varphi(\sigma,t)=2q\pi+\arg{(a+ib)}$ in the two-dimensional real plane.

Consequently, the roots of (1.1) in $\mathbb{C}\cup\{\infty\}$ and the points $(\sigma, t)$ satisfying  $ \varphi(\sigma,t)=2q\pi+\arg{(a+ib)}$  in the two-dimensional real plane are identical. If $\mathbb{C}\cup\{\infty\}$ and the two-dimensional real plane are considered equivalent, when the root locus of the arbitrary $2q\pi+\alpha$ degree of (1.1) constitutes the locus in $\mathbb{C}\cup\{\infty\}$, then the points that are determined by two real coordinates satisfying  $ \varphi(\sigma,t)=2q\pi+\arg{(a+ib)}$ also constitute the locus in the two-dimensional real plane, and the two kinds of loci are fully identical.

Since the two types of root loci are identical, their properties, such as existence and continuity, are equivalent.
\end{proof}

By Lemma 2.10, problems involving functions with one complex variable are transformed into problems involving functions with real variables. Therefore, we can use theorems of real variable functions and the definition of continuity to prove the existence and continuity of the  loci that are determined by two real variables $\sigma$ and $t$ satisfying  $ \varphi(\sigma,t)=2q\pi+\arg{(a+ib)}$ in the two-dimensional real plane.

To prove the existence and continuity of the roots of the implicit function equation $\varphi(\sigma,t)=2q\pi+\arg{(a+ib)}$, we invoke the implicit function theorem and related concepts\cite{Stein}.

\begin{definition}
Let $f:R^{n+m}\longrightarrow R^{m}$ be continuously differentiable. We consider $R^{n+m}$ as the Cartesian product $R^{n}\times R^{m}$, and we write a point of this product as
$(x,y)=(x_1, \cdots, x_n, y_1, \cdots, y_m) $. Starting from the given function $f$, our goal is to construct a function $\mu:R^{n}\longrightarrow R^{m}$ whose graph $(x,\mu(x))$ is precisely the set of all $(x,y)$ such that $f(x,y)=0$.
\end{definition}

As noted above, this may not always be possible. Therefore, we fix a point $(\zeta,\eta)=(\zeta_1, \cdots, \zeta_n, \eta_1, \cdots, \eta_m)$ satisfying $f(\zeta,\eta)=0$, and seek a function $\mu$ that works near $(\zeta,\eta)$. Specifically, we require an open set $U$ of $R^{n}$ containing $\zeta$, an open set $V$ of $R^{m}$ containing $\eta$, and a function $\mu:U\longrightarrow V$ such that the graph of $\mu$ satisfies the relation $f(x,y)=0$ on $U\times V$, and no other points within $U\times V$ do so. In the symbols, $((x,\mu(x))|x\in U)$ = $((x,y)\in U\times V|f(x,y)=0)$.

To state the implicit function theorem, we need the Jacobian matrix of $f$, which is the matrix of the partial derivatives of $f$. Abbreviating $(\zeta_1, \cdots, \zeta_n, \eta_1, \cdots, \eta_m)$ to $(\zeta,\eta)$, the Jacobian matrix is

$(Df)(\zeta,\eta)=[
\begin{array}{ccc}
\frac{\partial{f_1}}{\partial{x_1}}(\zeta,\eta)&\cdots&\frac{\partial{f_1}}{\partial{x_n}}(\zeta,\eta)\\
\vdots&\ddots&\vdots\\
\frac{\partial{f_m}}{\partial{x_1}}(\zeta,\eta)&\cdots&\frac{\partial{f_m}}{\partial{x_n}}(\zeta,\eta)
\end{array}
|\begin{array}{ccc}
\frac{\partial{f_1}}{\partial{y_1}}(\zeta,\eta)&\cdots&\frac{\partial{f_1}}{\partial{y_m}}(\zeta,\eta)\\
\vdots&\ddots&\vdots\\
\frac{\partial{f_m}}{\partial{y_1}}(\zeta,\eta)&\cdots&\frac{\partial{f_m}}{\partial{y_m}}(\zeta,\eta)
\end{array}]$
$=[X|Y]$.

Where $X$ is the matrix of partial derivatives with respect to the variables $x_l$ and $Y$ is the matrix of partial derivatives with respect to the variables $y_j$. The implicit function theorem states that if $Y$ is invertible, then $U$, $V$, and $\mu$ exist as desired. Combining all the hypotheses, we obtain the following statement.

The implicit function theorem with respect to $x$ is as follows:

\begin{theorem}
Let $f:R^{n+m}\longrightarrow R^{m}$ be continuously differentiable, and let $R^{n+m}$ have coordinates $(x,y)$. Fix a point $(\zeta,\eta)=(\zeta_1, \cdots, \zeta_n, \eta_1, \cdots, \eta_m)$ with $f(\zeta,\eta)=\chi$, where $\chi\in R^{m}$. If the Jacobian matrix $J_{f,y}(\zeta,\eta)=[(\partial f_l/\partial y_j)(\zeta,\eta)]$ is invertible, then there exists an open set $U$ containing $\zeta$, an open set $V$ containing $\eta$, and a unique continuously differentiable function $\mu:U\longrightarrow V$ such that $((x,\mu(x))|x\in U)$ = $((x,y)\in U\times V|f(x,y)=\chi)$.
\end{theorem}

The implicit function theorem with respect to $y$ is as follows:

\begin{theorem}
Let $f:R^{n+m}\longrightarrow R^{n}$ be continuously differentiable, and let $R^{n+m}$ have coordinates $(x,y)$. Fix a point $(\zeta,\eta)=(\zeta_1, \cdots, \zeta_n, \eta_1, \cdots, \eta_m)$ with $f(\zeta,\eta)=\chi$, where $\chi\in R^{n}$. If the Jacobian matrix $J_{f,x}(\zeta,\eta)=[(\partial f_l/\partial x_k)(\zeta,\eta)]$ is invertible, then there exists an open set $V$ containing $\eta$, an open set  $U$ containing $\zeta$, and a unique continuously differentiable function $\lambda:V\longrightarrow U$ such that  $((\lambda(y),y)|y\in V)$ = $((x,y)\in U\times V|f(x,y)=\chi)$.
\end{theorem}

The rational function $RF(s)$  forms the root locus equation:

\begin{equation}
K\frac{RF^{'}}{RF}=a+ib
\end{equation}

Let $RF(s)=\frac{p_n(s)}{p_d(s)}$.  In which, $p_n(s)$ is the numerator polynomial of $RF$ and $p_d(s)$ is the denominator polynomial of $RF$. $RF^{'}(s)=\frac{p_n^{'}(s)p_d(s)-p_d^{'}(s)p_n(s)}{(p_d(s))^2}$.

$\frac{RF^{'}(s)}{RF(s)}=\frac{p_n^{'}(s)p_d(s)-p_d^{'}(s)p_n(s)}{p_n(s)p_d(s)}$. By this result, the zeros of the root locus equation (2.1) are the breakaway points of the root locus equation (1.1). The poles of (2.1) are composed of two types of points. One type is the zeros of $RF(s)$, and the other type is the poles of $RF(s)$. Therefore, only on the zeros of (2.1), $\frac{\partial{\varphi}}{\partial{\sigma}}(\sigma,t)\mid{_{s=s_0}}=0$
and $\frac{\partial{\varphi}}{\partial{t}}(\sigma,t)\mid{_{s=s_0}}=0$ hold simultaneously.

$\frac{\partial{\varphi}}{\partial{t}}(\sigma,t)=Re({\frac{RF^{'}(s)}{RF(s)}})$, $\frac{\partial{\varphi}}{\partial{\sigma}}(\sigma,t)=Im({\frac{RF^{'}(s)}{RF(s)}})$.

If $\frac{\partial{\varphi}}{\partial{t}}(\sigma,t)=0$, then, $Re({\frac{RF^{'}(s)}{RF(s)}})=0$. If a point $s=\sigma+it$ in the complex plane satisfies $\frac{\partial{\varphi}}{\partial{t}}(\sigma,t)=0$, then the point $s=\sigma+it$ must lie on the root locus of $2q\pi+\frac{\pi}{2}$ and $2q\pi+\frac{3\pi}{2}$ of (2.1).

If $\frac{\partial{\varphi}}{\partial{\sigma}}(\sigma,t)=0$, then, $Im({\frac{RF^{'}(s)}{RF(s)}})=0$.
If a point $s=\sigma+it$ in the complex plane satisfies $\frac{\partial{\varphi}}{\partial{\sigma}}(\sigma,t)=0$, then the point $s=\sigma+it$ must distribute on the root loci of $2q\pi$ and $2q\pi+\pi$ degrees of (2.1).

In the complex plane, the root loci of the $2q\pi$, or $2q\pi+\pi$, or $2q\pi+\frac{\pi}{2}$ or $2q\pi+\frac{3\pi}{2}$ degree of (2.1), each has $\max{(m,n)}$ branches. The points that satisfy $\frac{\partial{\varphi}}{\partial{t}}(\sigma,t)=0$ or $\frac{\partial{\varphi}}{\partial{\sigma}}(\sigma,t)=0$ form $2\times \max{(m,n)}$  curves in the complex plane.

We establish two sets.
\begin{definition}
All points in the complex plane except those on the root loci of $2q\pi+\frac{\pi}{2}$ and $2q\pi+\frac{3\pi}{2}$ degrees of (2.1) constitute the set $\varkappa_1$.
\end{definition}

\begin{definition}
All points in the complex plane except those on the root loci of $2q\pi$ and $2q\pi+\pi$ degrees of (2.1) constitute the set $\varkappa_2$.
\end{definition}

For all points in the set $\varkappa_1$, $\frac{\partial{\varphi}} {\partial{t}}(\sigma,t) \neq 0$. For all points in the set $\varkappa_2$, $\frac{\partial{\varphi}}{\partial{\sigma}}(\sigma,t)\neq 0$.

\begin{lemma}
Let $s_0=(\sigma_0+it_0)\in\varkappa_1$, where $s_0$ satisfies $\varphi(\sigma_0,t_0)=2q\pi+\arg{(a_0+ib_0)}$. There exists an open set $U$ containing $\sigma_0$, an open set $V$ containing $t_0$, and a unique continuously differentiable function $\mu:U\longrightarrow V$ such that $((\sigma,\mu(\sigma))|\sigma\in U)$ = $((\sigma,t)\in U\times V|\varphi(\sigma,t)=2q\pi+\arg{(a_0 +ib_0)})$.
\end{lemma}
\begin{proof}
The function $\varphi(\sigma,t)$, $\varphi(\sigma,t):R^2\longrightarrow R$ is continuously differentiable, and $R^2$ has two coordinates $(\sigma,t)$.

$\varphi(\sigma,t)$ is an implicit function contains two variables. Thus, its Jacobian matrix $Y$ is a $1\times 1$ matrix. In its Jacobian matrix $Y$, there is only one element $\frac{\partial{\varphi}}{\partial{t}}(\sigma,t)$. When $\frac{\partial{\varphi}}{\partial{t}}(\sigma,t)\mid{_{s=s_0}}\neq 0$, $Y$ is invertible.

An arbitrary point $s_0=(\sigma_0+it_0)\in\varkappa_1$ satisfies the implicit function equation $ \varphi(\sigma_0,t_0)=2q\pi+\arg{(a_0+ib_0)}$, where $\arg{(a_0+ib_0)}\in R$.

If the point $s_0=(\sigma_0+it_0)\in\varkappa_1$, $\frac{\partial{\varphi}}{\partial{t}}(\sigma,t)\mid{_{s=s_0}}\neq 0$, $Y$ is invertible, $s_0$ satisfies all conditions of Theorem 2.12. By Theorem 2.12, there exists an open set $U$ containing $\sigma_0$, an open set $V$ containing $t_0$, and a unique continuously differentiable function $\mu:U\longrightarrow V$ such that $((\sigma,\mu(\sigma))|\sigma\in U)$ = $((\sigma,t)\in U\times V|\varphi(\sigma,t)=2q\pi+\arg{(a_0+ib_0)})$.
\end{proof}

\begin{lemma}
For $s_0=(\sigma_0+it_0)\in\varkappa_1$, the root locus of an arbitrary $2q\pi+\alpha$ degree of (1.1) is continuous at $s_0$.
\end{lemma}
\begin{proof}
By Lemma 2.16, $t=\mu(\sigma)$ is a continuous function on an open set $U$. For any $\varepsilon>0$, there surely exists $\delta_{\sigma1}<\frac{\varepsilon}{\sqrt{2}}$ such that when $|\sigma-\sigma_0|<\delta_{\sigma1}<\frac{\varepsilon}{\sqrt{2}}$. Then, $| t-t_0 |=|\mu(\sigma)-\mu(\sigma_0) |<\frac{\varepsilon}{\sqrt{2}}$. Therefore, for any given $\varepsilon>0$, let $\delta_{\sigma2}=\delta_{\sigma1}<\frac{\varepsilon}{\sqrt{2}}$, when $|\sigma-\sigma_0|<\delta_{\sigma1}<\frac{\varepsilon}{\sqrt{2}}$, then, $| t-t_0 |=|\mu(\sigma)-\mu(\sigma_0) |<\frac{\varepsilon}{\sqrt{2}}$, and obtain: $|\sigma+it-\sigma_0-it_0 |=|\sigma+i\mu(\sigma)-\sigma_0-i\mu(\sigma_0) |=\sqrt{(\sigma-\sigma_0)^2+(\mu(\sigma)-\mu(\sigma_0))^2}<\sqrt{\frac{\varepsilon^2}{2}+\frac{\varepsilon^2}{2}}=\varepsilon$.

Thus, the root locus of the arbitrary $2q\pi+\alpha$ degree of (1.1) is continuous at $(\sigma_0, t_0)$.
\end{proof}

\begin{lemma}
Let $s_0=(\sigma_0+it_0)\in\varkappa_2$, where $s_0$ satisfies $\varphi(\sigma_0,t_0)=2q\pi+\arg{(a_0+ib_0)}$. There  exists an open set $V$ containing $t_0$, an open set $U$ containing $\sigma_0$, and a unique continuously differentiable function $\lambda:V\longrightarrow U$ such that $((\lambda(t),t)|t\in V)$ = $((\sigma,t)\in U\times V|\varphi(\sigma,t)=2q\pi+\arg{(a_0+ib_0)})$.
\end{lemma}
\begin{proof}
The function $\varphi(\sigma,t)$, $\varphi(\sigma,t):R^2\longrightarrow R$ is continuously differentiable, and $R^2$ has two coordinates $(\sigma,t)$.

$\varphi(\sigma,t)$ is an implicit function contains two variables. Thus, its Jacobian matrix $X$ is a $1\times 1$ matrix. In its Jacobian matrix $X$, there is only one element $\frac{\partial{\varphi}}{\partial{\sigma}}(\sigma,t)$. When $\frac{\partial{\varphi}}{\partial{\sigma}}(\sigma,t)\mid{_{s=s_0}}\neq 0$, $X$ is invertible.

An arbitrary point $s_0=(\sigma_0+it_0)\in\varkappa_2$ satisfies the implicit function equation $ \varphi(\sigma_0,t_0)=2q\pi+\arg{(a_0+ib_0)}$, where $\arg{(a_0+ib_0)}\in R$.

If the point $s_0=(\sigma_0+it_0)\in\varkappa_2$, $\frac{\partial{\varphi}}{\partial{\sigma}}(\sigma,t)\mid{_{s=s_0}}\neq0$, $X$ is invertible, $s_0$ satisfies the condition of Theorem 2.13. By Theorem 2.13,
there exists an open set $V$ containing $t_0$, an open set $U$ containing $\sigma_0$, and a unique continuously differentiable function $\lambda:V\longrightarrow U$ such that $((\lambda(t),t)|t\in V)$ =
$((\sigma,t)\in U\times V|\varphi(\sigma,t)=2q\pi+\arg{(b_0/a_0)})$.
\end{proof}

\begin{lemma}
For $s_0=(\sigma_0+it_0)\in\varkappa_2$, the root locus of an arbitrary $2q\pi+\alpha$ degree of (1.1) is continuous at $s_0$.
\end{lemma}
\begin{proof}
by Lemma 2.18, $\sigma=\lambda(t)$ is a continuous function on an open set $V$. For any $\varepsilon>0$, there surely exists $\delta_{\sigma1}<\frac{\varepsilon}{\sqrt{2}}$, when $|t-t_0|<\delta_{\sigma1}<\frac{\varepsilon}{\sqrt{2}}$, then, $| \sigma-\sigma_0 |=|\lambda(t)-\lambda(t_0) |<\frac{\varepsilon}{\sqrt{2}}$  is true. Therefore, for any given $\varepsilon>0$, let $\delta_{\sigma2}=\delta_{\sigma1}<\frac{\varepsilon}{\sqrt{2}}$, when $|t-t_0|<\delta_{\sigma1}<\frac{\varepsilon}{\sqrt{2}}$, then, $| \sigma-\sigma_0 |=|\lambda(t)-\lambda(t_0) |<\frac{\varepsilon}{\sqrt{2}}$, and obtain:  $|\sigma+it-\sigma_0-it_0 |=|\lambda(t)+it-\lambda(t_0)-it_0|=\sqrt{(t-t_0)^2+(\lambda(t)-\lambda(t_0))^2}<\sqrt{\frac{\varepsilon^2}{2}+\frac{\varepsilon^2}{2}}=\varepsilon$

Thus, the root locus of the arbitrary $2q\pi+\alpha$ degree of (1.1) is continuous at $(\sigma_0, t_0)$.
\end{proof}
According to the definitions of the three sets $\Delta_{d}$, $\varkappa_1$ and $\varkappa_2$, both $\varkappa_1$ and $\varkappa_2$ exclude the zeros, poles and breakaway points of (1.1). These three types of points are those excluded by both sets. All other points in the complex plane belong to at least one of the two sets.

Since \(\Delta_d\) excludes critical points, and \(\varkappa_1, \varkappa_2\) cover points where \(\partial \varphi / \partial t \neq 0\) or \(\partial \varphi / \partial \sigma \neq 0\), their union is \(\Delta_d\).  Therefore, there is a relationship among these three sets. $\Delta_{d}=\varkappa_1\cup\varkappa_2$.

Since $s_0=(\sigma_0, t_0)\in\Delta_d$ is an arbitrary point. On the basis that the root locus of the arbitrary $2q\pi+\alpha$ degree of (1.1) is continuous on the point $(\sigma_0, t_0)$, the root locus of the arbitrary $2q\pi+\alpha$ degree of (1.1) cannot have any break points excluding critical points of $RF(s)$ and zeros and poles of (1.1). When $s_0=(\sigma_0+it_0)\in\Delta_d$, the root locus of an arbitrary $2q\pi+\alpha$ degree of (1.1) must be continuous. So, we have:
\begin{lemma}
For $s_0=(\sigma_0+it_0)\in\Delta_{d}$, the root locus of an arbitrary $2q\pi+\alpha$ degree of (1.1) is continuous.
\end{lemma}

Owing to the limitations of the implicit function theorem, the zeros, poles, and critical points of the rational function $RF(s)$ are excluded from the continuity result of the root locus. The root loci are defined in $\mathbb{C} \cup \{\infty\}$. However, the continuity of the root locus is proven via the phase angle function, which is a real-valued function of two real variables.

In the previous results, except for three categories of points (zeros, poles, and critical points) that are discretely distributed in $\mathbb{C}\cup\{\infty\}$, each root locus is continuous. The proof of Lemma 2.17 and Lemma 2.19 relies on the definition of continuity. Since the root loci are continuous around these three categories of points, we can use theorems about the continuity of real-valued functions of real variables to prove that the root loci are indeed continuous at the critical points of $RF(s)$. For the zeros and poles of (1.1), the root loci are surely left-continuous and right-continuous. The proof method is the same as that of Lemma 2.17 and Lemma 2.19, so we omit the proofs and provide the results directly.

Using the same method outlined above, we can also prove the continuity of the root locus at the infinite points in $\mathbb{C} \cup \{\infty\}$. We state this result without repeating the proof.

\begin{theorem}
The root locus of an arbitrary $2q\pi+\alpha$ degree of (1.1) is continuous everywhere in $\mathbb{C}\cup\{\infty\}$.
\end{theorem}

\section{Proofs for the Gauss-Fourier Conjecture}
All lemmas and theorems in this section are original to the authors and do not rely on prior literature.

The general real polynomial $p(s)$ of degree $n$ is given by: $p(s)=a_0s^n+a_1s^{n-1}+\cdots+a_{n-1}s+a_n$.
The real polynomial $p(s)$ forms the root locus equation:
\begin{equation}
K\frac{1}{p(s)}=a+ib.
\end{equation}

Each pole of (3.1) is the zero of $p(s)$. Each zero of $p(s)$ is the pole of (3.1).
In the following proof, the concept of root locus extension is essential. Since the gain values of the root locus of (1.1) increase strictly monotonically from $0$ at the poles to positive infinity at the zeros, the root locus extension of (1.1) implies that the root locus moves from points with small gains to points with large gains. Alternatively, the root locus extension means that a root locus moves from a point closer to a pole to a point closer to zero.

By $\frac{1}{p(s)}$, all infinite points in $\mathbb{C}\cup\{\infty\}$ are zeros of (3.1), which can only receive the root loci but cannot emit the root loci. In $\mathbb{C}\cup\{\infty\}$, which includes infinite points, there are no other poles that can emit the root loci anymore. Therefore, we have:

\begin{lemma}
In $\mathbb{C}\cup\{\infty\}$, only zeros of the real polynomial $p(s)$ can emit the root loci of (3.1). Each infinite point in $\mathbb{C}\cup\{\infty\}$ can only receive the root locus of (3.1).
\end{lemma}

Notation 4:
Let the real polynomial $p(s)$ have $r$ distinct real zeros. Denote a real zero of $p(s)$ by $Q_{j1}$, where $j1=1,\cdots,r$. $Q_{1}, \cdots, Q_{r}$. The root loci emitted by $Q_{j1}$ towards the real axis are represented as $L_{i,j1}$, with $i=1,2$. Specifically, $L_{1,j1}$ denotes the root locus emitted from the right side of $Q_{j1}$ towards the real axis, whereas $L_{2,j1}$ represents the root locus emitted from the left side of $Q_{j1}$ towards the real axis.

Let $Q_k$ be one of the real zeros of $p(s)$. $w_{i,k}$ denotes an interval of $Q_k$ on the real axis, where $1\le k\le r$. $w_{1,k}$ denotes the interval at right side of $Q_k$. $w_{2,k}$ denotes the interval at left side of $Q_k$. These real zeros $Q_{j1}$ of $p(s)$ divide the real axis into $r+1$ intervals denoted by $w_{j2}$, where $j2=1,\cdots,r+1$.

Let $s_k$ and $s_k^{'}$ denote a pair of non-real zeros of $p(s)$. The point at which the root locus $L_{i,k}$ departs from the real axis within interval $w_{i,k}$ is referred to as $t1_{i,k}$. Let $t2_{i,j3}$ represent a point where the root locus $L_{i,k}$ re-intersects the real axis in another interval $w_{j3}$, ensuring that it satisfies the condition $w_{j3}\neq w_{k,i}$. $w_{j3}$ is one of $w_{j2}$, but, two intervals $w_{i,k}$ and $w_{j3}$ are distinct. Let $L_{i,k}^{'}$ denote a root locus symmetric to $L_{i,k}$ with respect to the real axis.

Let $U_k$ and $U_k^{'}$ denote a pair of the root loci emitted from a pair of non-real zeros $s_k$ and $s_k^{'}$, respectively. Let $T_l$ and $T_l^{'}$ denote two distinct root loci emitted directly into $\mathbb{C}^{*}$ by the multiple real zero $Q_l$. Let $d_l$ denote the intersection point of $T_l$ and $T_l^{'}$ on the real axis. Let $T3_l$ denote another root locus emitted by $Q_l$, where $1\le l\le r$. Note that $U_k$ is symmetric to $U_k^{'}$ with respect to the real axis, and $T_l$ is symmetric to $T_l^{'}$ with respect to the real axis.

When $p(s)$ has no real zeros, there is no need to consider the influence of root loci emitted by real zeros of $p(s)$ on the generation of the second kind of zeros of $(p^{'}(s)/p(s))^{'}$. However, it is essential for Conjectures 0.1 and 0.2 to consider how these root loci emitted by the real zeros of $p(s)$ affect the generation of the second kind of zeros of $(p^{'}(s)/p(s))^{'}$; $p(s)$ must have at least one real zero, as addressed in Lemmas 3.3 to 3.21.

The real poles of (3.1) emit root loci of infinitely many degrees. This paper focuses exclusively on root loci of $2q\pi$ degree and $2q\pi+\pi$ degree. For a simple real pole of (3.1):

If it emits a root locus of $2q\pi$ degree to the real axis on its right side, it emits a root locus of $2q\pi+\pi$ degree to the real axis on its left side.

Conversely, if it emits a root locus of $2q\pi+\pi$ degree to the real axis on its right side, it emits a root locus of $2q\pi$ degree to the real axis on its left side.

For a multiple real pole of even degree of (3.1):

If it emits a root locus of $2q\pi$ degree to the real axis on its right side, it emits a root locus of $2q\pi$ degree to the real axis on its left side.

If it emits a root locus of $2q\pi+\pi$ degree to the real axis on its right side, it emits a root locus of $2q\pi+\pi$  degree to the real axis on its left side.

For a multiple real pole of odd degree of (3.1):

If it emits a root locus of $2q\pi$ degree to the real axis on its right side, it emits a root locus of $2q\pi+\pi$ degree to the real axis on its left side.

If it emits a root locus of $2q\pi+\pi$ degree to the real axis on its right side, it emits a root locus of $2q\pi$  degree to the real axis on its left side.

Let $Q_1$ denote the largest real zero of $p(s)$, and let $Q_r$ denote the smallest real zero of $p(s)$. The distribution rule of the root locus $L_{i,j1}$ is the same as the distribution rule of $Q_{j1}$. Under the condition that the root loci emitted by the non-real zeros of $p(s)$ do not intersect the interval $w_{i,j1}$ of $Q_{j1}$, the root loci of (3.1) have the following properties:

1. The infinite interval $w_{1,1}$ to the right of $Q_1$ is the complete root locus $L_{1,1}$ of $2q\pi$ degree emitted from $Q_1$.

2. The infinite interval $w_{2,r+1}$ to the left of $Q_r$ is the complete root locus $L_{2,r}$ of $2q\pi$ or $2q\pi+\pi$ degree emitted from $Q_r$.

3. For cases of the finite interval, where $r>1$:

For $j1=1,\cdots,r$, in the case where $i=1$, within the finite interval $w_{1,j1}$, there are two root loci: $L_{1,j1}$ emitted by $Q_{j1}$ and $L_{2,j1-1}$ emitted by $Q_{j1-1}$. These two root loci extend towards each other. They must intersect at a point within the finite interval $w_{1,j1}$. After their intersection occurs, both $L_{1,j1}$ and $L_{2,j1-1}$ leave the real axis from the intersection point and enter $\mathbb{C}^{*}$.

In the case where $i=2$, within the finite interval $w_{2,j1}$, there are two root loci: $L_{2,j1}$ emitted by $Q_{j1}$ and $L_{1,j1+1}$ emitted by $Q_{j1+1}$. These two root loci extend towards each other. They must intersect at a point within the finite interval $w_{2,j1}$. After intersecting, $L_{2,j1}$ and $L_{1,j1+1}$ leave the real axis from the intersection point and enter $\mathbb{C}^{*}$.

Let $w_{i,k}$ denote two intervals: For $i=1$, $w_{1,k}$ is the interval between real zeros $Q_k$ and $Q_{k-1}$. For $i=2$, $w_{2,k}$ is the interval between real zeros $Q_{k+1}$ and $Q_k$. The real zero $Q_k$ emits two root loci $L_{1,k}$ and $L_{2,k}$ to two intervals $w_{1,k}$ and $w_{2,k}$ respectively.

If at least one pair of non-real zeros of $p(s)$ emits root loci that intersect the interval $w_{i,k}$ of $Q_k$:

Because $U_k^{'}$ and $U_k$ are symmetric with respect to the real axis, the two root loci cannot pass through the real axis, from one side of the real axis to the other side. Therefore, we can show the following result.

Let $s_k^{'}$ be in the upper half of $\mathbb{C}^{*}$, emitting the root locus $U_k^{'}$. $U_k^{'}$ is also in the upper half of $\mathbb{C}^{*}$.

Let $s_k$ be in the lower half of $\mathbb{C}^{*}$, emitting the root locus $U_k$. $U_k$ is also in the lower half of $\mathbb{C}^{*}$.

Assume that $U_k$ or $U_k^{'}$ is adjacent to the root locus $L_{i,k}$ emitted from $Q_k$ on the real axis. After $U_k$ and $U_k^{'}$ intersect at $c_{i,k}$ in $w_{i,k}$, they extend in opposite directions. Suppose that $U_k$ extends to $Q_k$. The real zero $Q_k$ emits the root locus $L_{i,k}$ towards $U_k$. The root loci $L_{i,k}$ and $U_k$ extend towards each other, intersect at point $b_{i,k}$, and then leave the real axis to enter $\mathbb{C}^{*}$.

The second kind of zeros of $(p^{'}(s)/p(s))^{'}$ are the real multiple zeros of $p^{'}(s)$. The real multiple zeros of $p^{'}(s)$ arise from the intersection of at least three root loci of (3.1) on the real axis. If the root loci originating from the real zeros of $p(s)$ leave the real axis and enter $\mathbb{C}^{*}$, these root loci cannot intersect on the real axis. Only the root loci emitted from the non-real zeros of $p(s)$ can intersect with other root loci on the real axis. Proving this result would establish that the generation of a real multiple zero of $p^{'}(s)$ must involve the non-real zeros of $p(s)$. Consequently, it follows that the second kind of zeros of $(p^{'}(s)/p(s))^{'}$  must be related to the non-real zeros of $p(s)$.

Therefore, we must prove that when a root locus $L_{i,k}$ (or $T_l$, $T_l^{'}$)  emitted from a real zero $Q_k$ of $p(s)$ leaves the real axis and enters $\mathbb{C}^{*}$, no such $L_{i,k}$ (or $T_l$, $T_l^{'}$) can intersect on the real axis.

If the zeros of $p(s)$ are distinct, then $p(s)$ itself is distinct. The root loci of (3.1) are likewise distinct. There are infinitely many polynomials $p(s)$ with distinct zeros, and consequently, infinitely many root loci of (3.1).

\begin{definition}
Infinitely many root loci $L_{i,k}$ (or $T_l$, $T_l^{'}$) emitted from all real zeros $Q_k$ (or $Q_l$) of $p(s)$. If $L_{i,k}$ (or $T_l$, $T_l^{'}$) leaves the real axis and enters $\mathbb{C}^{*}$,  subsequently, assume that $L_{i,k}$ (or $T_l$, $T_l^{'}$)  intersects the real axis at real breakaway points. The set $\Gamma$ consists of all such root loci $L_{i,k}$ (or $T_l$, $T_l^{'}$) possessing these properties.
\end{definition}

If $L_{i,k}$ (or $T_l$, $T_l^{'}$) leaves the real axis and enters $\mathbb{C}^{*}$, it does not intersect at a point on the real axis. $L_{i,k}$ (or $T_l$, $T_l^{'}$) directly reaches the point at the infinity in $\mathbb{C}^{*}$. The set $\Gamma$ does not contain these root loci $L_{i,k}$ (or $T_l$, $T_l^{'}$).

Owing to the many possible ways $L_{i,k}$ (or $T_l$, $T_l^{'}$) leaves and assume that $L_{i,k}$ (or $T_l$, $T_l^{'}$) re-intersects the real axis, we classify $L_{i,k}$ (or $T_l$, $T_l^{'}$) on the basis of its departure and intersection behavior with the real axis.

There are two possible ways for $L_{i,k}$ (or $T_l$, $T_l^{'}$) to leave the real axis:

1. Direct departure from a multiple real zero of $p(s)$:  $T_l$ or $T_l^{'}$ is not on the real axis and directly enters $\mathbb{C}^{*}$.

Since simple real zeros of $p(s)$ can only emit one root locus of either $2q\pi$ and $2q\pi+\pi$ degree to the real axis, root loci of $2q\pi$ and $2q\pi+\pi$ degree from simple real zeros cannot directly enter $\mathbb{C}^{*}$.

2. Departure via a real breakaway point: The real zero $Q_k$ of $p(s)$ first emits the root locus $L_{i,k}$ along the real axis, and then $L_{i,k}$ leaves the real axis at a real breakaway point (not real zero $Q_k$) and enters $\mathbb{C}^{*}$.

On the basis of these two behaviors, we partition $\Gamma$ into two subsets $\Gamma_1$ and $\Gamma_2$:

$\Gamma_1$: $\Gamma_1 \subseteq \Gamma$, where a real zero $Q_l$ of $p(s)$ with multiplicity emits two root loci $T_l$ and $T_l^{'}$, both departing directly from $Q_l$ and entering $\mathbb{C}^{*}$.

$\Gamma_2$: $\Gamma_2 \subseteq \Gamma$, where a real zero $Q_k$ of $p(s)$ first emits root locus $L_{i,k}$ along the real axis, then $L_{i,k}$ leaves the real axis at a real breakaway point and enters $\mathbb{C}^{*}$.

These subsets satisfy: $\Gamma=\Gamma_1\cup\Gamma_2$, $\Gamma_1\cap\Gamma_2=\phi$. In $\Gamma_1$, root loci $T_l$ and $T_l^{'}$ are directly emitted from multiple zeros  to $\mathbb{C}^{*}$. In $\Gamma_2$, the root locus $L_{i,k}$ first extends along the real axis, then leaves the real axis, enters $\mathbb{C}^{*}$.

For a root locus $L_{i,k}$ in the set $\Gamma_2$, when $L_{i,k}$ leaves the real axis and enters $\mathbb{C}^{*}$, assume that $L_{i,k}$ intersect either at $w_{i,k}$ or at $w_{j3}$, where $w_{i,k} \neq w_{j3}$, with $i=1,2$, $1\le k\le r$. On the basis of the different intervals $L_{i,k}$ intersect, we further partition $\Gamma_2$ into two subsets:

$\Gamma_{21}$: $\Gamma_{21} \subseteq \Gamma_2$. Here, $L_{i,k}$ leaves the real axis and enters $\mathbb{C}^{*}$, assuming that $L_{i,k}$ intersect at $w_{j3}$, where $w_{j3}$ is one of $w_{j2}$, $j2=1,\cdots,r+1$. $w_{i,k} \neq w_{j3}$.

$\Gamma_{22}$: $\Gamma_{22} \subseteq \Gamma_2$. Here, $L_{i,k}$ departs from the real axis and enters $\mathbb{C}^{*}$, assuming that $L_{i,k}$ intersect at $w_{i,k}$.

These subsets satisfy: $\Gamma_2=\Gamma_{21}\cup\Gamma_{22}$, $\Gamma_{21}\cap\Gamma_{22}=\phi$. In $\Gamma_{21}$, assume that the root locus $L_{i,k}$ intersect at a point in another interval $w_{j3}$ on the real axis again. In $\Gamma_{22}$, assume that the root locus $L_{i,k}$  intersects at a point in the interval $w_{i,k}$ where $L_{i,k}$ leaves on the real axis again.

There are two ways in which root loci leave the real axis in Lemma 3.3:

1. The two root loci emitted from the adjacent real zero of $p(s)$ intersect, causing root loci $L_{i,k}$ from real zeros $Q_k$ of $p(s)$ to leave the real axis.

2. The root loci are emitted from non-real zeros of $p(s)$ that intersect the real axis, and the root locus of the non-real zero intersects with the root locus $L_{i,k}$ of the real zero, causing root loci $L_{i,k}$ from real zeros $Q_k$ of $p(s)$ to leave the real axis.

The point $t1_{1,k}$ may be the intersection point of root loci $L_{1,k}$ and $L_{2,k-1}$ in $w_{1,k}$. The point $t1_{2,k}$ may be the intersection point of root loci $L_{2,k}$ and  $L_{1,k+1}$ in $w_{2,k}$. The point $t1_{i,k}$ may also be the intersection point of root loci $L_{i,k}$ and $U_k$ or $U_k^{'}$. The root locus $L_{i,k}^{'}$ is symmetric to $L_{i,k}$ with respect to the real axis. To complete the proof of Lemma 3.3, we assume the existence of the point $t2_{i,j3}$ and the root locus $L_{i,k}^{'}$.

The root locus $L_{i,k}^{'}$ is not emitted from the real zero $Q_k$. When no root loci of non-real zeros intersect the interval $w_{i,k}$, the root loci $L_{2,k-1}$ of real zero $Q_{k-1}$ or $L_{1,k+1}$ of real zero $Q_{k+1}$  intersect $L_{i,k}$ at a point $t1_{i,k}$, then $L_{i,k}^{'}$ is one of $L_{2,k-1}$ and $L_{1,k+1}$. When the root loci of non-real zeros intersect the interval $w_{i,k}$, the root loci $U_k$ or $U_k^{'}$ of the non-real zeros $s_k$ or $s_k'$ intersect $L_{i,k}$ at a point $t1_{i,k}$, then $L_{i,k}^{'}$ is one of $U_k$ and $U_k^{'}$.

\begin{lemma}
Suppose root locus $L_{i,k}\in\Gamma_{21}$ departs the real axis at point $t1_{i,k}$ in interval $w_{i,k}$, where $t1_{i,k} \neq Q_k$, and enters $\mathbb{C}^{*}$. Then $L_{i,k}$ cannot intersect any point in another interval $w_{j_3}$, where  $w_{j3}$ is one of $w_{j2}$, $j2=1,\cdots,r+1$, $w_{i,k} \neq w_{j3}$, for $i=1,2$, $1\le k\le r$.
\end{lemma}
\begin{proof}
The root locus $L_{i,k}$ is emitted from the real pole $Q_k$ of (3.1) and leaves the real axis at point $t1_{i,k}$ within the interval $w_{i,k}$, where $t1_{i,k}$ is not a pole of (3.1). $L_{i,k}$ enters $\mathbb{C}^{*}$. Suppose that $L_{i,k}$ intersects a point $t2_{i,j3}$ in another interval $w_{j3}$, where $w_{i,k} \neq w_{j3}$.  By Theorem 1.27, in $\mathbb{C}^{*}$, which is symmetric with respect to the real axis, there must exist another root locus $L_{i,k}^{'}$. Thus, the two root loci $L_{i,k}$ and $L_{i,k}'$ form a closed region.

These $r$ distinct real zeros of $p(s)$ divide the real axis into $r+1$ intervals, $w_{i,k}$ and $w_{j3}$ are two distinct intervals. Therefore, between the intervals $w_{i,k}$ and $w_{j3}$, there exists at least one real zero $QQ$ of $p(s)$. Within this closed region, there must be at least one real zero $QQ$ of $p(s)$.

The degrees of $L_{i,k}$ and $L_{i,k}^{'}$ are either $2q\pi$ or $2q\pi+\pi$. However, the real zero $QQ$ emits at least one root locus of $2q\pi$ degree and one of $2q\pi+\pi$ degree. Here, we assume that the degree of $L_{i,k}$ and $L_{i,k}^{'}$ are  $2q\pi$. The real zero $QQ$ must emit at least one root locus of  $2q\pi+\pi$ degree. The unit complex value of this root locus is distinct from those of $L_{i,k}$ and $L_{i,k}^{'}$. By Theorem 1.15, at least one root locus with a distinct unit complex value cannot pass through $L_{i,k}$ or $L_{i,k}^{'}$. Therefore, at least one root locus of  $2q\pi+\pi$ degree is enclosed within the closed region. By Theorem 1.14, all the root loci of (3.1) must eventually reach the point at infinity in $\mathbb{C}\cup\{\infty\}$.

Thus, assumption that $L_{i,k}$ of (3.1) leaves one interval and intersects a point in another interval $w_{j3}$ leads to a contradiction, proving the lemma.
\end{proof}
The interval $w_{j3}$ may be to the right or left of the interval $w_{i,k}$. The root locus $L_{i,k}$ may intersect a point in $w_{j3}$ to the right or left of the real zero $Q_k$. In both cases, the proof of Lemma 3.3 applies, and separate proofs are unnecessary. Details such as the shape of the root loci $L_{i,k}$ and $L_{i,k}^{'}$ are not required for the proof. Two continuous root loci connect two points in the intervals $w_{j3}$ and $w_{i,k}$, forming a closed region. These are the aspects that need to be provided in the proof of Lemma 3.3.

By Lemma 3.3, a root locus emitted from a real zero of $p(s)$ leaving an interval $w_{i,k}$ on the real axis cannot intersect any point in other intervals $w_{j3}$. This excludes the entry of root loci from real zeros into other intervals. If the root loci $U_k$ and $U_k^{'}$, emitted from the non-real zeros ($s_k$ and $s_k^{'}$) of $p(s)$, intersect a point in an interval $w_{i,k}$ on the real axis and then leave this interval to re-enter $\mathbb{C}^{*}$, we must consider whether they intersect a point in other intervals $w_{j3}$, where $w_{j3}$ is one of $w_{j2}$, $j2=1,\cdots,r+1$, $w_{i,k} \neq w_{j3}$, for $i=1,2$, $1\le k\le r$. If $U_k$ and $U_k^{'}$ leave $w_{i,k}$ at point $o_{i,k}$ and intersect a point $o_{i,j3}$ in another interval $w_{j3}$, their structure is identical to that in Lemma 3.3. Thus, by repeating the proof of Lemma 3.3, we establish Lemma 3.4.

\begin{lemma}
Suppose root loci $U_k$ and $U_k^{'}$, emitted from a pair of non-real zeros ($s_k$ and $s_k^{'}$), intersect in an interval $w_{i,k}$. If $U_k$ and $U_k^{'}$ leave the real axis from $w_{i,k}$ and re-enter $\mathbb{C}^{*}$, then they cannot intersect any point in another interval $w_{j3}$, where $w_{j3}$ is one of $w_{j2}$, $j2=1,\cdots,r+1$, $w_{i,k} \neq w_{j3}$, for $i=1,2$, $1\le k\le r$.
\end{lemma}
A simple real zero of the real polynomial $p(s)$ can only emit root loci of $2q\pi$ and $2q\pi+\pi$ degree to the left and right sides of the real axis, respectively. Thus, these $2q\pi$ and $2q\pi+\pi$ degree root loci emitted from the simple real zero of $p(s)$ cannot enter $\mathbb{C}^{*}$ directly. However, real multiple zeros of $p(s)$ must directly emit root loci of $2q\pi$ or $2q\pi+\pi$ degree to $\mathbb{C}^{*}$. Whether these $2q\pi$ and $2q\pi+\pi$ degree root loci directly emitted by the real multiple zeros of $p(s)$ to $\mathbb{C}^{*}$ can intersect at a point in the real axis, will be addressed later.

\begin{lemma}
For root locus $T_{l}\in\Gamma_{1}$, the real multiple pole $Q_l$ of (3.1) emits the root loci $T_l$ and $T_l'$ directly into $\mathbb{C}^{*}$. Assume that $T_l$ intersects a point $d_l$ on the real axis. Then, the root loci $T_l$, $T_l'$, $d_l$, and the real multiple zero $Q_l$ enclose a region. The difference in the degrees of $T_l$ and $T_l^{'}$ is $2\pi n_l$, where $n_l$ is a positive integer.
\end{lemma}
\begin{proof}
The real multiple pole $Q_l$ of (3.1) emits a root locus $T_l$ of $2q\pi$ or $2q\pi+\pi$ degree directly into $\mathbb{C}^{*}$. $T_l$ is not emitted in the direction of the real axis. If $T_l$ intersects a point $d_l$ on the real axis, Theorem 1.27 implies the existence of another root locus $T_l^{'}$, which is symmetric to $T_l$ about the real axis. $T_l^{'}$ has the same unit complex value as $T_l$ and is also emitted from $Q_l$ to $d_l$. Together, $T_l$, $T_l^{'}$, $d_l$, and $Q_l$ form a closed region.

Since $Q_l$ emits $T_l$ and $T_l^{'}$, $Q_l$ is divided into two parts: one facing the enclosed region and the other facing outwards. Both $T_l$ and $T_l^{'}$ are $2q\pi$ or $2q\pi+\pi$ degree. By Theorem 1.33 and Theorem 1.35, the difference between their degrees is $2\pi n_l$, where $n_l$ is a positive integer.
\end{proof}

The root loci $T_l$ and $T_l^{'}$ may intersect a point to the right or left of $Q_l$. In both cases, the proof of Lemma 3.5 applies, and separate proofs are unnecessary. Details such as the shape of the root loci are not required for the proof. Two continuous root loci connect $Q_l$ with a point $d_l$, forming an enclosed region. These are the aspects that need to be provided in the proof of Lemma 3.5.

\begin{lemma}
For root locus $T_{l}\in\Gamma_{1}$: The real pole $Q_l$ of (3.1) emits a root locus $T_l$ directly into $\mathbb{C}^{*}$. When $n_l>1$, then $T_l$ lies entirely in $\mathbb{C}^{*}$ and cannot intersect the real axis, where $1\le l\le r$.
\end{lemma}

\begin{proof}
By Lemma 3.5, if $n_l>1$, Theorems 1.33 and Theorems 1.35 implies that the part of $Q_l$ facing the closed region must emit at least one root locus of $2q\pi+\pi$ degree, and at least one root locus of $2q\pi$ degree. The root loci $T_l$ and $T_l^{'}$ select one of these two degrees. The real zero $Q_l$ emits at least one root locus of $2q\pi$ or $2q\pi+\pi$ degree, denoted as $T3_l$.

This root locus $T3_l$ is emitted from the section of $Q_l$ facing the closed region. If the degree of $T3_l$ is $2q\pi$, the degrees of $T_l$ and $T_l^{'}$ are $2q\pi+\pi$. If the degree of $T3_l$ is $2q\pi+\pi$, the degrees of $T_l$ and $T_l^{'}$ are $2q\pi$. Furthermore, the unit complex value of $T3_l$ is distinct from those of $T_l$ and $T_l^{'}$. By Theorem 1.15, $T3_l$ cannot intersect with $T_l$ or $T_l^{'}$ and cannot leave the closed region.

By Theorem 1.14, all the root loci of (3.1) must eventually reach the point at infinity of $\mathbb{C}\cup\{\infty\}$. If the root locus $T_l$ is assumed to intersect at a point $d_l$ on the real axis, a contradiction arises. Therefore, the multiple real zero $Q_l$ emits $T_l$ directly into $\mathbb{C}^{*}$, ensuring that $T_l$ is only distributed on $\mathbb{C}^{*}$ and does not intersect the real axis.
\end{proof}

\begin{lemma}
For root locus $T_{l}\in\Gamma_{1}$: The real pole $Q_l$ of (3.1) emits a root locus $T_l$ directly into $\mathbb{C}^{*}$. When $n_l=1$, then $T_l$ lies entirely in $\mathbb{C}^{*}$ and cannot intersect the real axis, where $1\le l\le r$.
\end{lemma}
\begin{proof}
By Lemma 3.5, when $n_l=1$, if the root loci $T_l$ and $T_l^{'}$ are $2q\pi$ degree, by Theorem 1.33 and Theorem 1.35, the portion of the multiple real zero $Q_l$ facing the closed region must emit a root locus of $2q\pi+\pi$ degree. This root locus of $2q\pi+\pi$ degree is confined to this closed region, can not end at infinity zero. Similarly, if $T_l$ and $T_l^{'}$ are $2q\pi+\pi$ degree, then a portion of $Q_l$ facing the closed region must emit a root locus of $2q\pi$ degree, which is also confined to the closed region and cannot end at infinity zero.

By Theorem 1.14, the root loci of all degrees of (3.1) must finally reach the point at infinity of $\mathbb{C}\cup\{\infty\}$. If $T_l$ were to intersect a point $d_l$ on the real axis, a contradiction would arise. Therefore, the multiple real zero $Q_l$ emits $T_l$ directly into $\mathbb{C}^{*}$, ensuring that $T_l$ is confined to $\mathbb{C}^{*}$ and does not intersect with the real axis.
\end{proof}
When $n>1$, we show Lemma 3.6. When $n=1$, we show Lemma 3.7. Lemma 3.8 is the result of $n\ge1$.
Combining Lemma 3.6 and Lemma 3.7, we obtain Lemma 3.8.
\begin{lemma}
For root locus $T_{l}\in\Gamma_{1}$: The multiple real pole $Q_l$ of (3.1) emits a root locus $T_l$ directly into $\mathbb{C}^{*}$. Then $T_l$ lies entirely in $\mathbb{C}^{*}$ and cannot intersect the real axis, where $1\le l\le r$.
\end{lemma}

Notation 5:

$c_{i,k}$ denotes an intersection point of the root loci $U_k$ and $U_k^{'}$ in the interval $w_{i,k}$. Similarly, $b_{i,k}$ denotes an intersection point of the root loci $U_k$ or $U_k^{'}$ and $L_{i,k}$ in the interval $w_{i,k}$. Furthermore, let us define that point $e_k^{'}$ as an intersection point of $U_k^{'}$ and $L_{i,k}$ in the upper half of $\mathbb{C}^{*}$. Let point $e_k$ denote an intersection point of $U_k$, situated in the lower half of $\mathbb{C}^{*}$. $d_{i,k}$ denotes an intersection point at which $L_{i,k}$ re-intersects the real axis in the interval between $c_{i,k}$ and $b_{i,k}$.

Each of the two intervals $w_{i,k}$ has two endpoints. $y_{i,k}\neq Q_k$. For $i=1$, $w_{1,k}$=[$Q_{k}$, $Q_{k-1}$]. For $i=2$, $w_{2,k}$=[$Q_{k+1}$, $Q_{k}$]. Let $y_{i,k}$ denote one of the two endpoints of each finite interval $w_{i,k}$, when $i=1$, $y_{1,k}=Q_{k-1}$. When $i=2$, $y_{2,k}=Q_{k+1}$. The two infinite intervals ($w_{2,r+1}$ and $w_{1,1}$) are ($-\infty$, $Q_{r+1}$] and [$Q_{1}$, $+\infty$), respectively. $y_{1,1}=+\infty$, $y_{2,r+1}=-\infty$.

We define $v_{i,k}$ as a intersection point where $L_{i,k}$ re-intersects the real axis in the interval between points $c_{i,k}$ and $y_{i,k}$. $LT_k$ denotes a root locus emitted by the real zero $Q_k$ of $p(s)$. When $LT_k$ leaves the real axis at a point on the real axis and enters $\mathbb{C}^{*}$, $LT_k$ is $L_{i,k}$.  When $LT_k$ directly  enters $\mathbb{C}^{*}$ from the multiple real zero $Q_l$, $LT_k$ is $T_{l}$.

The two infinite points on the real axis are the zeros of $1/p(s)$. When $p(s)$ has at least two real zeros, a situation arises where two real poles of (3.1) are adjacent. If the root loci in $\mathbb{C}^{*}$ do not intersect the interval between these two adjacent real poles, and only the root loci of the two adjacent real poles lie within this interval, then these root loci of the two adjacent real poles intersect, generating a real breakaway point. From this real breakaway point, the root loci of the two adjacent real poles leave the real axis and enter $\mathbb{C}^{*}$. These root loci, $L_{i,k}$, may later re-intersect the interval $w_{i,k}$. This constitutes the first case (a).

When $p(s)$ has at least two real zeros, there exist two adjacent real poles of (3.1). If the root loci $U_k$ or $U_k^{'}$ of the non-real zeros $(s_k, s_k^{'})$ intersect in the interval $w_{i,k}$ between these two adjacent real poles, the following occurs: If both $U_k$ and $U_k^{'}$ intersect in the interval $w_{i,k}$, first, they intersect at a real breakaway point. Then, they leave the real breakaway point and extend in opposite directions, and they intersect with other root loci, generating two additional real breakaway points.

- If only one pair of root loci from non-real zeros intersect in the interval $w_{i,k}$ between two adjacent real poles of (3.1), they intersect with the root loci of the two real poles, producing three real breakaway points.

- If $n_1$ pairs of root loci from non-real zeros intersect in the interval $w_{i,k}$ between two adjacent real poles of (3.1), they intersect with each other and with the root loci of the two real poles, producing $2n_1+1$ real breakaway points.

This constitutes the second case (b).

The proofs required for our results differ between these two cases. Therefore, we partition the set $\Gamma_{22}$ into two subsets on the basis of whether root loci from non-real zeros intersect in the interval between the adjacent real poles of (3.1).

When $p(s)$ has only one real zero:

- If no root loci from non-real zeros intersect in the interval between the real zero of $p(s)$ and positive or negative infinity, root locus emitted from real zero cannot leave the real axis, ends directly at the positive and negative infinity on the real axis. No root loci leave the real axis, no real breakaway points are generated, so this scenario does not require proof in our classification.

- If root loci from non-real zeros do intersect the interval between the real zero of $p(s)$ and positive or negative infinity, the root locus of real zero of $p(s)$ leaves the real axis and enters $\mathbb{C}^{*}$. This behavior aligns with the second case described above. Thus, when $p(s)$ has only one real zero and root loci from non-real zeros intersect in the interval between the real zero and infinity, we classify this special case under the second case (b).

$\Gamma_{221}$: $\Gamma_{221} \subseteq \Gamma_{22}$. Here, $L_{i,k}$ departs from $w_{i,k}$ and enters $\mathbb{C}^{*}$, and no root loci from non-real zeros intersect in $w_{i,k}$.

$\Gamma_{222}$: $\Gamma_{222} \subseteq \Gamma_{22}$. Here, $L_{i,k}$ departs from $w_{i,k}$ and enters $\mathbb{C}^{*}$, and root loci from non-real zeros intersect $w_{i,k}$. This includes the special case where $p(s)$ has only one real zero, and root loci from non-real zeros intersect the interval between real zero of $p(s)$ and infinity.

These subsets satisfy: $\Gamma_{22}=\Gamma_{221}\cup\Gamma_{222}$ and $\Gamma_{221}\cap\Gamma_{222}=\phi$.

The two root loci emitted from the adjacent real zero of $p(s)$ intersect, causing the root loci from real zeros of $p(s)$ to leave the real axis.
The point $b_{1,k}$ is the intersection point of root loci $L_{1,k}$ and $L_{2,k-1}$ in $w_{1,k}$. The point $b_{2,k}$ is the intersection point of root loci $L_{2,k}$ and  $L_{1,k+1}$ in $w_{2,k}$.

\begin{lemma}
Suppose root locus $L_{i,k}\in\Gamma_{221}$ departs the real axis at $b_{i,k}$ ($b_{i,k} \neq Q_k$) and enters $\mathbb{C}^{*}$. Then $L_{i,k}$ cannot intersect any point in the interval between $b_{i,k}$ and the real zero $Q_k$ of $p(s)$, for $i=1,2$, $1\le k\le r$.
\end{lemma}

\begin{proof}
If $L_{i,k}$ leaves the real axis from $b_{i,k}$ in the interval $w_{i,k}$, it may intersect a point in $w_{i,k}$. If $L_{i,k}$ intersects a point $t_{i,k}$ in the interval between $Q_k$ and $b_{i,k}$, this would imply that $L_{i,k}$ intersects itself at $t_{i,k}$. By Theorem 1.39, such a case is impossible. Therefore, $L_{i,k}$ cannot intersect any point in the interval between $Q_k$ and $b_{i,k}$.
\end{proof}

When the point $b_{i,k}$ coincides with $b_{1,k}$, $b_{1,k}$ lies within the interval $w_{1,k}$. The root locus $L_{1,k}$ intersects at a point in the interval to the right of the real zero $Q_k$, which is the interval between $Q_k$ and $b_{1,k}$. Similarly, when $b_{i,k}$ coincides with $b_{2,k}$, $b_{2,k}$ lies within the interval $w_{2,k}$. The root locus $L_{2,k}$ intersects at a point in the interval to the left of the real zero $Q_k$, which is in the interval between $Q_k$ and $b_{2,k}$. The root locus $L_{i,k}$ is in the upper half of $\mathbb{C}^{*}$, and it is also in the lower half of $\mathbb{C}^{*}$. The proof of Lemma 3.9 applies to all four possible distributions of the root loci without requiring separate proofs.

Furthermore, details such as the shape of the root locus $L_{i,k}$ are unnecessary for the proof of Lemma 3.9. This continuous root locus, defined on $\mathbb{C}^{*}$, rotating 180 degrees and reaches point $t_{i,k}$ from point $b_{i,k}$. The root locus $L_{i,k}$ intersects itself at point $t_{i,k}$. These aspects are essential for the proof of Lemma 3.9.

The root loci $L_{i,k}$ in the set $\Gamma_{222}$ originate from the real pole $Q_k$ of (3.1), enter $\mathbb{C}^{*}$, and then intersect in $w_{i,k}$.

- $U_k$ and $U_k^{'}$ intersect on the real axis at the real breakaway point $c_{i,k}$.

- $U_k$ or $U_k^{'}$ also generates a real breakaway point $b_{i,k}$ with the root locus $L_{i,k}$.

- Any other root loci (if they exist) emitted from non-real zeros and intersecting $w_{i,k}$ do not intersect with $L_{i,k}$.

Thus, the two breakaway points $b_{i,k}$ and $c_{i,k}$ partition $w_{i,k}$ into three sub-intervals:

1. The sub-interval between $Q_k$ and $b_{i,k}$,

2. The sub-interval between $b_{i,k}$ and $c_{i,k}$,

3. The sub-interval between $c_{i,k}$ and $y_{i,k}$ .

Assume that the root locus $L_{i,k}$ intersects in any of these three sub-intervals, leading to three distinct behaviors:

1. Assume that $L_{i,k}$ intersects in  the sub-interval between $Q_k$ and $b_{i,k}$,

2. Assume that $L_{i,k}$ intersects in the sub-interval between $b_{i,k}$ and $c_{i,k}$,

3. Assume that $L_{i,k}$ intersects in the sub-interval between $c_{i,k}$ and $y_{i,k}$.

These three cases require distinct proofs for the results we seek. Therefore, Assume that $L_{i,k}$ intersects in $w_{i,k}$, the root loci in $\Gamma_{222}$ can be classified into three subsets:

$\Gamma_{2221}$: $\Gamma_{2221} \subseteq \Gamma_{222}$. Here, assume that $L_{i,k}$ intersects in the interval between $Q_k$ and $b_{i,k}$.

$\Gamma_{2222}$: $\Gamma_{2222} \subseteq \Gamma_{222}$. Here, assume that $L_{i,k}$ intersects in the interval between $b_{i,k}$ and $c_{i,k}$.

$\Gamma_{2223}$: $\Gamma_{2223} \subseteq \Gamma_{222}$. Here, assume that $L_{i,k}$ intersects in the interval between $c_{i,k}$ and $y_{i,k}$.

These subsets satisfy: $\Gamma_{222}=\Gamma_{2221}\cup\Gamma_{2222}\cup\Gamma_{2223}$, $\Gamma_{2221}\cap\Gamma_{2222}=\phi$. $\Gamma_{2221}\cap\Gamma_{2223}=\phi$ and $\Gamma_{2222}\cap\Gamma_{2223}=\phi$.

Repeating the proof of Lemma 3.9, we can establish Lemma 3.10.
The root loci are emitted from non-real zeros of $p(s)$ that intersect in $w_{i,k}$, and the root locus of non-real zero intersects with the root locus of the real zero, causing root loci from real zeros of $p(s)$ to leave the real axis. The point $b_{i,k}$ is the intersection point of root loci $L_{i,k}$ and $U_k$ or $U_k^{'}$.

\begin{lemma}
Suppose root locus $L_{i,k}\in\Gamma_{2221}$ departs the real axis at point $b_{i,k}$ ($b_{i,k} \neq Q_k$) and enters $\mathbb{C}^{*}$. Then $L_{i,k}$ cannot intersect any point in the interval between $b_{i,k}$ and the real zero $Q_k$ of $p(s)$, for $i=1,2$, $1\le k\le r$.
\end{lemma}

The root locus $L_{i,k}$ leaves the real axis and enters $\mathbb{C}^{*}$. Two conditions can lead to this situation. The first condition is that root loci emitted from non-real zeros of $p(s)$ do not intersect in the interval $w_{i,k}$, and $w_{i,k}$ lies between two adjacent real zeros of $p(s)$. After the root loci emitted from these two adjacent real zeros of $p(s)$ intersect, they leave the real axis. Lemma 3.9 applies to this condition. The second condition is that root loci emitted from non-real zeros of $p(s)$ intersect in the interval $w_{i,k}$, causing the root locus $L_{i,k}$ emitted from the real zero $Q_k$ to leave the real axis. Lemma 3.10 applies to this condition.

Once $L_{i,k}$ leaves the real axis from a point that is not a pole of (3.1), $L_{i,k}$ enters $\mathbb{C}^{*}$ and cannot re-intersect the real axis. This result requires the consideration of two distinct cases. The first case is that $L_{i,k}$ cannot intersect any point in another interval $w_{j3}$ on the real axis, where  $w_{j3}$ is one of $w_{j2}$, $j2=1,\cdots,r+1$. $w_{i,k} \neq w_{j3}$. For $i=1,2$, $1\le k\le r$. The second case is that $L_{i,k}$ cannot intersect any point in the interval $w_{i,k}$ that it has previously left. After these two cases are proven, the results are combined to conclude that $L_{i,k}$ leaves the real axis from a point that is not pole of (3.1) and enters $\mathbb{C}^{*}$, $L_{i,k}$ cannot re-intersect at a point in the real axis.

By Lemma 3.3, the real zeros of $p(s)$ emit root loci that leave the interval $w_{j4}$. These root loci cannot intersect any point in interval $w_{i,k}$, where  $w_{j4}$ is one of $w_{j2}$, $j2=1,\cdots,r+1$. $w_{i,k} \neq w_{j4}$. For $i=1,2$, $1\le k\le r$.  Lemma 3.4 provides an analogous result for non-real zeros. By Lemma 3.4, the root loci emitted from the non-real zeros of $p(s)$ intersect a point in $w_{j4}$, but upon re-entering $\mathbb{C}^{*}$, they cannot intersect any point in $w_{i,k}$.

Additionally, by Lemma 3.8, the root locus $T_l$ emitted from all multiple real zeros of $p(s)$ enters $\mathbb{C}^{*}$ directly and cannot intersect any point in interval $w_{i,k}$. If one of the endpoints $Q_e$ of $w_{i,k}$ (i.e., $Q_{k+1}$, $Q_k$, or $Q_{k-1}$) is a real multiple zero, the root loci emitted from $Q_e$ enter $\mathbb{C}^{*}$ directly and cannot intersect any point in $w_{i,k}$.

Therefore, the root locus from any point (including multiple real zeros of $p(s)$) on the real axis not in interval $w_{i,k}$ cannot intersect any point in $w_{i,k}$. Only the root loci emitted from non-real zeros of $p(s)$ may intersect a point in $w_{i,k}$. Next, we prove the following: After the root loci emitted from non-real zeros intersect in $w_{i,k}$, and the root loci $L_{i,k}$ emitted from real zeros leave $w_{i,k}$ and enter $\mathbb{C}^{*}$, these root loci $L_{i,k}$ emitted from real zeros cannot intersect any point in $w_{i,k}$.

Lemma 3.10 proves that after the root loci emitted from non-real zeros intersect at a point in $w_{i,k}$, and the root locus $L_{i,k}$ emitted from real zeros leaves $w_{i,k}$ and enters $\mathbb{C}^{*}$, $L_{i,k}$ cannot intersect any point in the interval between $b_{i,k}$ and the real zero $Q_k$ of $p$ on the real axis, where $i=1,2$, $1\le k\le r$.

The real pole $Q_k$ ($1\le k\le r$) of (3.1) emits a root locus $L_{i,k}$ to the real axis. Assume that $L_{i,k}$ leaves the real axis from point $b_{i,k}$, where $(b_{i,k}\neq Q_k)$, and enters $\mathbb{C}^{*}$. The root loci emitted from at least one pair of non-real zeros of $p$ intersect in $w_{i,k}$. A pair of the root loci $U_k$ and $U_k^{'}$ emitted from a pair of non-real zeros ($s_k$, $s_k^{'}$) of $p(s)$ intersect at point $c_{i,k}$ in $w_{i,k}$. $U_k$ or $U_k^{'}$ are adjacent to $L_{i,k}$ on the real axis, and $L_{i,k}$ intersects with $U_k$ or $U_k^{'}$ at point $b_{i,k}$.

Under these conditions, we need to prove two results:

1. $L_{i,k}$ cannot intersect any point in the interval between $c_{i,k}$ and $b_{i,k}$.

2. $L_{i,k}$ cannot intersect any point in the interval between $c_{i,k}$ and $y_{i,k}$.

To prove these results, we must consider whether $L_{i,k}$ intersects with $U_k^{'}$ in the upper half of $\mathbb{C}^{*}$. Therefore, the proofs under these conditions are divided into four cases:

1. First Case: $L_{i,k}$ does not intersect with $U_k^{'}$ in the upper half of $\mathbb{C}^{*}$, and $L_{i,k}$ intersects at point $d_{i,k}$ in the interval between $c_{i,k}$ and $b_{i,k}$.

2. Second Case: $L_{i,k}$ intersects with $U_k^{'}$ in the upper half of $\mathbb{C}^{*}$, and $L_{i,k}$ intersects at point $d_{i,k}$ in the interval between $c_{i,k}$ and $b_{i,k}$.

3. Third Case: $L_{i,k}$ intersects with $U_k^{'}$ in the upper half of $\mathbb{C}^{*}$, and $L_{i,k}$ intersects at point $v_{i,k}$ in the interval between $c_{i,k}$ and $y_{i,k}$.

4. Fourth Case: $L_{i,k}$ does not intersect with $U_k^{'}$ in the upper half of $\mathbb{C}^{*}$, and $L_{i,k}$ intersects at point $v_{i,k}$ in the interval between $c_{i,k}$ and $y_{i,k}$.

The root locus $L_{i,k}$ in the set $\Gamma_{2222}$ departs from the real axis, enters $\mathbb{C}^{*}$. Assume that $L_{i,k}$ intersects in the interval between $b_{i,k}$ and $c_{i,k}$ in $w_{i,k}$. The root loci $U_k$ and $U_k^{'}$ intersect on the real axis at the real breakaway point $c_{i,k}$, generating two symmetric segments in $\mathbb{C}^{*}$: $s_kc_{i,k}$ and $s_k^{'}c_{i,k}$. For $L_{i,k}$ in $\Gamma_{2222}$, two distinct cases may occur:
1. $L_{i,k}$ intersects with $s_k^{'}c_{i,k}$, or
2. $L_{i,k}$ does not intersect with $s_k^{'}c_{i,k}$.
Since the required proofs differ for these two cases, for the root locus $L_{i,k}$ in the set $\Gamma_{2222}$, on the basis of two distinct cases, we further partition $\Gamma_{2222}$ into two subsets:

$\Gamma_{22221}$: $\Gamma_{22221} \subseteq \Gamma_{2222}$. Here, $L_{i,k}$ does not intersect with $s_k^{'}c_{i,k}$.

$\Gamma_{22222}$: $\Gamma_{22222} \subseteq \Gamma_{2222}$. Here, $L_{i,k}$ intersects with $s_k^{'}c_{i,k}$.

These subsets satisfy: $\Gamma_{2222}=\Gamma_{22221}\cup\Gamma_{22222}$ and $\Gamma_{22221}\cap\Gamma_{22222}=\phi$.

The root locus $L_{i,k}$ in the set $\Gamma_{2223}$ departs from the real axis, enters $\mathbb{C}^{*}$. Subsequently, assume that  $L_{i,k}$ intersects in the interval between $c_{i,k}$ and $y_{i,k}$. For $L_{i,k}$ in $\Gamma_{2223}$, two distinct cases may occur:
1. $L_{i,k}$ intersects $s_k^{'}c_{i,k}$, or
2. $L_{i,k}$ does not intersect with $s_k^{'}c_{i,k}$.
Since the required proofs differ for these two cases, for the root locus $L_{i,k}$ in the set $\Gamma_{2223}$, on the basis of two distinct cases, we further partition $\Gamma_{2223}$ into two subsets:

$\Gamma_{22231}$: $\Gamma_{22231} \subseteq \Gamma_{2223}$. Here, $L_{i,k}$ does not intersect with $s_k^{'}c_{i,k}$.

$\Gamma_{22232}$: $\Gamma_{22232} \subseteq \Gamma_{2223}$. Here, $L_{i,k}$ intersects with $s_k^{'}c_{i,k}$.

These subsets satisfy: $\Gamma_{2223}=\Gamma_{22231}\cup\Gamma_{22232}$ and $\Gamma_{22231}\cap\Gamma_{22232}=\phi$.

\begin{lemma}
For $L_{i,k}\in\Gamma_{22221}$: Suppose root locus $U_k$ intersects $L_{i,k}$ at point $b_{i,k}$ in interval $w_{i,k}$, then $U_k$ and $L_{i,k}$ depart the real axis from $b_{i,k}$ in $w_{i,k}$, entering $\mathbb{C}^{*}$.
\end{lemma}

\begin{proof}
The root loci emitted from at least a pair of non-real zeros of $p(s)$ intersect in the interval $w_{i,k}$. Specifically, a pair of root loci $U_k$ and $U_k^{'}$ emitted from a pair of non-real zeros ($s_k$, $s_k^{'}$) of $p(s)$ intersect at point $c_{i,k}$ in the interval $w_{i,k}$. A segment $s_k^{'}c_{i,k}$ of $U_k^{'}$ is from the non-real zero $s_k^{'}$ to point $c_{i,k}$ in the upper half of $\mathbb{C}^{*}$, whereas a segment $s_kc_{i,k}$ of $U_k$ is from the non-real zero $s_k$ to point $c_{i,k}$ in the lower half of $\mathbb{C}^{*}$. Since the root loci of $2q\pi$ and $2q\pi+\pi$ degree of (3.1) are symmetric about the real axis, these two segments  $s_kc_{i,k}$ and $s_k^{'}c_{i,k}$ are symmetric about the real axis.

After the root loci $U_k$ and $U_k^{'}$ intersect at point $c_{i,k}$, they separate and extend in opposite directions on the real axis. $U_k$ or $U_k^{'}$ is adjacent to $L_{i,k}$ on the real axis. Here, we assume that the root loci $U_k$ and $L_{i,k}$ extend toward each other. Consequently, the root locus $U_k$ intersects with the root locus $L_{i,k}$ at point $b_{i,k}$ in the interval $w_{i,k}$. After this intersection, the root loci $U_k$ and $L_{i,k}$ leave the real axis and enter $\mathbb{C}^{*}$.
\end{proof}

\begin{lemma}
If the root locus $L_{i,k}\in\Gamma_{22221}$, then $L_{i,k}$ cannot intersect any point in the interval between $c_{i,k}$ and $b_{i,k}$.
\end{lemma}
\begin{proof}
By Lemma 3.11, if the root locus $L_{i,k}$ intersects at point $d_{i,k}$ in the interval between points $b_{i,k}$ and $c_{i,k}$, and if $L_{i,k}$ does not intersect with $s_k^{'}c_{i,k}$ in the upper half of $\mathbb{C}^{*}$, then a segment of $L_{i,k}$ is from point $b_{i,k}$ to point $d_{i,k}$ in the upper half of $\mathbb{C}^{*}$. Similarly, a segment of $U_k$ is from point $b_{i,k}$ to point $d_{i,k}$ in the lower half of $\mathbb{C}^{*}$. Owing to the symmetry of the root loci of $2q\pi$ degree and $2q\pi+\pi$ degree of (3.1) about the real axis, these two segments are symmetric about the real axis.

On the real axis, the root locus $U_k$ first extends from point $c_{i,k}$ to point $d_{i,k}$, then from point $d_{i,k}$ to point $b_{i,k}$. At point $b_{i,k}$, $U_k$ leaves the real axis and enters the lower half of $\mathbb{C}^{*}$. In the lower half of $\mathbb{C}^{*}$, $U_k$ extends from point $b_{i,k}$ to point $d_{i,k}$, where it intersects with itself at $d_{i,k}$. By Theorem 1.39, $U_k$ cannot intersect with itself. Such a case cannot hold. Since the root locus $L_{i,k}$ is assumed to intersect point $d_{i,k}$ on the real axis, a contradiction arises. Thus, under the conditions of this lemma, the root locus $L_{i,k}$ cannot intersect any point in the interval between points $c_{i,k}$ and $b_{i,k}$.
\end{proof}

In the proofs of Lemma 3.11 and Lemma 3.12, the segment from $c_{i,k}$ to $b_{i,k}$ is part of $U_k$. However, this segment can also be part of $U_k^{'}$. Therefore, for the latter kind of case of  $U_k$ and $U_k^{'}$, an alternative proof of Lemma 3.12 is provided.

\begin{proof}
When the segment from point $c_{i,k}$ to point $b_{i,k}$ is a part of the root locus $U_k^{'}$. In this case, $L_{i,k}$ in the upper half of $\mathbb{C}^{*}$ extends from point $b_{i,k}$ to point $d_{i,k}$. At these two points, the gains satisfy $K(d_{i,k})>K(b_{i,k})$. On the real axis, $U_k^{'}$ extends from point $d_{i,k}$ to point $b_{i,k}$, where the gains satisfy $K(d_{i,k})<K(b_{i,k})$. This leads to a contradiction. Since the root locus $L_{i,k}$ is assumed to intersect at point $d_{i,k}$ on the real axis, a contradiction arises. Thus, under the conditions of Lemma 3.11, the root locus $L_{i,k}$ cannot intersect at a point in the interval between points $c_{i,k}$ and $b_{i,k}$.
\end{proof}

When point $b_{i,k}$ is $b_{1,k}$, it lies within the interval $w_{1,k}$. The root locus $L_{1,k}$ intersects a point in the interval to the right of the real zero $Q_k$, which is the interval between points $c_{1,k}$ and $b_{1,k}$. When point $b_{i,k}$ is $b_{2,k}$, it lies within the interval $w_{2,k}$. The root locus $L_{2,k}$ intersects a point in the interval to the left of the real zero $Q_k$, which is the interval between points $c_{2,k}$ and $b_{2,k}$. For both cases, the proofs of Lemmas 3.11 and 3.12 apply without requiring separate proofs. Additionally, details such as the shape of the root locus $U_k$ are not necessary in the proofs of Lemmas 3.11 and 3.12. The continuous root locus $U_k$ travels 180 degree in $\mathbb{C}^{*}$ from point $b_{i,k}$ to point $d_{i,k}$, where it intersects with itself. These aspects are sufficient for the proof.

\begin{lemma}
If the root locus $L_{i,k}\in\Gamma_{22222}$, then $L_{i,k}$ cannot intersect any point in the interval between $c_{i,k}$ and $b_{i,k}$.
\end{lemma}
\begin{proof}
If the root locus $L_{i,k}$ ($1\le k\le r$) intersects at point $d_{i,k}$ in the interval between points $b_{i,k}$ and $c_{i,k}$, and if $L_{i,k}$ intersects with the root locus $s_k^{'}c_{i,k}$ in the upper half of $\mathbb{C}^{*}$, then $L_{i,k}$ has an intersection point $e_k^{'}$ with $s_k^{'}c_{i,k}$. In this case, on the real axis, the root locus $U_k$ first extends from point $c_{i,k}$ to point $d_{i,k}$, then from point $d_{i,k}$ to point $b_{i,k}$. At point $b_{i,k}$, $U_k$ leaves the real axis and enters the lower half of $\mathbb{C}^{*}$. A segment of $L_{i,k}$ is from point $b_{i,k}$ to point $d_{i,k}$ in the upper half of $\mathbb{C}^{*}$, whereas a segment of $U_k$ is from point $b_{i,k}$ to point $d_{i,k}$ in the lower half of $\mathbb{C}^{*}$. These two segments are symmetric about the real axis.

The root locus $U_k$ extends into the lower half of $\mathbb{C}^{*}$ and reaches point $e_k$, which is symmetric to point $e_k^{'}$ about the real axis. At point $e_k$, $U_k$ intersects with itself. By Theorem 1.39, such a case cannot hold. Since the root locus $L_{i,k}$ is assumed to intersect point $d_{i,k}$ in the interval between points $b_{i,k}$ and $c_{i,k}$, a contradiction arises. Thus, under the conditions of this lemma, the root locus $L_{i,k}$ cannot intersect any point in the interval between points $c_{i,k}$ and $b_{i,k}$.

In the above proof, the segment between points $c_{i,k}$ and $b_{i,k}$ is part of $U_k$. However, this segment can also be part of the root locus $U_k^{'}$. In this case, $L_{i,k}$ in the upper half of $\mathbb{C}^{*}$ extends from point $b_{i,k}$ to point $e_k^{'}$, where the gains satisfy $K(e_k^{'})>K(b_{i,k})$. $U_k^{'}$ extends from point $e_k^{'}$ in the upper half of $\mathbb{C}^{*}$ to point $c_{i,k}$, then from point $c_{i,k}$ to point $b_{i,k}$ on the real axis, where the gains satisfy $K(e_k^{'})<K(b_{i,k})$. This leads to a contradiction. Since $L_{i,k}$ is assumed to intersect point $d_{i,k}$ on the real axis, a contradiction arises. Thus, under the conditions of this lemma, the root locus $L_{i,k}$ cannot intersect any point in the interval between points $c_{i,k}$ and $b_{i,k}$.
\end{proof}

When point $b_{i,k}$ is point $b_{1,k}$, it lies within the interval $w_{1,k}$. The root locus $L_{1,k}$ intersects at a point in the interval to the right of the real zero $Q_k$, which is the interval between points $c_{1,k}$ and $b_{1,k}$. When point $b_{i,k}$ is point $b_{2,k}$, it lies within the interval $w_{2,k}$. The root locus $L_{2,k}$ intersects at a point in the interval to the left of the real zero $Q_k$, which is the interval between points $c_{2,k}$ and $b_{2,k}$. For both cases, the proof of Lemma 3.13 is applicable, and separate proofs are unnecessary. Additionally, details such as the shape of the root locus $U_k$ are not required in the proof of Lemma 3.13. This continuous root locus $U_k$, in $\mathbb{C}^{*}$, extends from point $b_{i,k}$ to point $e_k$, where it intersects with itself. These are the aspects that need to be provided in the proof of Lemma 3.13.

\begin{lemma}
If the root locus $L_{i,k}\in\Gamma_{22231}$, then $L_{i,k}$ cannot intersect any point in the interval between $c_{i,k}$ and $y_{i,k}$.
\end{lemma}
\begin{proof}
If the root locus $L_{i,k}$ ($1\le k\le r$) intersects at point $v_{i,k}$ in the interval between points $c_{i,k}$ and $y_{i,k}$, and if the root locus $L_{i,k}$ intersects with $s_k^{'}c_{i,k}$, then by repeating the proof of Lemma 3.13, we conclude that under the conditions of this lemma, the root locus $L_{i,k}$ cannot intersect at a point in the interval between points $c_{i,k}$ and $y_{i,k}$.
\end{proof}

The root loci $L_{i,k}$ and $U_k^{'}$ intersect in the upper half of $\mathbb{C}^{*}$. There are two situations in which they intersect. The first situation is that there is only one intersection point in segment $s_k^{'}c_{i,k}$ of $U_k^{'}$. This is the situation presented in the proof of Lemma 3.13.

The second situation is that there are two intersection points of the root loci $L_{i,k}$ and $s_k^{'}c_{i,k}$. The root locus $L_{i,k}$ is emitted from the real zero $Q_k$ of $p(s)$. The segment $s_k^{'}c_{i,k}$ of the root locus $U_k^{'}$ serves as the line of demarcation. One side of $s_k^{'}c_{i,k}$ contains the real zero $Q_k$, whereas the other side does not. $L_{i,k}$ lies on the side of $s_k^{'}c_{i,k}$ where $Q_k$ exists. $L_{i,k}$ passes through $s_k^{'}c_{i,k}$, extending from the side where $Q_k$ exists to the other side of $s_k^{'}c_{i,k}$. Then, from the other side of $s_k^{'}c_{i,k}$, $L_{i,k}$ passes through $s_k^{'}c_{i,k}$ again, returning to the side where $Q_k$ exists.

In the second situation, there are two intersection points of the root loci $L_{i,k}$ and $s_k^{'}c_{i,k}$. The root locus $L_{i,k}$ passes through $s_k^{'}c_{i,k}$, extending from the side where $Q_k$ exists to the other side of $s_k^{'}c_{i,k}$, producing the first intersection point. From the other side of $s_k^{'}c_{i,k}$, $L_{i,k}$ passes through $s_k^{'}c_{i,k}$ again, returning to the side where $Q_k$ exists, producing the second intersection point.

The proof of Lemma 3.13 addresses the case of one intersection point. For the case of two intersection points, only the first intersection point is needed. By using the first intersection point and repeating the proof of Lemma 3.13, the case of two intersection points can be proven. Therefore, in the proofs of Lemma 3.13 and Lemma 3.14, new results and proofs are unnecessary. Only Lemma 3.13 and Lemma 3.14 are provided. Additionally, in these lemmas, it is not necessary to specify whether there is one or two intersection points. The lemmas hold for both cases.

If there are cases where the root locus $L_{i,k}$ passes through $s_k^{'}c_{i,k}$ multiple times, the analysis is the same as above. These cases are also included in Lemma 3.13 and Lemma 3.14, and no additional results or proofs are needed.

\begin{lemma}
If root locus $L_{i,k}\in\Gamma_{22232}$, then the root loci $L_{i,k}$ and $U_k$ form a closed zone in $\mathbb{C}^{*}$ between points $b_{i,k}$ and $v_{i,k}$.
\end{lemma}
\begin{proof}
If the root locus $L_{i,k}$ ($1\le k\le r$) intersects at point $v_{i,k}$ in the interval between points $c_{i,k}$ and $y_{i,k}$, and if the root locus $L_{i,k}$ does not intersect with $s_k^{'}c_{i,k}$, then $L_{i,k}$ enters the upper half of $\mathbb{C}^{*}$ from $b_{i,k}$ and extends from point $b_{i,k}$ in the upper half of $\mathbb{C}^{*}$ to $v_{i,k}$.

Points $v_{i,k}$ and $b_{i,k}$ lie on opposite sides of point $c_{i,k}$ on the real axis. The root locus $U_k^{'}$, emitted from the non-real zero $s_k^{'}$, extends to point $c_{i,k}$. The root locus $L_{i,k}$ intersects at point $v_{i,k}$ but does not intersect with $s_k^{'}c_{i,k}$. Therefore, $L_{i,k}$ must pass above $U_k^{'}$ and the non-real zero $s_k^{'}$ to reach the real axis. A segment of $L_{i,k}$ is from point $b_{i,k}$ to point $v_{i,k}$ in the upper half of $\mathbb{C}^{*}$, whereas another segment of $U_k$ is from point $b_{i,k}$ to point $v_{i,k}$ in the lower half of $\mathbb{C}^{*}$. These two segments are symmetric about the real axis.

The root locus $U_k$ enters the lower half of $\mathbb{C}^{*}$ from point $b_{i,k}$, and extends from $b_{i,k}$ in the lower half of $\mathbb{C}^{*}$ to point $v_{i,k}$. Therefore, the root loci $L_{i,k}$ and $U_k$ form a closed zone in $\mathbb{C}^{*}$ between points $b_{i,k}$ and $v_{i,k}$.
\end{proof}

\begin{lemma}
If root locus $L_{i,k}\in\Gamma_{22232}$, then $L_{i,k}$ cannot intersect any point in the interval between $c_{i,k}$ and $y_{i,k}$.
\end{lemma}
\begin{proof}
By Lemma 3.15, in this case, the root loci $L_{i,k}$ and $U_k$ form a closed zone in $\mathbb{C}^{*}$ between points $b_{i,k}$ and $v_{i,k}$. The two non-real zeros $s_k$ and $s_k^{'}$ are connected to point $c_{i,k}$. The root locus $L_{i,k}$ does not intersect with $s_k^{'}c_{i,k}$ and extends from one side of $s_k^{'}c_{i,k}$ to the other side. Both $L_{i,k}$ and $U_k$ pass above the two non-real zeros ($s_k$ and $s_k^{'}$ ) and point $c_{i,k}$. Therefore, the closed zone contains at least the pair of non-real zeros $s_k$ and $s_k^{'}$.

The root loci $L_{i,k}$ and $U_k$ correspond to one of the two degrees $2q\pi$ or $2q\pi+\pi$. However, the pair of non-real zeros $s_k$ and $s_k^{'}$ must emit at least two root loci of $2q\pi$ and $2q\pi+\pi$ degrees. Therefore, for the root loci whose unit complex values differ from that of $L_{i,k}$, they cannot pass through $L_{i,k}$ and $U_k$. At least two of the root loci of $2q\pi$ or $2q\pi+\pi$ degrees emitted from $s_k$ and $s_k^{'}$ cannot leave this closed region.

By Theorem 1.14, the root loci of all degrees of (3.1) must eventually reach the point at infinity in $\mathbb{C}\cup\{\infty\}$. Thus, a contradiction arises under the assumption that the root locus $L_{i,k}$ intersects at point $v_{i,k}$ in the interval from $c_{i,k}$ to $y_{i,k}$. Therefore, the root locus $L_{i,k}$ cannot intersect at a point in the interval between points $c_{i,k}$ and $y_{i,k}$.
\end{proof}

When the point $b_{i,k}$ coincides with $b_{1,k}$, the point $b_{1,k}$ lies within the interval $w_{1,k}$. The root locus $L_{1,k}$ intersects at a point in the interval to the right of the real zero $Q_k$, specifically in the interval between points $c_{1,k}$ and $y_{1,k}$. When $b_{i,k}$ coincides with $b_{2,k}$, the point $b_{2,k}$ lies within the interval $w_{2,k}$. The root locus $L_{2,k}$ intersects at a point in the interval to the left of the real zero $Q_k$, specifically in the interval between $c_{2,k}$ and $y_{2,k}$. In both cases, the proofs of Lemma 3.15 and Lemma 3.16 are directly applicable, and no separate proofs are needed. Additionally, details such as the shape of the root locus $U_{k}$ are not necessary in the proofs of Lemma 3.15 and Lemma 3.16. The two continuous root loci $L_{i,k}$ and $U_{k}$ connect two points within the interval $w_{i,k}$ and form a closed region. These aspects must be included in the proofs of Lemma 3.15 and Lemma 3.16.

In the proofs of Lemma 3.15 and Lemma 3.16, the segment on the real axis between $c_{i,k}$ and $b_{i,k}$ belongs to the root locus $U_{k}$. However, this segment can also belong to the root locus $U_k^{'}$. The proof of this scenario is identical to the proofs of Lemma 3.12 and Lemma 3.13. By repeating these proofs, we obtain Lemma 3.16.

From Lemma 3.10 to Lemma 3.16 and their proofs, the two root loci $U_k^{'}$ and $U_{k}$ intersect at $c_{i,k}$, whereas the root loci $L_{i,k}$ and $U_{k}$ intersect at $b_{i,k}$. In general, $c_{i,k}$ and $b_{i,k}$ are distinct points. However, there is a special case where the root loci $L_{i,k}$, $U_k^{'}$, and $U_{k}$ intersect at a single point, meaning that $c_{i,k}=b_{i,k}$. In this special case, the interval between $c_{i,k}$ and $b_{i,k}$ reduces to a single point.

This special case is not addressed in Lemma 3.15, so the above results cannot be written as lemmas. However, this special case does not affect Lemma 3.10, Lemma 3.12, Lemma 3.13, Lemma 3.14, or Lemma 3.16; thus, no further explanation is needed.

Given $\Gamma_{2222}=\Gamma_{22221}\cup\Gamma_{22222}$, Lemmas 3.12 and 3.13 imply:

\begin{lemma}
If root locus $L_{i,k}\in\Gamma_{2222}$, then $L_{i,k}$ cannot intersect any point in the interval between $c_{i,k}$ and $b_{i,k}$.
\end{lemma}

Given $\Gamma_{2223}=\Gamma_{22231}\cup\Gamma_{22232}$, Lemmas 3.14 and 3.16 imply:

\begin{lemma}
If root locus $L_{i,k}\in\Gamma_{2223}$, then $L_{i,k}$ cannot intersect any point in the interval between $c_{i,k}$ and $y_{i,k}$.
\end{lemma}

Given $\Gamma_{222}=\Gamma_{2221}\cup\Gamma_{2222}\cup\Gamma_{2223}$, Lemmas 3.10, 3.17 and 3.18 imply:

\begin{lemma}
If root locus $L_{i,k}\in\Gamma_{222}$ departs the real axis at point $b_{i,k}$($b_{i,k} \neq Q_k$) and enters $\mathbb{C}^{*}$, then $L_{i,k}$ cannot intersect any point in $w_{i,k}$.
\end{lemma}
The point $b_{1,k}$ is the intersection point of root loci $L_{1,k}$ and $L_{2,k-1}$ in $w_{1,k}$. The point $b_{2,k}$ is the intersection point of root loci $L_{2,k}$ and  $L_{1,k+1}$ in $w_{2,k}$.

When the root loci of non-real zeros of $p(s)$ don't intersect in the intervals between two adjacent real poles of (3.1). In the intervals between two adjacent real poles of (3.1), there exist the root loci ($L_{1,k}$ and $L_{2,k-1}$) or ($L_{2,k}$ and  $L_{1,k+1}$) emitted from real poles of (3.1). The two root loci in these interval intersect and generate the real breakaway points  ($b_{1,k}$, or, $b_{2,k}$). From the real breakaway points  ($b_{1,k}$, or, $b_{2,k}$), the root loci leave the real axis and enter  $\mathbb{C}^{*}$. If $L_{1,k}$ or $L_{2,k}$ re-intersects in $w_{i,k}$. That is, re-intersects in the sub-interval between $b_{1,k}$ and $Q_k$; re-intersects in the sub-interval between $b_{1,k}$ and $Q_{k-1}$; re-intersects in the sub-interval between $b_{2,k}$ and $Q_k$; re-intersects in the sub-interval between $b_{2,k}$ and $Q_{k+1}$.

Lemma 3.9 has proved $L_{1,k}$ or $L_{2,k}$ doesn't intersect in two sub-intervals. In the proof of Lemma 3.9,   $L_{i,k}$ cannot intersect itself. We use this result to complete our proof. Repeated the proof of Lemma 3.9,  $L_{1,k+1}$ or $L_{2,k-1}$ cannot intersect itself, by utilizing this result, we can prove $L_{1,k}$ or $L_{2,k}$ doesn't intersect in another two sub-intervals. Therefore, we prove $L_{1,k}$ or $L_{2,k}$ doesn't intersect in four sub-intervals. Hence, we have: Suppose root locus $L_{i,k}\in\Gamma_{221}$ departs the real axis at $b_{i,k}$ ($b_{i,k} \neq Q_k$) and enters $\mathbb{C}^{*}$. Then $L_{i,k}$ cannot intersect any point in  $w_{i,k}$.

Given $\Gamma_{22}=\Gamma_{221}\cup\Gamma_{222}$, Lemmas 3.19 and this result imply:

\begin{lemma}
When the root locus $L_{i,k}\in\Gamma_{22}$ leaves the real axis at the point $b_{i,k}$($b_{i,k} \neq Q_k$) in the interval $w_{i,k}$ and enters $\mathbb{C}^{*}$. Then, the root locus $L_{i,k}$ cannot intersect at any point in $w_{i,k}$.
\end{lemma}

Given $\Gamma_{2}=\Gamma_{21}\cup\Gamma_{22}$, Lemmas 3.3 and 3.20 imply:

\begin{lemma}
If root locus $L_{i,k}\in\Gamma_{2}$ departs the real axis at point $b_{i,k}$($b_{i,k} \neq Q_k$) in interval $w_{i,k}$ and enters $\mathbb{C}^{*}$, then $L_{i,k}$ cannot re-intersect  on the real axis.
\end{lemma}

Given $\Gamma=\Gamma_{1}\cup\Gamma_{2}$, Lemmas 3.8 and 3.21 imply Theorem 3.22.

\begin{theorem}
Let $p(s)$ be a real polynomial with at least one real zero.  $LT_{k}\in\Gamma$. Each real zero $Q_k$ of $p(s)$ emits a root locus $LT_k$ of $2q\pi$ or $2q\pi+\pi$ degree. If $LT_k$ enters $\mathbb{C}^{*}$, then it cannot re-intersect  on the real axis.
\end{theorem}

For the root locus of (3.1) emitted from the real zeros of the polynomial $p(s)$ entering $\mathbb{C}^{*}$, there are three possible scenarios:

1. First scenario: Two adjacent real zeros of $p(s)$ form an interval. Two root loci intersect at a point within this interval, and from this intersection point, the root loci enter $\mathbb{C}^{*}$.

2. Second scenario: Two non-real complex zeros of $p(s)$ emit root loci toward the real axis. These root loci intersect on the real axis with the root locus emitted from the real zeros at a point. From this intersection point, the root locus emitted from the real zero enters $\mathbb{C}^{*}$.

3. Third scenario: For a multiple real zero of $p(s)$, the multiple real zero directly emits root loci into $\mathbb{C}^{*}$.

After leaving the real axis and entering $\mathbb{C}^{*}$ in these three ways, the root loci of (3.1) emitted from the real zeros of $p(s)$ cannot intersect any point on the real axis again. Theorem 3.22 states that when proving the generation of the second kind of zeros of $(p^{'}(s)/p(s))^{'}$, we do not need to consider the root loci of (3.1) emitted from the real zeros of $p(s)$ that enter $\mathbb{C}^{*}$. This is because these root loci do not intersect any point on the real axis and thus cannot participate in the generation of the second kind of zeros of $(p^{'}(s)/p(s))^{'}$.

The proof of Theorem 0.3.

\begin{proof}
When $p(s)$ has at least two distinct real zeros, by Theorem 3.22, the root loci emitted by the real zeros of $p(s)$ enter $\mathbb{C}^{*}$ and cannot re-intersect the real axis. If the root loci emitted by the non-real zeros of $p(s)$ do not intersect with other root loci on the intervals between each pair of adjacent real zeros of $p(s)$, there must exist a real zero of $p^{'}(s)$ in each such interval $[Q_{j+1}, Q_{j}]$. This real zero of $p^{'}(s)$ is produced by the intersection of two root loci of (3.1), the two root loci are emitted by $Q_{j+1}$ and $Q_{j}$, not by the intersection of more than two root loci of (3.1). By Theorem 1.40, such a real zero of $p^{'}(s)$ can only be a simple real zero of $p^{'}(s)$ and cannot be the second kind of zero of $(p^{'}(s)/p(s))^{'}$. $j=1,\cdots, r-1$.

In the infinite interval $[Q_{1},+\infty)$ from the largest real zero of $p(s)$ to positive infinity and in the infinite interval $(-\infty, Q_{r}]$ from the smallest real zero of $p(s)$ to negative infinity, if the root loci emitted by the non-real zeros of $p(s)$ do not intersect with other root loci on these intervals, the two infinite intervals consist of two complete root loci from the poles of (3.1) to the two infinity points. On these complete root loci, there are no zeros of $p^{'}(s)$ and no second kind of zeros of $(p^{'}(s)/p(s))^{'}$.

When $p(s)$ has a real zero, repeating the previous proof yields the result of Theorem 0.3.

When $p(s)$ has no real zeros, Theorem 0.3 holds trivially.
\end{proof}

The non-real zeros $s_k$ and $s_k^{'}$ emit the root loci $U_k$ and $U_k^{'}$. The root loci $U_k$ and $U_k^{'}$ intersect at a point $c_{i,k}$ in $w_{i,k}$. The root locus $L_{i,k}$ intersects with $U_k$ or $U_k^{'}$ at a point $b_{i,k}$. Both $L_{i,k}$ and $U_k$ (or $U_k^{'}$) leave the real axis at the point $b_{i,k}$. The portions of $L_{i,k}$ and $U_k$ (or $U_k^{'}$) that enter $\mathbb{C}^{*}$ are symmetric about the real axis. Lemma 3.20 proves that  the root locus $L_{i,k}$ cannot intersect at any point in $w_{i,k}$. Therefore, $U_k$ or $U_k^{'}$ cannot intersect at any point in $w_{i,k}$. That is Lemma 3.20  also applies to the root loci $U_k$ and $U_k^{'}$. No additional proof is needed, and Lemma 3.23 can be stated directly.

\begin{lemma}
Let $p(s)$ be a real polynomial with at least one real zero.
Assume that the root loci $U_k$ or $U_k^{'}$ leave the real axis from the point $b_{i,k}$($b_{i,k}\neq Q_k$) in the interval $w_{i,k}$ and enter $\mathbb{C}^{*}$. Then, the root loci $U_k$ and $U_k^{'}$ cannot intersect at any point in the interval $w_{i,k}$.
\end{lemma}

Combining Lemma 3.4 and Lemma 3.23, we obtain Theorem 3.24.

\begin{theorem}
For a real polynomial $p(s)$ with at least one real zero, each real zero $Q_k$ of $p(s)$ emits a root locus $LT_k$ of $2q\pi$ or $2q\pi+\pi$ degree. The root loci $U_k$ and $U_k^{'}$ intersect at a point $c_{i,k}$ in $w_{i,k}$. Assume that the root loci $LT_k$ and $U_k$ (or $U_k^{'}$) intersect at the point $b_{i,k}$ in the interval $w_{i,k}$, they leave the real axis at  $b_{i,k}$($b_{i,k}\neq Q_k$) in  $w_{i,k}$ and enter $\mathbb{C}^{*}$. Then, the root loci $U_k$ and $U_k^{'}$ cannot intersect at any point on the real axis except $c_{i,k}$.
\end{theorem}

By Theorem 3.24, any intersection of $U_k$ or $U_k^{'}$ with the real axis must occur at $c_{i,k}$, as other intersections would violate the uniqueness of breakaway points.

The proof of Theorem 0.4.
\begin{proof}
When $p(s)$ has no real zeros, Theorem 0.4 holds trivially.

Sufficiency:
at least two root loci $f_a$ and $f_a^{'}$ of $2q\pi$ or $2q\pi+\pi$ degree, emitted from a pair of non-real conjugate zeros of $p(s)$, in $\mathbb{C}^{*}$, directly intersect with at least one other root locus at $M$ in the real axis. There exists at least one root locus $f_r$ on the real axis, which has the same unit complex value as $f_a$ and $f_a^{'}$. $f_r$ intersects with $f_a$ and $f_a^{'}$ at $M$.

According to the conditions of this theorem, the root locus $f_r$ may originate from non-real zeros or from real zeros. By  Theorem 1.41 the other root locus $f_r$ is distinct from the two root loci $f_a$ and $f_a^{'}$ emitted from the pair of non-real conjugate zeros. There exist at least three distinct root loci of (3.1) intersecting at $M$. By Theorem 1.40, $M$ is a real multiple zero of $p^{'}(s)$, not a real simple zero of $p^{'}(s)$. Additionally, $M$ is a intersection point of the root loci of (3.1), rather then a zero of $p(s)$, so, $M$ is not a real multiple zero of $p(s)$. Therefore, $M$ is a second kind of zero of $(p^{'}(s)/p(s))^{'}$.

Necessity:
Suppose that $M$ is a second kind of zero of $(p^{'}(s)/p(s))^{'}$. If only two root loci of (3.1) intersect at $M$, then by Theorem 1.40, $M$ must be a real simple zero of $p^{'}(s)$, not a real multiple zero of $p^{'}(s)$. Consequently, $M$ cannot be a second kind of zero of $(p^{'}(s)/p(s))^{'}$. Therefore, if $M$ is indeed a second kind of zero of $(p^{'}(s)/p(s))^{'}$, by Theorem 1.40, there must be at least three distinct root loci of (3.1) intersecting at $M$.

When at least three distinct root loci intersect at $M$. Two cases need to be proven.

1. Suppose that there is only one root locus passing through the point $M$ on the real axis, meaning that the segments of the real axis on both sides of $M$ belong to a single root locus $L_{r}$. $L_{r}$ passes through $M$ and is distributed on the real axis. $L_{r}$ may originate from non-real zeros or from real zeros.

If at least three distinct root loci of (3.1) intersect at $M$ on the real axis, by Theorem 3.22 and Theorem 3.24, then at least one of these root loci, denoted $L_{a_1}$, must originate from a non-real zero $a_1$. $L_{a_1}$ intersects $M$ from $\mathbb{C}^{*}$. Except for point $M$, the root locus $L_{a_1}$ emitted from non-real zeros cannot intersect the real axis. $L_{a_1}$ must directly intersect at point $M$ in $\mathbb{C}^{*}$.

By Theorem 1.27, the $2q\pi$ and $2q\pi+\pi$ degree root loci of (3.1) are symmetric with respect to the real axis. Consequently, on the opposite side of the real axis, at the symmetric position, the root locus $L_{a_1}^{'}$, emitted from the conjugate non-real zero $a_1^{'}$ (the conjugate of the non-real zero $a_1$), also intersects at point $M$ in $\mathbb{C}^{*}$. By Theorem 1.39 and Theorem 1.41, the root locus $L_{r}$ on the real axis is the third root locus, and it is a distinct root locus with $L_{a_1}$ and $L_{a_1}^{'}$. Therefore, when $M$ is a second kind of zero of $(p^{'}(s)/p(s))^{'}$, at least two root loci of $2q\pi$ or $2q\pi+\pi$ degree, originating  from a pair of non-real conjugate zeros of the real polynomial $p(s)$, in $\mathbb{C}^{*}$, directly intersect with at least one other root locus at $M$ in the real axis.

2. Suppose that two root loci on the real axis intersect at point $M$, meaning that the segments of the real axis on both sides of $M$ belong to two distinct root loci $L_{r1}$ and $L_{r2}$. $L_{r1}$ and $L_{r2}$ are distributed on the real axis. $L_{r1}$ and $L_{r2}$ intersect at $M$. $L_{r1}$ and $L_{r2}$ may originate from non-real zeros or from real zeros.

If at least three distinct root loci of (3.1) intersect at $M$, then at least one of these root loci, denoted $L_{a_1}$. By Theorem 3.22 and Theorem 3.24, $L_{a_1}$ must originate from a non-real zero $a_1$. Except for point $M$, the root locus $L_{a_1}$ emitted from non-real zeros cannot intersect the real axis. $L_{a_1}$ must directly intersect at point $M$ in $\mathbb{C}^{*}$.

Consequently, on the opposite side of the real axis, at the symmetric position, the root locus $L_{a_1}^{'}$, emitted from the conjugate non-real zero $a_1^{'}$ (the conjugate of the non-real zero $a_1$), also  directly intersects at point $M$ in $\mathbb{C}^{*}$. By Theorem 1.39 and Theorem 1.41,  two root loci $L_{r1}$ and $L_{r2}$ on the real axis are the third and fourth root loci, and they are the distinct with  root loci to $L_{a_1}$ and $L_{a_1}^{'}$. Therefore, when $M$ is a second kind of zero of $(p^{'}(s)/p(s))^{'}$, at least two root loci of $2q\pi$ or $2q\pi+\pi$ degree, originating from a pair of non-real conjugate zeros of the real polynomial $p(s)$, in $\mathbb{C}^{*}$, directly intersect with at least one other root locus at $M$ in the real axis.
\end{proof}

\begin{example}
The real polynomial $p_{e1}(s)$ is given by $p_{e1}(s)=(s+1)(s^2-s+1)=s^{3}+1$.
\end{example}
The polynomial $p_{e1}(s)$ has a real zero at $s_{e1,3}=-1$ and a pair of complex conjugate non-real zeros at $s_{e1,1,2}=\frac{1}{2}\pm\frac{\sqrt{3}}{2}i$.

Consider the root locus equation for the rational function $\frac{1}{p_{e1}(s)}$: $K\frac{1}{p_{e1}(s)}=\pm 1$.
This equation has poles at $s_{e1,1,2} = \frac{1}{2} \pm \frac{\sqrt{3}}{2}i$ (complex conjugate non-real) and $s_{e1,3} = -1$ (real), with zeros exclusively at infinity.

The interval $(-\infty, s_{e1,3})=(-\infty, -1)$ constitutes a $2q\pi+\pi$ degree root locus segment. Two $2q\pi+\pi$ degree root loci emitted from $s_{e1,1,2}$ and extend to the points at infinity respectively. These three loci do not intersect, and no breakaway points exist.

The interval $(s_{e1,3}, +\infty)=(-1, +\infty)$ constitutes a $2q\pi$ degree root locus segment. Three $2q\pi$ degree root loci emitted from $s_{e1,1,2}$ and $s_{e1,3}$, intersecting at the breakaway point $b_{e1,1} = 0$, which is a second kind of zero.

Figure 1 illustrates the $2q\pi+\pi$ and $2q\pi$ degree root loci in this example, computed using MATLAB.

\begin{figure}[H]
\centering 
\subfigure[$2q\pi+\pi$ degree root locus]{
\label{Fig.sub.1}
\includegraphics[width=0.45\textwidth]{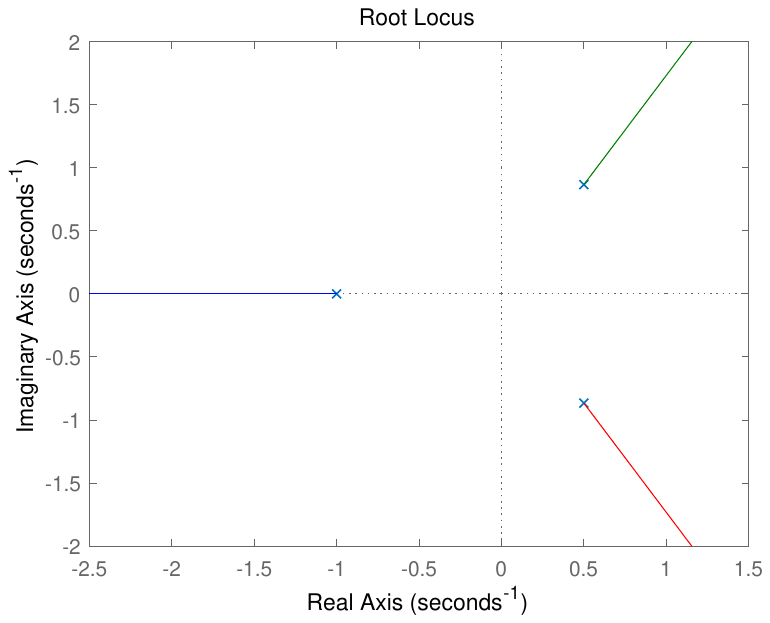}}
\subfigure[$2q\pi$ degree root locus]{
\label{Fig.sub.2}
\includegraphics[width=0.45\textwidth]{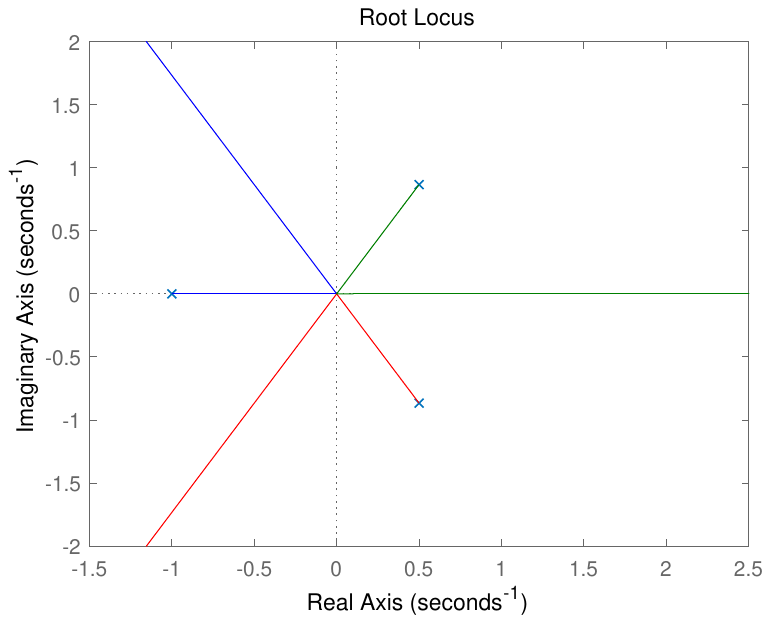}}
\caption{Example 3.25}
\label{Fig.main}
\end{figure}

\begin{example}
The real polynomial $p_{e2}(s)$ is given by $p_{e2}(s)=(s+1)(s-1)(s-i)(s+i)=s^{4}-1$.
\end{example}

The real polynomial $p_{e2}(s)$ has two conjugate non-real zeros $s_{e2,1,2}= \pm i$, and two real zeros $s_{e2,3}=1$, $s_{e2,4}=-1$.

Consider the root locus equation for the rational function $\frac{1}{p_{e2}(s)}$: $K\frac{1}{p_{e2}(s)}=\pm 1$. This root locus equation has two conjugate non-real poles $s_{e2,1,2}= \pm i$, and two real poles $s_{e2,3}=1$, $s_{e2,4}=-1$, with only zeros at infinity.

Four $2q\pi+\pi$ degree root loci emitted from $s_{e2,1,2}$, $s_{e2,3}$, and $s_{e2,4}$. The loci from $s_{e2,1,2}$ extend toward each other, while those from $s_{e2,3}$ and $s_{e2,4}$ also extend toward each other. All four root loci intersect at the breakaway point $b_{e2,1}=0$.

The $2q\pi$ degree root locus comprises four intervals: $(-\infty, s_{e2,4}) = (-\infty, -1)$, $(s_{e2,3}, +\infty) = (1, +\infty)$, $(-i\infty, s_{e2,2}) = (-i\infty, -i)$, $(s_{e2,1}, +i\infty) = (i, +i\infty)$. Four $2q\pi$ degree root loci emitted from $s_{e2,1,2}$, $s_{e2,3}$, and $s_{e2,4}$ and extend to infinity along these paths without intersection. No breakaway points exist.

MATLAB computations confirm $p_{e2}(s)$ produces only one real Fourier breakaway critical point (at $b_{e2,1}=0$). Figure 2 illustrates the $2q\pi+\pi$ and $2q\pi$ degree root loci in this example.

\begin{figure}[H]
\centering 
\subfigure[$2q\pi+\pi$ degree root locus]{
\label{Fig.sub.1}
\includegraphics[width=0.45\textwidth]{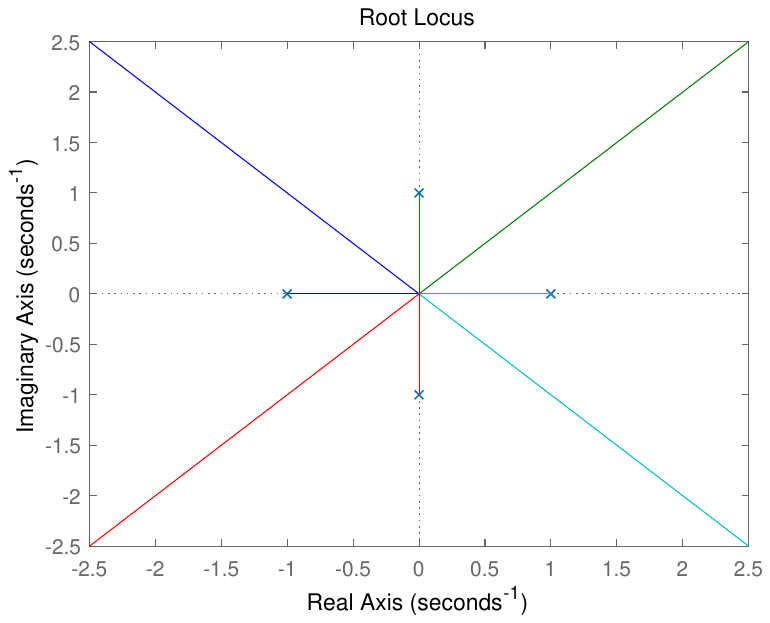}}
\subfigure[$2q\pi$ degree root locus]{
\label{Fig.sub.2}
\includegraphics[width=0.45\textwidth]{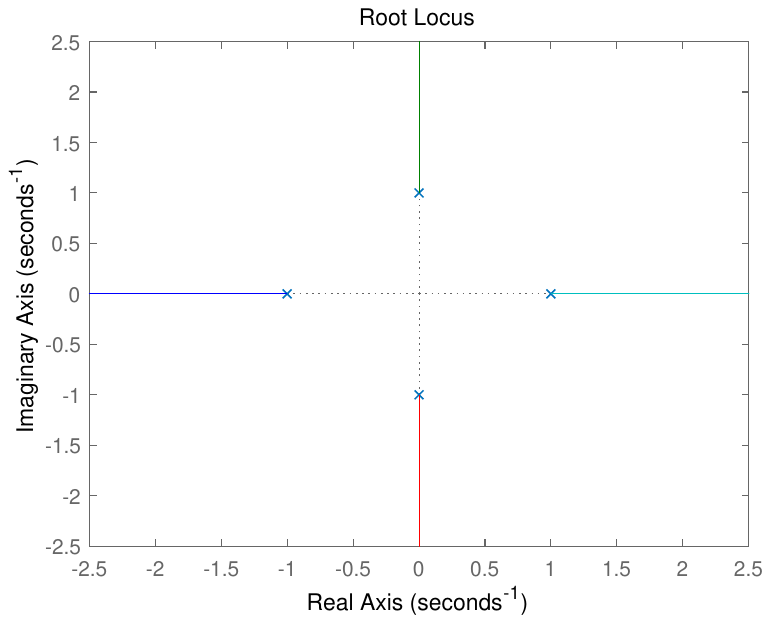}}
\caption{Example 3.26}
\label{Fig.main}
\end{figure}

Conflict of interest: On behalf of all the authors, the corresponding author states that there are no conflicts of
interest.

Availability of data and material: No data were used to support this study.

\bibliographystyle{unsrt}

\end{document}